\documentclass[final,12pt,a4paper]{amsart} 

\usepackage{graphicx}
\usepackage[all]{xy}
\usepackage{placeins}
\usepackage{enumitem}
\usepackage{amssymb}
\usepackage{latexsym}
\usepackage{amsmath}
\usepackage{mathrsfs}
\usepackage{array,booktabs}
\usepackage{verbatim}
\usepackage{fullpage}
\usepackage[notref, notcite]{showkeys}  % for draft versions only
\usepackage{color}
\usepackage[normalem]{ulem}

\newcommand{\K}[1]{K^b(\text{proj-}#1)}
\newcommand{\Db}[1]{D^b(\text{mod-}#1)}
\newcommand{\D}[1]{D(\text{Mod-}#1)}

\newcommand{\add}{{\rm add}}

\renewcommand{\D}{\mathcal{D}}

\newtheorem{corollary}{Corollary}[section]
\newtheorem{lemma}[corollary]{Lemma}
\newtheorem{remark}[corollary]{Remark}
\newtheorem{example}[corollary]{Example}

\newtheorem{definition}[corollary]{Definition}
\newtheorem{theorem}[corollary]{Theorem}
\newtheorem{proposition}[corollary]{Proposition}

\newtheorem{setup}[corollary]{Setup}

\newtheorem*{theorem*}{Theorem}
\newtheorem*{question*}{Question}
\newtheorem*{proposition*}{Proposition}
\newtheorem*{corollary*}{Corollary}
\newtheorem{term}[corollary]{Notation and Terminology}

\begin{document}

\title{Lifting of recollements and gluing of partial silting sets}

\author{Manuel Saor\'in}
\thanks{Manuel Saor\'in was supported   by research projects from the  Spanish Ministerio de Econom\'ia y Competitividad
(MTM2016-77445-P)  and from the Fundaci\'on 'S\'eneca' of Murcia (19880/GERM/15), with a part of FEDER funds.}
\address{Departamento de Matem\'aticas, Universidad de Murcia, Aptdo. 4021, 30100 Espinardo, Murcia, Spain.}
\email{msaorinc@um.es}

\author{Alexandra Zvonareva}
\thanks{Alexandra Zvonareva was first supported by the RFBR Grant 16-31-60089 and later by a fellowship from Alexander von Humboldt Foundation.}
\address{Universit\"at Stuttgart,
Institut f\"ur Algebra und Zahlentheorie,
Pfaffenwaldring 57,
D-70569 Stuttgart,
Germany.}
\email{alexandra.zvonareva@mathematik.uni-stuttgart.de}

\maketitle

\begin{abstract}
This paper focuses on recollements and silting theory in triangulated categories. It consists of two main parts. In the first part a criterion for a recollement of triangulated subcategories  to lift to a torsion torsion-free triple (TTF triple) of ambient triangulated categories with coproducts is proved. As  a  consequence, lifting of TTF triples is possible for   recollements  of  stable  categories  of  repetitive  algebras  or  self-injective  finite  length algebras and recollements of bounded derived categories of separated Noetherian schemes.
When, in addition,  the outer subcategories in the recollement are derived categories of 
small linear categories the conditions from the criterion are sufficient to lift the recollement to a recollement of ambient triangulated categories up to equivalence. In the second part we use these results to study the problem of constructing
silting sets in the central category of a recollement generating the t-structure glued from the silting t-structures in the outer categories. In the case of a recollement of bounded derived categories of Artin algebras we provide an explicit construction for gluing classical silting objects.
\end{abstract}
\smallskip

\begin{comment}
\begin{abstract}
This paper focuses on recollements and silting theory in triangulated categories. It consists of two main parts. In the first part we study the lifting problem for recollements of triangulated subcategories of triangulated categories with coproducts 
and for the associated torsion torsion-free triples (TTF triples, for short). We prove, under relatively mild assumptions, that, when these latter categories are compactly generated and the subcategories 
in the recollement contain compact objects the preservation of compactness by the four upper
 functors in the 
recollement is sufficient to lift the TTF triple. When, in addition,  the outer subcategories in the recollement are derived categories of 
small linear categories the condition is sufficient to lift the recollement up to equivalence. In the second part we use these results to study the problem of constructing
partial silting sets in the central category of a recollement generating the t-structure glued from the partial silting t-structures in the outer categories of the recollement. In the case of a recollement of bounded derived categories of Artin algebras or, a bit more generally, finite length algebras we provide an explicit construction for gluing classical silting objects.
\end{abstract}
\end{comment}


%16G10: Representations of Artinian rings
%18E30: Derived categories, triangulated categories

\begin{comment}
\begin{acknowledgements}
Manuel Saor\'in was supported   by research projects from the  Spanish Ministerio de Econom\'ia y Competitividad
(MTM2016-77445-P)  and from the Fundaci\'on 'S\'eneca' of Murcia (19880/GERM/15), with a part of FEDER funds.
Alexandra Zvonareva was first supported by the RFBR Grant 16-31-60089 and later by a fellowship from Alexander von Humboldt Foundation.
\end{acknowledgements}
\end{comment}

\section{Introduction}

Recollements of triangulated categories were introduced by Beilinson, Bernstein and Deligne \cite{BBD} as a tool to get information about
the derived
category of sheaves over a topological space $X$ from the corresponding derived categories for an open subset $U\subseteq X$ 
and its complement $F=X\setminus U$. In the general abstract picture, when there exists a recollement $(\mathcal{Y}\equiv\mathcal{D}\equiv\mathcal{X})$ of triangulated categories 
\begin{equation}\label{Reca}
\begin{xymatrix}{\mathcal{Y} \ar[r]^{i_*}& \mathcal{D} \ar@<3ex>[l]_{i^!}\ar@<-3ex>[l]_{i^*}\ar[r]^{j^*} & \mathcal{X} 
\ar@<3ex>_{j_*}[l]\ar@<-3ex>_{j_!}[l]},
\end{xymatrix}
\end{equation}
 the properties of $\mathcal{D}$, $\mathcal{X}$ and $\mathcal{Y}$ are closely related  (for unexplained terminology see Section \ref{sect.preliminaries}) and the sequence $(\mathcal{Y}\equiv\mathcal{D}\equiv\mathcal{X})$ can be thought of as a short exact sequence of triangulated categories.  
In representation theory of finite dimensional algebras, a lot of homological invariants can be glued with respect to a recollement. For instance, given a recollement of (bounded) derived categories of finite dimensional algebras $A,$ $B$ and $C$ over a field, Hochschild and cyclic homology and cohomology of $A,$ $B$ and $C$, as well as, $K$-theory, are related by long exact sequences \cite{K98, KN09, H14, TT90, Y92, NR04, Sch06}. Furthermore, global dimension, Cartan determinants and finitistic dimensions of $A,$ $B$ and $C$ are related \cite{KoePerp, AKLY, QH16, H2, CX17}.

Another topic closely related to the structure of triangulated categories is silting theory. Silting theory is a new and dynamically developing topic in representation theory, it studies a special type of generators of triangulated categories, which have very nice properties but are sufficiently widespread. Silting and more generally partial silting objects in  the bounded homotopy category
$\mathcal{K}^b(\text{proj-}A)$ of finitely generated projective modules over a finite dimensional algebra $A$
were introduced in \cite{KV} and further developed in \cite{AI} as a completion of tilting objects under mutation. These objects parametrize bounded t-structures in
$\mathcal{D}^b(\text{mod-}A)$ whose heart is a length category and bounded co-t-structures in $\mathcal{K}^b(\text{proj-}A)$ 
(see \cite{KY} and \cite{KN2}), they also correspond to derived equivalences from $A$ to a non-positive dg algebra. The notion was extended to the unbounded setting 
first in \cite{Wei} and \cite{AMV}, for unbounded derived categories of algebras, and later in \cite{PV} and \cite{NSZ} for arbitrary triangulated categories
with coproducts.  Concepts like complete exceptional collections or sequences in Algebraic Geometry can be interpreted as  silting sets in the derived category of a scheme (see \cite[Section 8]{NSZ}). The main feature of these generalized silting sets is that they naturally define a 
t-structure in the ambient triangulated category whose heart is, in many situations, a module category (see \cite[Section 4]{NSZ}). One of the results of the development of silting theory in the unbounded setting is  the introduction of silting modules in \cite{AMV}, which have turned
out to be very useful to classify homological ring epimorphisms and universal localisations  (see \cite{AMV2} and \cite{MS}). 

Traditionally, for the study of homological properties of a variety $\mathbb{X}$  over a field (or 
a Noetherian scheme of finite type) or a finite dimensional algebra $A$ the bounded derived category of coherent sheaves 
$\mathcal{D}^b(\mathbb{X}):=\mathcal{D}^b(\text{coh}(\mathbb{X}))$ or the 
bounded derived category $\mathcal{D}^b(\text{mod-}A)$ of finitely generated $A$-modules are considered. Therefore, a lot of studies are concentrated on recollements of bounded derived categories and silting objects in the homotopy category of finitely generated projectives. However, the study of 
the structure of the unbounded derived categories is sometimes easier, since a rich arsenal of techniques of compactly generated triangulated categories with products and coproducts becomes available. The generation of recollements and t-structures also becomes more accessible on the unbounded level. Hence it is natural to look for conditions under which recollements of triangulated 
categories at the 'bounded' level lift to recollements at the 'unbounded' level. This is the first goal of this paper. The motivation for 
this stems from \cite{AKLY}, where the authors show that recollements of bounded derived categories of finite dimensional algebras lift up to equivalence to 
corresponding recollements of the unbounded derived categories, which, in addition, can be extended upward and downward to ladders of 
recollements of height three. We wanted to know to what extent this relies on the context of finite dimensional algebras and to what extent this is a general phenomenon occurring in triangulated categories. As it will become apparent through the paper, our  results  are applicable to  other areas of Mathematics, especially to  Algebraic Geometry. In our general framework it turns out that  lifting TTF triples can be guaranteed under  quite general assumptions, whereas lifting of recollements  is more subtle (see Section \ref{section.some points}).

Since t-structures can be glued via a recollement (see \cite{BBD}) another natural question is the following:  given a 
recollement  (\ref{Reca}) of triangulated categories  and silting sets $\mathcal{T}_X$ and $\mathcal{T}_Y$ in $\mathcal{X}$ and 
$\mathcal{Y}$, is it possible to construct  a silting set $\mathcal{T}$ in $\mathcal{D}$ corresponding to the glued t-structure?  This problem was studied in the context of tilting objects in \cite{AKL2} under some restrictions and 
in the context of gluing with respect to co-t-structures   in \cite{LVY}. The second goal of this paper is to show that the process of 
gluing t-structures allows  to construct partial silting sets  in the central category of 
a recollement out of partial silting sets in its outer categories.

The paper is organized as follows. In Section \ref{sect.preliminaries} we introduce most of  the concepts and terminology used throughout 
the paper. In 
Section \ref{section.some points} we study lifting of recollements and the associated TTF triples. In particular we prove a criterion for a recollement $(\mathcal{Y}\equiv\mathcal{D}\equiv\mathcal{X})$ of thick subcategories of 
compactly generated triangulated categories  $\hat{\mathcal{Y}}$, $\hat{\mathcal{D}}$ and $\hat{\mathcal{X}}$ to lift to a  TTF triple in $\hat{\mathcal{D}}$ under the assumption that the subcategories $\mathcal{Y},\mathcal{D},\mathcal{X}$  contain 
the respective subcategories of compact objects (see Theorem \ref{thm.lifting of recollements}).
As a consequence of this theorem lifting of TTF triples is possible for several types of recollements, such as recollements of stable categories of repetitive algebras or self-injective finite length algebras or recollements of bounded derived categories of separated Noetherian schemes (see Example \ref{ex repet}). However, lifting of the recollement is more delicate and the answer to the following question seems to be unknown.
\begin{question*} Does the lifting of the TTF  triple corresponding to the recollement $({\mathcal{Y}}\equiv {\mathcal{D}}\equiv{\mathcal{X}})$ imply the lifting of this recollement to a recollement $(\hat{\mathcal{Y}}\equiv \hat{\mathcal{D}}\equiv\hat{\mathcal{X}})$ at least up to equivalence?
\end{question*}
It seems to be related to the problem of constructing a triangulated equivalence $\hat{\mathcal{D}}\stackrel{\cong}{\longrightarrow}\hat{\mathcal{E}}$, having a triangulated equivalence
$\hat{\mathcal{D}}^c\stackrel{\cong}{\longrightarrow}\hat{\mathcal{E}}^c$, for compactly generated triangulated categories 
$\hat{\mathcal{D}}$ and $\hat{\mathcal{E}}$. Due to results of Rickard (see \cite{R1} and \cite{R2}) and Keller \cite{K},  such a construction is 
possible (even though it may not be the lift of the equivalence $\hat{\mathcal{D}}^c\stackrel{\cong}{\longrightarrow}\hat{\mathcal{E}}^c$) when 
one of the categories $\hat{\mathcal{D}}$ or $\hat{\mathcal{E}}$ is the derived category of an algebra or, more generally, a small $K$-category. 
This is the  
reason for the following consequence of Theorem \ref{thm.Case Y homologically loc.bounded}.

\begin{theorem*} \label{thm.Case Y homologically loc.bounded2}
Let $\mathcal{B}$ and $\mathcal{C}$ be small $K$-linear categories, viewed as dg categories concentrated in zero degree,  let 
$\mathcal{A}$ be a dg category, and suppose that there is a recollement $(\mathcal{D}^{\star}_\dagger (\mathcal{B})\equiv \mathcal{D}^{\star}_\dagger (\mathcal{A})\equiv \mathcal{D}^{\star}_\dagger (\mathcal{C}))$, where $\star \in\{ \emptyset,+,-,b\}$ and $\dagger\in\{ \emptyset,fl\}$ (here $fl$ means 'finite length') and all the subcategories contain the respective 
subcategories of compact objects. 

If the functors $j_!,j^*,i^*,i_*$ preserve compact objects, the given recollement lifts up to equivalence to a recollement $(\mathcal{D}(\mathcal{B})\equiv \mathcal{D}(\mathcal{A})\equiv \mathcal{D}(\mathcal{C}))$, which is the upper part of a ladder of recollements of height two. 
\end{theorem*}

In certain circumstances, one can guarantee that the functors preserve compact objects. For example, this holds in the situation of a recollement $\mathcal{Y}\equiv\mathcal{D}\equiv\mathcal{X}$ of homologically non-positive homologically finite length dg algebras $A$, $B$ and $C$. This allows us to generalise some of the  results from \cite{AKLY} (see Proposition \ref{prop.homologically non-positive locally fd}).

In Section \ref{sect.partial silting sets} we define partial silting sets and objects in arbitrary triangulated categories, give 
some examples,   and study when partial silting sets are uniquely determined by the associated t-structure. In Section \ref{SecPreenv} we revise the connection between the construction of (pre)envelopes, t-structures and co-t-structures. It is seen, in particular, that the bijection between silting objects and bounded co-t-structures  of \cite{AI, MSAuBu, KY} extends to any small triangulated category with split idempotents, more generally, one can consider one-sided bounded co-t-structures in the bijection (see Proposition \ref{prop.bijection psilting-cotstructures}). The study of envelopes is later used 
 in the last section to give an explicit construction of a classical silting object glued with respect to a recollement of bounded derived categories of finite length algebras.

Section \ref{sect.gluing partial silting} is devoted to the construction of partial silting sets in arbitrary triangulated categories by gluing 
t-structures via recollements. Our results on gluing partial silting sets are based on a technical criterion (Theorem \ref{thm.gluing of silting t-structures}), which allows to glue partial silting sets $\mathcal{T}_X$
and $\mathcal{T}_Y$ in triangulated categories $\mathcal{X}$ and $\mathcal{Y}$ with respect to a recollement $(\mathcal{Y}\equiv \mathcal{D}\equiv \mathcal{X})$.
Conditions of Theorem \ref{thm.gluing of silting t-structures} are easier to check when $\mathcal{T}_X$ and $\mathcal{T}_Y$ consist of compact objects and some 
of the functors in the recollement preserve compact objects. We refer the reader to a more general Theorem
\ref{thm.gluing at compact level}, which has the following  consequence. Note that  when 
$A$, $B$ and $C$ are finite length algebras the corollary can be applied, replacing  $\mathcal{D}^b_{fl}(-)$ by the equivalent category $\mathcal{D}^b(\text{mod}-)$.

\begin{corollary*}[see Corollary \ref{cor.Liu-Vitoria-Yang-generalized}, \ref{cor.gluing Artin case}] \label{cor.last in introduction}
Let $A$, $B$ and $C$ be homologically finite length dg algebras, the first of which is homologically non-positive. Let $(\mathcal{D}^b_{fl} (B)\equiv \mathcal{D}^b_{fl} (A) \equiv \mathcal{D}^b_{fl} (C))$ be a recollement, let $T_C$ and $T_B$ be silting objects in $\mathcal{D}^c(C)$ and $\mathcal{D}^c(B)$, respectively, with the associated t-structures $(\mathcal{X}^{\leq 0},\mathcal{X}^{\geq 0})$ and $(\mathcal{Y}^{\leq 0},\mathcal{Y}^{\geq 0})$. There exists a triangle 
$\tilde{T}_B\longrightarrow i_*(T_B)\longrightarrow U_{T_B}[1]\stackrel{+}{\longrightarrow}$ in $\mathcal{D}^c(A)$ such that 
$U_{T_B}\in j_!(\mathcal{X}^{\leq 0})$ and $\tilde{T}_B\in {}^\perp j_!(\mathcal{X}^{\leq 0})[1]$. In particular 
$T=j_!(T_C)\oplus\tilde{T}_B$ is a silting object in $\mathcal{D}^c(A)$, uniquely determined up to add-equivalence, which generates the 
glued t-structure $(\mathcal{D}^{\leq 0},\mathcal{D}^{\geq 0})$ in $\mathcal{D}^b_{fl}(A)$. 
\end{corollary*}

We finish the paper, comparing our results on gluing silting objects in the particular context of finite 
dimensional 
algebras over a field with the results of \cite{LVY}. As mentioned before, our methods provide an explicit inductive construction of the glued silting object in this case, we illustrate how to apply this explicit inductive construction with an example.

\section{Preliminaries} \label{sect.preliminaries}

All categories considered in this paper are $K$-categories over some commutative ring $K$ and all functors are assumed to be $K$-linear. 
Unless explicitly said otherwise,  the  categories which appear will be either triangulated
$K$-categories with split idempotents or their subcategories, and all of them are assumed to have Hom-sets.  All subcategories will  be full and closed 
under isomorphisms. Coproducts and products are always  small (i.e. set-indexed). The expression '$\mathcal{D}$ has coproducts (resp. products)' will mean that $\mathcal{D}$ has arbitrary set-indexed coproducts (resp. products).   
When $\mathcal{S}\subset\text{Ob}(\mathcal{D})$ is a class of objects, we shall
denote by $\text{add}_\mathcal{A}(\mathcal{S})$ (resp.
$\text{Add}_\mathcal{A}(\mathcal{S})$) the subcategory of $\mathcal{D}$ consisting of the objects which are direct summands of finite (resp. arbitrary) coproducts
of objects in $\mathcal{S}$.

Let $\mathcal{D}$ be a triangulated category, we will
denote by $[1]:\mathcal{D}\longrightarrow\mathcal{D}$ the
\emph{suspension functor}, $[k]$
will denote the $k$-th power of $[1]$, for each integer $k$. \emph{(Distinguished) triangles} in $\mathcal{D}$ will be denoted by
$X\longrightarrow Y\longrightarrow Z\stackrel{+}{\longrightarrow}$. A \emph{triangulated functor} $F: \mathcal{D}_1\longrightarrow \mathcal{D}_2$ between
triangulated categories is an additive functor together with a natural  isomorphism $F\circ [1]\simeq [1]\circ F$, which sends triangles to triangles. For more details on triangulated categories see \cite{N}.

Let $\mathcal{D}$ be a triangulated category and let $\mathcal{S}$ be a class of objects in $\mathcal{D}$. We are going to use the 
following subcategories of $\mathcal{D}$:
$$\mathcal{S}^\perp =\{X \in \mathcal{D} \mid \text{Hom}_\mathcal{D}(S,X)=0 \text{ for any }S\in\mathcal{S}\}$$
$$^\perp\mathcal{S} =\{X \in \mathcal{D} \mid \text{Hom}_\mathcal{D}(X,S)=0 \text{ for any }S\in\mathcal{S}\}$$
for an integer $n$ and $*$ standing for $\leq n$, $\geq n$, $>n$, $<n$ or $k\in\mathbb{Z}$
$$\mathcal{S}^{\perp *} =\{X \in \mathcal{D} \mid \text{Hom}_\mathcal{D}(S,X[k])=0 \text{ for any }S\in\mathcal{S} \text{ and } 
k\in \mathbb{Z} \text{ satisfying }*\}$$
$$^{\perp *}\mathcal{S} =\{X \in \mathcal{D} \mid \text{Hom}_\mathcal{D}(X,S[k])=0 \text{ for any }S\in\mathcal{S} \text{ and } 
k\in \mathbb{Z} \text{ satisfying }*\}.$$

Given two subcategories $\mathcal{X}$ and $\mathcal{Y}$ of a triangulated category $\mathcal{D}$, we will denote by $\mathcal{X}\star\mathcal{Y}$ the subcategory of $\mathcal{D}$ consisting of the objects $M$ which fit into a triangle $X\longrightarrow M\longrightarrow Y\stackrel{+}{\longrightarrow}$, where $X\in\mathcal{X}$ and $Y\in\mathcal{Y}$. Due to the octahedral axiom, the operation $\star$ is associative, so for a family of subcategories $(\mathcal{X}_i)_{1\leq i\leq n}$ the subcategory  $\mathcal{X}_1\star\mathcal{X}_2\star \cdots\star\mathcal{X}_n$ is well-defined (see \cite{BBD}). A subcategory $\mathcal{X}$ is \emph{closed under extensions} when $\mathcal{X}\star\mathcal{X}\subseteq\mathcal{X}$. 

Given a triangulated category $\mathcal{D}$, a subcategory $\mathcal{E}$ will be called a \emph{suspended} (resp. \emph{strongly suspended}) 
\emph{subcategory} 
 if $\mathcal{E}[1]\subseteq\mathcal{E}$ and $\mathcal{E}$ is closed under extensions (resp. extensions and direct summands).  If $\mathcal{E}$ is strongly suspended and $\mathcal{E}=\mathcal{E}[1]$, we will say that $\mathcal{E}$ is a \emph{thick subcategory}. When  $\mathcal{D}$ has coproducts, a triangulated subcategory closed under taking arbitrary coproducts is called a \emph{localizing subcategory}. Note that such a subcategory is always thick  by \cite[Proposition 1.6.8]{N}. Clearly, there are 
dual concepts of a \emph{(strongly) cosuspended subcategory} and a \emph{colocalizing subcategory}, while that of a thick subcategory is self-dual. 
Given a class $\mathcal{S}$ of objects of $\mathcal{D}$, we will denote by $\text{susp}_\mathcal{D}(\mathcal{S})$ (resp. $\text{thick}_\mathcal{D}(\mathcal{S})$) the smallest strongly suspended (resp. thick) subcategory of $\mathcal{D}$ containing $\mathcal{S}$. When $\mathcal{D}$ has coproducts, we will let $\text{Susp}_\mathcal{D}(\mathcal{S})$ and $\text{Loc}_\mathcal{D}(\mathcal{S})$ be the smallest (strongly) suspended subcategory closed under taking coproducts and the smallest localizing subcategory containing $\mathcal{S}$, respectively.

\subsection{Torsion pairs, t-structures and co-t-structures:} A pair of subcategories $(\mathcal{X},\mathcal{Y})$ in $\mathcal{D}$ is a \emph{torsion pair} if

$\bullet$ $\mathcal{X}$ and $\mathcal{Y}$ are closed under direct summands;

$\bullet$ $\text{Hom}_\mathcal{D}(X,Y)=0$, for all $X\in\mathcal{X}$ and $Y\in\mathcal{Y}$;

$\bullet$ $\mathcal{D}=\mathcal{X}\star\mathcal{Y}$.

A \emph{t-structure} in $\mathcal{D}$ is a pair $(\mathcal{D}^{\leq 0},\mathcal{D}^{\geq 0})$ such that  
$(\mathcal{D}^{\leq 0},\mathcal{D}^{\geq 0}[-1])$ is a torsion pair and $\mathcal{D}^{\leq 0}[1]\subseteq\mathcal{D}^{\leq 0}$. 
A \emph{co-t-structure} or a \emph{weight structure} is a pair $(\mathcal{D}_{\geq 0},\mathcal{D}_{\leq 0})$ such that 
$(\mathcal{D}_{\geq 0}[-1],\mathcal{D}_{\leq 0})$ is a torsion pair and $\mathcal{D}_{\geq 0}[-1]\subseteq\mathcal{D}_{\geq 0}$. 
Adopting the 
terminology used for t-structures, given a torsion pair $(\mathcal{X},\mathcal{Y})$, we will call $\mathcal{X}$ and 
$\mathcal{Y}$ the \emph{aisle} and the \emph{co-aisle} of the torsion pair. Note that the aisle of a torsion pair 
$(\mathcal{X},\mathcal{Y})$ is suspended (resp. cosuspended) if and only if 
$(\mathcal{X},\mathcal{Y}[1])$ (resp. $(\mathcal{X}[1],\mathcal{Y})$) is a t-structure (resp. co-t-structure).

For a t-structure $(\mathcal{D}^{\leq 0},\mathcal{D}^{\geq 0})$, the objects $U$ and $V$ in a triangle $U\longrightarrow M\longrightarrow V\stackrel{+}{\longrightarrow}$, with $U\in\mathcal{D}^{\leq 0}$ and $V\in\mathcal{D}^{>0}:=\mathcal{D}^{\geq 0}[-1]$,  are uniquely determined by  $M\in\mathcal{D}$ up to isomorphism. The assignments $M\rightsquigarrow U$ and $M\rightsquigarrow V$ 
coincide on objects with the action of the  functors $\tau^{\leq 0}:\mathcal{D}\longrightarrow\mathcal{D}^{\leq 0}$ and $\tau^{>0}:\mathcal{D}\longrightarrow\mathcal{D}^{>0}$, which are right and left adjoint to the inclusion functors. The functors $\tau^{\leq 0}$ and $\tau^{> 0}$ are called the \emph{left and right truncation functors} with respect to the t-structure. When $\tau=(\mathcal{D}^{\leq 0},\mathcal{D}^{\geq 0})$ (resp. $(\mathcal{D}_{\geq 0},\mathcal{D}_{\leq 0})$) is a t-structure (resp. a co-t-structure), the intersection $\mathcal{H}:=\mathcal{D}^{\leq 0}\cap\mathcal{D}^{\geq 0}$ (resp. $\mathcal{C}:=\mathcal{D}_{\geq 0}\cap\mathcal{D}_{\leq 0}$) is called the \emph{heart (resp. co-heart) of the t-structure (resp. co-t-structure)}. Recall that $\mathcal{H}$ is an abelian category in which the short exact sequences are induced by the triangles with all the three terms in $\mathcal{H}$  (see \cite{BBD}). Sometimes we shall use the term \emph{co-heart of the t-structure $\tau$}, meaning the intersection ${\mathcal{C}_\tau ={}^{\perp}(\mathcal{D}^{\leq 0})[1]\cap\mathcal{D}^{\leq 0}}$.   A \emph{semi-orthogonal decomposition} of $\mathcal{D}$ is a torsion pair $(\mathcal{X},\mathcal{Y})$ such that $\mathcal{X}=\mathcal{X}[1]$ (or, equivalently,  $\mathcal{Y}=\mathcal{Y}[1]$). Note that such a pair is both a t-structure and a co-t-structure in $\mathcal{D}$, and the corresponding truncation functors are triangulated. The notions of torsion pair, t-structure, co-t-structure and semi-orthogonal decomposition are self-dual. 

If $\mathcal{D}'$ is a thick subcategory of $\mathcal{D}$, we say that a torsion pair $\tau =(\mathcal{X},\mathcal{Y})$ in $\mathcal{D}$ 
\emph{restricts to $\mathcal{D}'$} when $\tau':=(\mathcal{X}\cap\mathcal{D}',\mathcal{Y}\cap\mathcal{D}')$ is a torsion pair in $\mathcal{D}'$. In this case $\tau'$ is called the \emph{restriction} of $\tau$ to $\mathcal{D}'$. Conversely, when $\tau'=(\mathcal{X}',\mathcal{Y}')$ is a torsion pair in $\mathcal{D}'$, we say that it \emph{lifts to $\mathcal{D}$} if there is a torsion pair $\tau=(\mathcal{X},\mathcal{Y})$ in $\mathcal{D}$ which restricts to $\tau'$. Then $\tau$ is called a \emph{lifting of $\tau'$ to $\mathcal{D}$}.

Given two torsion pairs $\tau=(\mathcal{X},\mathcal{Y})$ and $\tau'=(\mathcal{Y}',\mathcal{Z})$ in $\mathcal{D}$, we shall say that 
\emph{$\tau$ is left adjacent to $\tau'$} or that \emph{$\tau'$ is right adjacent to $\tau$} or that $\tau$ and $\tau'$ (in this order) are \emph{adjacent torsion pairs} when $\mathcal{Y}=\mathcal{Y'}$. Note that the torsion pairs associated to the co-t-structure $(\mathcal{D}_{\geq 0},\mathcal{D}_{\leq 0})$ and the t-structure $(\mathcal{D}^{\leq 0},\mathcal{D}^{\geq 0})$ are adjacent if and only if $\mathcal{D}_{\leq 0}=\mathcal{D}^{\leq 0}$. In this case their co-hearts coincide.  A triple of subcategories $(\mathcal{X},\mathcal{Y},\mathcal{Z})$ of $\mathcal{D}$ is called a \emph{torsion torsion-free triple (TTF triple, for short)} when $(\mathcal{X},\mathcal{Y})$ and $(\mathcal{Y},\mathcal{Z})$ are adjacent t-structures, which is equivalent to saying that they are adjacent semi-orthogonal decompositions.We shall say that such a TTF triple is \emph{extendable to the right} when $(\mathcal{Y},\mathcal{Z},\mathcal{Z}^\perp)$ is also a TTF triple. By \cite[Proposition 3.4]{NS1}, we know that this is always the case when $\mathcal{D}$  and the torsion pair $(\mathcal{Y},\mathcal{Z})$ are compactly generated (see Subsection 2.3 below for the definition of compact generation).
 As before, one can consider lifting and restriction of TTF triples (see \cite{NS2} for details).

\subsection{Recollements:} Let $\mathcal{D}$, $\mathcal{X}$ and $\mathcal{Y}$ be triangulated
categories. $\mathcal{D}$ is said to be a \emph{recollement} of
$\mathcal{X}$ and $\mathcal{Y}$ if there are six  triangulated functors
as in the following diagram
\begin{equation}\label{Recb}
\begin{xymatrix}{\mathcal{Y} \ar[r]^{i_*}& \mathcal{D} \ar@<3ex>[l]_{i^!}\ar@<-3ex>[l]_{i^*}\ar[r]^{j^*} & \mathcal{X} \ar@<3ex>_{j_*}[l]\ar@<-3ex>_{j_!}[l]}
\end{xymatrix}
\end{equation}
such that

1) $(i^*,i_*)$, $(i_*,i^!)$, $(j_!,j^*)$, $(j^*,j_*)$ are adjoint
pairs,

2) $i_*$, $j_*$, $j_!$ are full embeddings,

3) $i^!j_*=0$ (and, hence $j^*i_*=0$ and $i^*j_!=0$),

4) for any $Z \in \mathcal{D}$ the units and the counits of the
adjunctions give triangles:
$$i_*i^!Z \longrightarrow Z \longrightarrow j_*j^*Z \stackrel{+}{\longrightarrow},$$
$$j_!j^*Z \longrightarrow Z \longrightarrow i_*i^*Z \stackrel{+}{\longrightarrow}.$$

To any recollement one canonically associates the TTF triple $(\text{Im}(j_!),\text{Im}(i_*),\text{Im}(j_{*}))$ in $\mathcal{D}$. Conversely, if $(\mathcal{X},\mathcal{Y},\mathcal{Z})$ is a TTF triple in $\mathcal{D}$, then one obtains a recollement as above, where $j_!:\mathcal{X}\hookrightarrow\mathcal{D}$ and $i_*:\mathcal{Y}\hookrightarrow\mathcal{D}$ are the inclusion functors. 
Two recollements $(\mathcal{Y}\equiv\mathcal{D}\equiv\mathcal{X})$ and $(\tilde{\mathcal{Y}}\equiv\mathcal{D}\equiv\tilde{\mathcal{X}})$  are said to be \emph{equivalent} when the associated TTF triples coincide. It is easy to see that this is equivalent to the existence of triangulated equivalences $F:\mathcal{X}\stackrel{\cong}{\longrightarrow}\tilde{\mathcal{X}}$ and $G:\mathcal{Y}\stackrel{\cong}{\longrightarrow}\tilde{\mathcal{Y}}$ such that the sextuple of functors associated to the second recollement is pointwise naturally isomorphic to $(G\circ i^*,i_*\circ G^{-1}, G\circ i^!,j_!\circ F^{-1},F\circ j^*, j_*\circ F^{-1})$, for any choice of quasi-inverses $F^{-1}$ and $G^{-1}$ of $F$ and $G$.

Given thick subcategories $\mathcal{Y}'\subseteq\mathcal{Y}$, $\mathcal{D}'\subseteq\mathcal{D}$ and $\mathcal{X}'\subseteq\mathcal{X}$ and a recollement 
\begin{equation}\label{Recc}
\begin{xymatrix}{\mathcal{Y}' \ar[r]^{\tilde{i}_*}& \mathcal{D}' \ar@<3ex>[l]_{\tilde{i}^!}\ar@<-3ex>[l]_{\tilde{i}^*}\ar[r]^{\tilde{j}^*} & \mathcal{X}' \ar@<3ex>_{\tilde{j}_*}[l]\ar@<-3ex>_{\tilde{j}_!}[l]}
\end{xymatrix},
\end{equation}
 we say that the recollement (\ref{Recb}) \emph{restricts} to the recollement (\ref{Recc})
 or that the recollement (\ref{Recc}) \emph{lifts} to the recollement (\ref{Recb}), when the functors in the recollement (\ref{Recc}) are naturally isomorphic to the restrictions of the functors in the recollement (\ref{Recb}). We shall say that the recollement (\ref{Recb}) \emph{restricts, up to equivalence,} to the recollement (\ref{Recc}),  or that the recollement  (\ref{Recc}) \emph{lifts, up to equivalence,} to the recollement  (\ref{Recb}) 
 when the TTF triple  $(\text{Im}(j_!),\text{Im}(i_*),\text{Im}(j_*))$ restricts to $\mathcal{D}'$ and  the restriction coincides with $(\text{Im}(\tilde{j}_!),\text{Im}(\tilde{i}_*),\text{Im}(\tilde{j}_*))$. Obviously 'restricts' implies 'restricts up to equivalence'.
 
\begin{remark}
The typical situation throughout this paper is that of a recollement (\ref{Recc}), where we know that $\mathcal{Y}'$, $\mathcal{D}'$ and $\mathcal{X}'$ are thick subcategories of $\mathcal{Y}$, $\mathcal{D}$ and $\mathcal{X}$, respectively. The condition that the TTF triple $(\text{Im}(\tilde{j}_!),\text{Im}(\tilde{i}_*),\text{Im}(\tilde{j}_*))$ in $\mathcal{D}'$ lifts to a TTF triple in $\mathcal{D}$ does not mean that it lifts to a TTF triple coming from a recollement (\ref{Recb}), i.e.\ it might happen that we cannot find functors $i^*, i_*, i^!,j_!,j^*$ and $j_*$, which restrict to the functors $\tilde{i}^*, \tilde{i}_*, \tilde{i}^!,\tilde{j}_!,\tilde{j}^*$ and $\tilde{j}_*$. Therefore if the recollement (\ref{Recc}) lifts up to equivalence to a recollement (\ref{Recb}), then the TTF triple $(\text{Im}(\tilde{j}_!),\text{Im}(\tilde{i}_*),\text{Im}(\tilde{j}_*))$ lifts to a TTF triple in $\mathcal{D}$, but the converse need not be true. Similarly, it might happen that a TTF triple coming from the recollement (\ref{Recb}) restricts to the subcategory $\mathcal{D}'$ but the functors from (\ref{Recb}) do not restrict to $\mathcal{Y}',\mathcal{D}'$ and $\mathcal{X}'$.
\end{remark}

Given torsion pairs $(\mathcal{X}',\mathcal{X}'')$ and $(\mathcal{Y}',\mathcal{Y}'')$ in $\mathcal{X}$ and $\mathcal{Y}$, respectively, \emph{the torsion pair glued with respect to the recollement (\ref{Recb})} is the pair $(\mathcal{D}',\mathcal{D}'')$ in $\mathcal{D}$, where

$$\mathcal{D}' =\{Z \in \mathcal{D} \vert j^*Z \in \mathcal{X}',
i^*Z \in \mathcal{Y}'\},$$ $$\mathcal{D}'' =\{Z \in \mathcal{D}
\vert j^*Z \in \mathcal{X}'', i^!Z \in \mathcal{Y}''\}.$$
Moreover, when the original torsion pairs are associated to  t-structures (resp. co-t-structures or semi-orthogonal decompositions), the resulting torsion pair is associated to a t-structure (resp. co-t-structure or semi-orthogonal decomposition) in $\mathcal{D}$ (see \cite[Th\'eor\`eme 1.4.10]{BBD} for t-structures and  \cite[Theorem 8.2.3]{Bo} for co-t-structures).

A \emph{ladder of recollements} $\mathcal{L}$ is a finite or infinite diagram of
triangulated categories and triangulated functors

$$
\begin{xymatrix}{
\vdots&\vdots&\vdots\\
\mathcal{C'} \ar[r] \ar@<4ex>[r]\ar@<-4ex>[r]& \mathcal{C} \ar@<2ex>[l]\ar@<-2ex>[l] \ar@<4ex>[r]\ar@<-4ex>[r] \ar[r] & \mathcal{C''} \ar@<2ex>[l]\ar@<-2ex>[l]\\
\vdots&\vdots&\vdots\\
}
\end{xymatrix}
$$
such that any three consecutive rows form a recollement (see \cite{BGSKo, AKLY}). The
\emph{height} of a ladder is the number of recollements contained in
it (counted with multiplicities).

%\textcolor{red}{The following result is included in \cite[Proposition 3.4]{GP} since the proof of the equivalence of assertions i-iii in that proposition does not use the compact generation hypotheses:}

%\begin{proposition} \label{Gao-Psaroud}
%\textcolor{red}{Let \begin{equation}\label{Recb}
%\begin{xymatrix}{\mathcal{Y} \ar[r]^{i_*}& \mathcal{D} \ar@<3ex>[l]_{i^!}\ar@<-3ex>[l]_{i^*}\ar[r]^{j^*} & \mathcal{X} \ar@<3ex>_{j_*}[l]\ar@<-3ex>_{j_!}[l]}
%\end{xymatrix}
%\end{equation} be a recollement of triangulated categories. The following assertions are equivalent:}

%{\begin{enumerate}
%\item \textcolor{red}{It is the upper part of a ladder of recollements of height two.}
%\item \textcolor{red}{The functor $j_*$ has a right adjoint.}
%\item \textcolor{red}{The functor $i^!$ has a right adjoint.}
%\item \textcolor{red}{The associated TTF triple $(\text{Im}(j_!),\text{Im}(i_*),\text{Im}(j_*))$ in $\mathcal{D}$ is extendable to the right.}

%\textcolor{red}{When the categories $\mathcal{Y}$, $\mathcal{D}$ and $\mathcal{X}$ are compactly generated, they are also equivalent to the following assertions:}

%\item \textcolor{red}{The functor $i_*$ preserves compact objects.}
%\item \textcolor{red}{The functor $j^*$ preserves compact objects.} 
% \end{enumerate}}
%\end{proposition}

\subsection{Generation and Milnor colimits:} For a class of objects $\mathcal{S}$ in $\mathcal{D}$ one can consider the pair of subcategories 
$(\mathcal{X},\mathcal{Y})=({}^\perp (\mathcal{S}^\perp),\mathcal{S}^\perp)$. Then, $\mathcal{X}$ and $\mathcal{Y}$ are closed under direct 
summands and $\text{Hom}_\mathcal{D}(X,Y)=0$ for all $X\in\mathcal{X}$ and $Y\in\mathcal{Y}$. However, the inclusion 
$\mathcal{X}\star\mathcal{Y}\subseteq\mathcal{D}$ might be strict, so that $(\mathcal{X},\mathcal{Y})$ is not necessarily a torsion pair. 
We shall say that 
\emph{$\mathcal{S}$ generates a torsion pair in $\mathcal{D}$} or
 that \emph{the torsion pair generated by $\mathcal{S}$ in $\mathcal{D}$ exists} if $({}^\perp (\mathcal{S}^\perp),\mathcal{S}^\perp)$ 
is a torsion pair. Slightly abusing common terminology, we will say that \emph{$\mathcal{S}$ generates 
a t-structure (resp. co-t-structure or semiorthogonal decomposition)} or that \emph{the t-structure (resp. co-t-structure or 
semiorthogonal decomposition) generated by $\mathcal{S}$ in $\mathcal{D}$ exists} when the torsion pair generated by $\bigcup_{k\geq 0}\mathcal{S}[k]$ (resp.  $\bigcup_{k\leq 0}\mathcal{S}[k]$ or $\bigcup_{k\in\mathbb{Z}}\mathcal{S}[k]$) exists. In all those cases $\mathcal{S}$ is contained in the aisle of the corresponding t-structure (resp. co-t-structure, resp. semi-orthogonal decomposition). That is,  if $(\mathcal{X},\mathcal{Y})$ is the torsion pair constructed above, then $(\mathcal{X},\mathcal{Y}[1])=({}^\perp (\mathcal{S}^{\perp_{\leq 0}}),\mathcal{S}^{\perp_{<0}})$ and $(\mathcal{X},\mathcal{Y}[-1])=({}^\perp (\mathcal{S}^{\perp_{\geq  0}}),\mathcal{S}^{\perp_{ >0 }})$ are the t-structure and co-t-structure generated by $\mathcal{S}$. 
The semi-orthogonal decomposition  generated by $\mathcal{S}$ is  $({}^\perp (\mathcal{S}^{\perp_{k\in\mathbb{Z}}}),\mathcal{S}^{\perp_{k\in\mathbb{Z}}})$). The definition of the dual notions is left to the reader. We just point out that, keeping the dual philosophy of forcing $\mathcal{S}$ to be contained in the co-aisle, the \emph{t-structure (resp. co-t-structure) cogenerated by $\mathcal{S}$}, when it exists, is the pair $({}^{\perp_{<0}}\mathcal{S},({}^{\perp_{\leq 0}}\mathcal{S})^\perp)$ (resp. $({}^{\perp_{> 0}}\mathcal{S},({}^{\perp_{\geq  0}}\mathcal{S})^\perp)$).
A class (resp.  set)
$\mathcal{S}\subset\text{Ob}(\mathcal{D})$ is a \emph{generating class (resp. set) of $\mathcal{D}$} if $\mathcal{S}^{\perp_{k\in\mathbb{Z}}}=0$, in this case we will also say, that \emph{$\mathcal{S}$ generates $\mathcal{D}$}.  We say that $\mathcal{D}$ satisfies the \emph{property of infinite d\'evissage with respect to $\mathcal{S}$} when $\mathcal{D}=\text{Loc}_\mathcal{D}(\mathcal{S})$, a fact that implies that $\mathcal{S}$ generates $\mathcal{D}$.
 
When $\mathcal{D}$  has
coproducts,  a \emph{compact object} is an object $X$ such that the canonical map $\coprod_{i\in I}\text{Hom}_\mathcal{D}(X,M_i)\longrightarrow\text{Hom}_\mathcal{D}(X,\coprod_{i\in I}M_i)$ is bijective. 
A torsion pair is called 
\emph{compactly
generated} when  there exists a \underline{set} of compact objects which generates the torsion pair. We say that \emph{$\mathcal{D}$ is a 
compactly generated triangulated category} when it has a generating \underline{set} of compact objects.  It is well-known that in this case
the subcategory $\mathcal{D}^c$ of the compact objects of $\mathcal{D}$ is skeletally small (see, e.g., \cite[Lemma 4.5.13]{N}). 
A triangulated category is called \emph{algebraic} when it is triangulated equivalent to the stable category of a Frobenius exact category  (see \cite{K2} and \cite{P}). 

Assuming that $\mathcal{D}$ has coproducts, for a sequence of morphisms 
\begin{equation}\label{star}
0=X_{-1}\stackrel{x_0}{\longrightarrow}X_0
\stackrel{x_1}{\longrightarrow}X_1\stackrel{x_2}{\longrightarrow}\cdots\stackrel{x_n}{\longrightarrow}X_n
\stackrel{x_{n+1}}{\longrightarrow}\cdots
\end{equation}
 let us denote by $\sigma:\coprod_{n\in\mathbb{N}}X_n\longrightarrow \coprod_{n\in\mathbb{N}}X_n$ the unique morphism such that $\sigma\circ\iota_k=\iota_{k+1}\circ x_{k+1}$, where $\iota_k:X_k\longrightarrow\coprod_{n\in\mathbb{N}}X_n$ is the canonical inclusion, for all $k\in\mathbb{N}$. The \emph{Milnor colimit} (or homotopy colimit \cite{N}) of the sequence is the object $X=\text{Mcolim}(X_n)$ which appears in the triangle $$\coprod_{n\in\mathbb{N}}X_n\stackrel{1-\sigma}{\longrightarrow} \coprod_{n\in\mathbb{N}}X_n\longrightarrow X\stackrel{+}{\longrightarrow}.$$ 
 We will frequently use the fact that,  when $\mathcal{D}$ has coproducts   and $\mathcal{S}$ is a set of compact objects, the pair 
$({}^\perp (\mathcal{S}^\perp),\mathcal{S}^\perp)$ is a torsion pair in $\mathcal{D}$ (see \cite[Theorem 4.3]{AI} or  \cite{AJS, NeTGen}).  In the case
of the t-structure (resp. semi-orthogonal decomposition) generated by $\mathcal{S}$, one has  
${}^\perp (\mathcal{S}^{\perp_{\leq 0}})=\text{Susp}_\mathcal{D}(\mathcal{S})$ 
(resp. ${}^\perp (\mathcal{S}^{\perp_{i\in\mathbb{Z}}})=\text{Loc}_\mathcal{D}(\mathcal{S})$) 
(see \cite[Theorem 12.1]{KN}, \cite[Lemma 2.3]{NS1}). Furthermore, the objects of 
$\text{Loc}_\mathcal{D}(\mathcal{S})$ (resp. $\text{Susp}_\mathcal{D}(\mathcal{S})$, for a non-positive $\mathcal{S}$) 
are precisely the Milnor colimits of sequences of the form (\ref{star}), where the cone of each $x_n$, denoted by $\text{cone}(x_n)$, is a coproduct of objects from 
$\bigcup_{k\in\mathbb{Z}}\mathcal{S}[k]$ (resp. of objects from $\mathcal{S}[n]$), for each $n\in\mathbb{N}$ (see the proof of 
\cite[Theorem 8.3.3]{N} for 
$\text{Loc}_\mathcal{D}(\mathcal{S})$, and 
\cite[Theorem 12.2]{KN} and \cite[Theorem 2]{NSZ} for $\text{Susp}_\mathcal{D}(\mathcal{S})$)  (see also \cite[Theorem 3.7]{PSt}).

A class of objects $\mathcal{T}$ is called \emph{non-positive} if
$\text{Hom}_\mathcal{D}(T,T'[i])=0$ for any $T,T' \in \mathcal{T}, i
>0$. Two non-positive sets $\mathcal{T}$ and $\mathcal{T}'$ are said to be \emph{$add$- (resp. $Add$-) equivalent} when $\text{add}(\mathcal{T})=\text{add}(\mathcal{T}')$ (resp. $\text{Add}(\mathcal{T})=\text{Add}(\mathcal{T}')$).

\subsection{DG categories and algebras:} A  \emph{differential graded (=dg) category} is
a category $\mathcal{A}$ such that, for each pair $(A,B)$ of its objects, the $K$-module of
morphisms, denoted by $\mathcal{A}(A,B)$, has a
structure of a differential graded $K$-module (or equivalently a
structure of a complex of $K$-modules) so that the composition map $\mathcal{A}(B,C)\otimes\mathcal{A}(A,B)\longrightarrow\mathcal{A}(A,C)$
($g\otimes f\rightsquigarrow g\circ f$) is a morphism of degree zero of the underlying graded $K$-modules which commutes with the differentials. This means that $d(g\circ f)=d(g)\circ f+(-1)^{|g|}g\circ d(f)$ whenever $g\in\mathcal{A}(B,C)$ and $f\in\mathcal{A}(A,B)$ are homogeneous morphisms and $|g|$ is the degree of $g$. The reader is referred to \cite{K} and \cite{K2} for details on dg categories.
The most important concept for us is the \emph{derived category of a small dg category}, denoted by $\mathcal{D}(\mathcal{A})$. It is the localization, in the sense of Gabriel-Zisman (\cite{GZ}) of $\mathcal{C}(\mathcal{A})$ with respect to the class of quasi-isomorphisms. Here $\mathcal{C}(\mathcal{A})$ denotes the category whose objects are the (right) dg $\mathcal{A}$-modules (i.e. the dg functors $M:\mathcal{A}^{op}\longrightarrow\mathcal{C}_{dg}K$, where $\mathcal{C}_{dg}K$ is the category of dg $K$-modules) and the  morphisms $f:M\longrightarrow N$ are the morphisms of degree zero in the underlying graded category which commute with the differentials. The category $\mathcal{D}(\mathcal{A})$ is triangulated and it turns out that, up to triangulated equivalence, the derived categories $\mathcal{D}(\mathcal{A})$ are precisely the compactly generated algebraic triangulated categories (see \cite[Theorem 4.3]{K}). The canonical set of compact generators of $\mathcal{D}(\mathcal{A})$ is the set of \emph{representable dg $\mathcal{A}$-modules} $\{A^\wedge\text{: }A\in\mathcal{A}\}$, where $A^\wedge :\mathcal{A}^{op}\longrightarrow\mathcal{C}_{dg}K$ takes $A'$ to $\mathcal{A}(A',A)$, for each $A'\in\mathcal{A}$. We will frequently use the fact that there is a natural isomorphism of $K$-modules $\text{Hom}_{\mathcal{D}(\mathcal{A})}(A^\wedge ,M[k])\cong H^k(M(A))$, for $A\in\mathcal{A}$ and $M\in\mathcal{D}(\mathcal{A})$. 

Two particular cases of small dg categories $\mathcal{A}$ will be of special interest  to us. Any small $K$-category  can be considered as a dg category 
concentrated in degree zero.  A dg algebra $A$, i.e. an associative unital graded algebra $A$ with a differential $d:A\longrightarrow A$ which satisfies the graded Leibniz rule,  is  a dg category with just one object. 
 The intersection of both cases  is the case of an associative unital algebra, 
called \emph{ordinary algebra} throughout the paper, which is then considered as a dg category with just one object 
concentrated in degree zero. Such an algebra will be called \emph{a finite length algebra} when it has finite length as a $K$-module 
(note that we are neither requiring $K$ to be a field nor an Artin ring). Note also that an \emph{Artin algebra} is just an ordinary algebra which is
finite length over its center, which, in turn, is a commutative Artin ring. Conversely, an ordinary $K$-algebra $A$ is finite length if, and only if, it  is an Artin algebra whose center is a (commutative) finite length $K$-algebra. For any ordinary algebra $A$, 
 we will denote by $\text{Mod-}A$ (resp. $\text{mod-}A$, $\text{fl-}A$,   $\text{Proj-}A$, $\text{proj-}A$) the category of all (resp. finitely presented, finite length, projective, finitely generated projective) 
right $A$-modules. We refer the reader to \cite{AF}, \cite{ASS},  \cite{ARS}  and \cite{LAZ}  for the classical terminology concerning ordinary rings, 
algebras and their modules.

\section{On lifting recollements and TTF triples} \label{section.some points}

In this section for thick subcategories of compactly generated algebraic triangulated categories we investigate the relation between the preservation of compactness by the functors of the recollement and lifting TTF triples and recollements. 

The following Lemma was proved in \cite[Lemma 2.4]{NS1} and \cite[Lemma 2.3]{NS1}.

\begin{lemma}\label{LemNS}
Let $\mathcal{T}$ be a triangulated category with coproducts. If $(\mathcal{E},\mathcal{F})$ is a compactly generated semi-orthogonal 
decomposition in $\mathcal{T}$, then  $\mathcal{E}$ 
is compactly generated as 
a triangulated category and $\mathcal{E}^c=\mathcal{E}\cap\mathcal{T}^c$. Moreover,  $\mathcal{F}$ has coproducts,
calculated as in $\mathcal{T}$, and the left adjoint 
$\tau:\mathcal{T}\longrightarrow\mathcal{F}$ of the inclusion functor preserves compact objects. When, in addition,
$\mathcal{T}$ is compactly generated, $\mathcal{F}$ is compactly generated by $\tau (\mathcal{T}^c)$. 
\end{lemma}

\begin{lemma}\label{LemmaNew}
Let $\hat{\mathcal{D}}$ be a compactly generated triangulated category and let $(\mathcal{U}_0,\mathcal{V}_0)$ be a semi-orthogonal decomposition of $\hat{\mathcal{D}}^c$.  Then $(\mathcal{U},\mathcal{V},\mathcal{W}):=(\text{Loc}_{\hat{\mathcal{D}}}(\mathcal{U}_0),\text{Loc}_{\hat{\mathcal{D}}}(\mathcal{V}_0), \text{Loc}_{\hat{\mathcal{D}}}(\mathcal{V}_0)^{\perp})$ is a TTF triple in $\hat{\mathcal{D}}$ such that $(\mathcal{U},\mathcal{V})$ and $(\mathcal{V},\mathcal{W})$ are compactly generated semiorthogonal decompositions. Moreover, $(\mathcal{U}\cap\hat{\mathcal{D}}^c,\mathcal{V}\cap\hat{\mathcal{D}}^c)=(\mathcal{U}_0,\mathcal{V}_0)=(\mathcal{U}^c,\mathcal{V}^c)$.
\end{lemma}

\begin{proof}
The pair $(\mathcal{V},\mathcal{W})$ is clearly a torsion pair. We need to prove that $(\mathcal{U},\mathcal{V})$ is a torsion pair in $\hat{\mathcal{D}}$. The argument is standard and can be found in the literature (see \cite{NS2}). We sketch it, leaving some details to the reader. Since objects in $\mathcal{U}_0$ are compact and objects in $\mathcal{V}$ are Milnor colimits of sequences of morphisms with cones in $\text{Add}(\mathcal{V}_0)$, we see that $\mathcal{V}\subseteq\mathcal{U}^\perp$. For $M\in\mathcal{U}^\perp =(\mathcal{U}_0)^\perp$ let us consider the truncation triangle $V\longrightarrow M\longrightarrow W\stackrel{+}{\longrightarrow}$ with respect to $(\mathcal{V},\mathcal{W})$. We get $W\in\mathcal{U}^\perp $ and $W\in\mathcal{V}^\perp$. Therefore, $W\in (\mathcal{U}_0\cup \mathcal{V}_0)^\perp$. But for each $D\in\hat{\mathcal{D}}^c$ there is a triangle $U_0\longrightarrow D\longrightarrow V_0\stackrel{+}{\longrightarrow}$, whose outer terms are in $\mathcal{U}_0$ and $\mathcal{V}_0$, respectively. It follows that $\text{Hom}_{\hat{\mathcal{D}}}(D,W)=0$, for all $D\in\hat{\mathcal{D}}^c$. This implies that $W=0$, so $V\cong M$ belongs to $\mathcal{V}$. Then the pair $(\mathcal{U},\mathcal{V})$ is of the form $(\text{Loc}_{\hat{\mathcal{D}}}(\mathcal{U}_0),\text{Loc}_{\hat{\mathcal{D}}}(\mathcal{U}_0)^{\perp})$, and hence is a torsion pair. The torsion pairs $(\mathcal{U},\mathcal{V})$ and $(\mathcal{V},\mathcal{W})$ are  
compactly generated by  sets $\mathcal{U}_0$ and $\mathcal{V}_0$, respectively.

By Lemma \ref{LemNS}, we know that $(\mathcal{U}\cap\hat{\mathcal{D}}^c,\mathcal{V}\cap\hat{\mathcal{D}}^c)=(\mathcal{U}^c,\mathcal{V}^c)$. Since $\mathcal{U}_0$ is a  thick subcategory of a compactly generated triangulated category $\mathcal{U}$ which generates $\mathcal{U}$ and consists of compact objects of $\mathcal{U}$, we get $\mathcal{U}_0=\mathcal{U}^c$. Similarly, $\mathcal{V}_0=\mathcal{V}^c$.
\end{proof}

The key result of the section is the following. 

\begin{theorem} \label{thm.lifting of recollements}
Let $$
\begin{xymatrix}{\mathcal{Y} \ar[r]^{i_*}& \mathcal{D} \ar@<3ex>[l]_{i^!}\ar@<-3ex>[l]_{i^*}\ar[r]^{j^*} & \mathcal{X} \ar@<3ex>_{j_*}[l]\ar@<-3ex>_{j_!}[l]}
\end{xymatrix}
$$ be a recollement, where $\mathcal{Y}$, $\mathcal{D}$ and $\mathcal{X}$ are thick subcategories of compactly generated triangulated categories  $\hat{\mathcal{Y}}$, $\hat{\mathcal{D}}$ and $\hat{\mathcal{X}}$ which contain the respective subcategories of compact objects. Consider the following assertions:

\begin{enumerate}
\item The given recollement lifts to a recollement $$
\begin{xymatrix}{\hat{\mathcal{Y}} \ar[r]^{\hat{i}_*}& \hat{\mathcal{D}} \ar@<3ex>[l]_{\hat{i}^!}\ar@<-3ex>[l]_{\hat{i}^*}\ar[r]^{\hat{j}^*} & \hat{\mathcal{X}} \ar@<3ex>_{\hat{j}_*}[l]\ar@<-3ex>_{\hat{j}_!}[l]}
\end{xymatrix}$$ which is the upper part of a ladder of recollements of height two. 

\item The TTF triple $(\text{Im}(j_!),\text{Im}(i_*),\text{Im}(j_*))$ in $\mathcal{D}$ lifts to a  TTF triple $(\mathcal{U},\mathcal{V},\mathcal{W})$ in $\hat{\mathcal{D}}$ such that:

\begin{enumerate}
\item The torsion pairs $(\mathcal{U},\mathcal{V})$ and $(\mathcal{V},\mathcal{W})$ are compactly generated (whence $(\mathcal{U},\mathcal{V},\mathcal{W})$ is extendable to the right);
\item $j_!(\hat{\mathcal{X}}^c)=\mathcal{U}\cap\hat{\mathcal{D}}^c$ and $i_*(\hat{\mathcal{Y}}^c)=\mathcal{V}\cap\hat{\mathcal{D}}^c$.
\end{enumerate}

\item The functors $j_!$, $j^*$, $i^*$ and $i_*$ preserve compact objects, i.e. $j_!(\hat{\mathcal{X}}^c)\subseteq\hat{\mathcal{D}}^c$ and similarly for  $j^*$, $i^*$ and $i_*$.
\end{enumerate}

The implications $1)\Longrightarrow 2)\Longrightarrow 3)$ hold. Moreover, when  $\text{Im}(i_*)$ cogenerates $\text{Loc}_{\hat{\mathcal{D}}}(i_*(\hat{\mathcal{Y}}^c))$ or $\mathcal{D}$ cogenerates $\hat{\mathcal{D}}$, the implication $3)\Longrightarrow 2)$  also  holds.
\end{theorem}
\begin{proof}

$1)\Longrightarrow 2)$ The functors $\hat{j}_!$, $\hat{j}^*$, $\hat{i}^*$ and $\hat{i}_*$ preserve compact objects,  since they have right 
adjoints which preserve coproducts, because they also have right adjoints.
Let us consider the TTF triple $(\mathcal{U},\mathcal{V},\mathcal{W}):=(\text{Im}(\hat{j}_!),\text{Im}(\hat{i}_*),\text{Im}(\hat{j}_*))$ 
associated to the recollement from assertion 1. The torsion pair $(\mathcal{U},\mathcal{V})$ is generated by  $\hat{j}_!(\hat{\mathcal{X}}^c)=j_!(\hat{\mathcal{X}}^c)$. Indeed, the inclusion 
$(\hat{j}_!(\hat{\mathcal{X}}^c))^{\perp}\supseteq (\text{Im}(\hat{j}_!))^{\perp}$ is obvious, the inverse inclusion follows from the fact 
that, by  infinite d\'evissage, $\hat{\mathcal{X}}=\text{Loc}_{\hat{\mathcal{X}}}(\hat{\mathcal{X}}^c)$ 
 and $\hat{j}_!$ commutes with coproducts.  Since $j_!(\hat{\mathcal{X}}^c)$ consists of compact objects and is skeletally small there is a set of compact objects generating $(\mathcal{U},\mathcal{V})$. Similarly, the torsion pair $(\mathcal{V},\mathcal{W})$ is generated by $\hat{i}_*(\hat{\mathcal{Y}}^c)=i_*(\hat{\mathcal{Y}}^c)$, and hence by a set of compact objects. Thus condition 2.a holds and, moreover, we have inclusions $j_!(\hat{\mathcal{X}}^c)\subseteq\mathcal{U}\cap\hat{\mathcal{D}}^c$ and $i_*(\hat{\mathcal{Y}}^c)\subseteq\mathcal{V}\cap\hat{\mathcal{D}}^c$. On the other hand, if $U\in\mathcal{U}\cap\hat{\mathcal{D}}^c$, then $\hat{j}^*U\in\hat{\mathcal{X}}^c$. Choosing now $X\in\hat{\mathcal{X}}$ such that $U=\hat{j}_!X$, we have $X\cong \hat{j}^*\hat{j}_!X\cong\hat{j}^*U\in\hat{\mathcal{X}}^c$ and, hence, $U\cong j_!X\in j_!(\hat{\mathcal{X}}^c)$. Similarly, $\mathcal{V}\cap\hat{\mathcal{D}}^c\subseteq i_*(\hat{\mathcal{Y}}^c)$. 

$2)\Longrightarrow 3)$ By condition 2.b we know that $j_!$ and $i_*$ preserve compact objects. Moreover, $j_!(\hat{\mathcal{X}}^c)=\text{Im}(j)\cap\hat{\mathcal{D}}^c$ and $i_*(\hat{\mathcal{Y}}^c)=\text{Im}(i_*)\cap\hat{\mathcal{D}}^c$, since $\mathcal{U}\cap\mathcal{D}=\text{Im}(j_!)$ and $\mathcal{V}\cap\mathcal{D}=\text{Im}(i_*)$. This implies that $j_!$ and $i_*$ also reflect compact objects, i.e. $j_!(X)\in\hat{\mathcal{D}}^c$ (resp.  $i_*(Y)\in\hat{\mathcal{D}}^c$) if and only if $X\in\hat{\mathcal{X}}^c$ (resp. $Y\in\hat{\mathcal{Y}}^c$). 
For any $D\in\hat{\mathcal{D}}^c$ let us consider the associated triangle $j_!j^*D\longrightarrow D\longrightarrow i_*i^*D\stackrel{+}{\longrightarrow}$, we get that $j^*$ preserves compact objects if and only if so does $i^*$.

Let us prove that $i^*$ preserves compact objects. Since the semi-orthogonal decomposition $(\text{Im}(j_!),\text{Im}(i_*))$ is the restriction to $\mathcal{D}$ of  $(\mathcal{U},\mathcal{V})$,  the associated truncation functor $\tau:\hat{\mathcal{D}}\longrightarrow\mathcal{V}$  has the property that $\tau(\mathcal{D})=\mathcal{V}\cap\mathcal{D}=\text{Im}(i_*)$. By Lemma \ref{LemNS}, we know that $\tau(\hat{\mathcal{D}^c})=\mathcal{V}^c=\mathcal{V}\cap\hat{\mathcal{D}}^c=\text{Im}(i_*)\cap\hat{\mathcal{D}}^c$. We next decompose $i_*$ as $\mathcal{Y}\stackrel{\cong}{\longrightarrow}\text{Im}(i_*)=\mathcal{D}\cap\mathcal{V}\hookrightarrow\mathcal{D}$, where the first arrow $i_*$  is an equivalence of categories. Then the left adjoint $i^*$ is naturally isomorphic to the composition $\mathcal{D}\stackrel{\tau}{\longrightarrow}\text{Im}(i_*)=\mathcal{D}\cap\mathcal{V}\stackrel{i_*^{-1}}{\longrightarrow}\mathcal{Y}$. But $i_*^{-1}(\mathcal{V}\cap\hat{\mathcal{D}}^c)=\hat{\mathcal{Y}}^c$, since $i_*(\hat{\mathcal{Y}}^c)=\mathcal{V}\cap\hat{\mathcal{D}}^c$. Therefore we get that $i^*(\hat{\mathcal{D}}^c)=\hat{\mathcal{Y}}^c$.

$3)\Longrightarrow 2)$ (assuming  any of the extra hypotheses). Setting $\mathcal{U}_0:=j_!(\hat{\mathcal{X}}^c)$ and $\mathcal{V}_0:=i_*(\hat{\mathcal{Y}}^c)$, by Lemma \ref{LemmaNew}, $(\mathcal{U},\mathcal{V},\mathcal{W}):=(\text{Loc}_{\hat{\mathcal{D}}}(j_!(\hat{\mathcal{X}}^c)),\text{Loc}_{\hat{\mathcal{D}}}(i_*(\hat{\mathcal{Y}}^c)), \text{Loc}_{\hat{\mathcal{D}}}(i_*(\hat{\mathcal{Y}}^c))^{\perp})$ is a TTF-triple with $(\mathcal{U},\mathcal{V})$ and $(\mathcal{V},\mathcal{W})$ compactly generated by $j_!(\hat{\mathcal{X}}^c)$ and $i_*(\hat{\mathcal{Y}}^c)$. 

We need to check that $\mathcal{U}\cap\mathcal{D}=\text{Im}(j_!)$ and $\mathcal{V}\cap\mathcal{D}=\text{Im}(i_*)$. The equality $\text{Im}(j_*)=\mathcal{W}\cap\mathcal{D}$ will then follow automatically. Indeed, the inclusion $\text{Im}(j_*)\subseteq \mathcal{W}\cap\mathcal{D}$ is obvious, the other inclusion follows from orthogonality.
 By properties of recollements (see \cite{BBD}),  $\text{Im}(i_*)=\text{Ker}(j^*)$. Since $\hat{\mathcal{X}}$ is compactly generated, an object  $D$ of $\mathcal{D}$ belongs to $\text{Ker}(j^*)$ if and only if $0=\text{Hom}_\mathcal{X}(X,j^*D)\cong\text{Hom}_\mathcal{D}(j_!X,D)$, for all $X\in\hat{\mathcal{X}}^c$. This happens exactly when $D\in\text{Loc}_{\hat{\mathcal{D}}}(j_!(\hat{\mathcal{X}}^c))^\perp\cap\mathcal{D}=\mathcal{V}\cap\mathcal{D}$, and thus $\mathcal{V}\cap\mathcal{D}=\text{Im}(i_*)$. 

Let us check that $\mathcal{U}\cap\mathcal{D}=\text{Im}(j_!)$. Since each object of $\mathcal{U}$ is the Milnor colimit of a 
sequence of morphisms in $\hat{\mathcal{D}}$ with successive cones in $\text{Add}(j_!(\hat{\mathcal{X}})^c)$ and 
since $\text{Hom}_{\hat{\mathcal{D}}}(j_!X,-)$ vanishes on $\text{Im}(i_*)$, for each $X\in\mathcal{X}$, we get that $\mathcal{U}\cap\mathcal{D}\subseteq\text{Im}(j_!)$. Indeed, $(\text{Im}(j_!),\text{Im}(i_*))$ is a torsion pair in $\mathcal{D}$ and hence $\text{Im}(j_!)= {}^\perp\text{Im}(i_*)\cap\mathcal{D}$.  Conversely, for $D\in\text{Im}(j_!)$ let us consider the truncation triangle $U\longrightarrow D\longrightarrow V\stackrel{+}{\longrightarrow}$ in $\hat{\mathcal{D}}$ with respect to $(\mathcal{U},\mathcal{V})$. As before, $\text{Hom}_{\hat{\mathcal{D}}}(U,-)$ vanishes on $\text{Im}(i_*)$, and hence $\text{Hom}_{\hat{\mathcal{D}}}(V,-)$ vanishes on $\text{Im}(i_*)$. In the assumption that $\text{Im}(i_*)$ cogenerates $\mathcal{V}=\text{Loc}_{\hat{\mathcal{D}}}(i_*(\hat{\mathcal{Y}}^c))$, we immediately get $V=0$. In the other case, we also have that 
 $\text{Hom}_{\hat{\mathcal{D}}}(V,-)$ vanishes on $\mathcal{W}$ and, hence, it also vanishes on $\text{Im}(j_*)$. It follows that $\text{Hom}_{\hat{\mathcal{D}}}(V,-)$ vanishes both on $\text{Im}(j_*)$ and $\text{Im}(i_*)$. This implies that $\text{Hom}_{\hat{\mathcal{D}}}(V,-)$ vanishes on $\mathcal{D}$, and hence that $V=0$ since, by hypothesis,  $\mathcal{D}$ cogenerates $\hat{\mathcal{D}}$. Under both extra hypotheses, we then get that $U\cong D\in\mathcal{U}\cap\mathcal{D}$. 

Let us  prove the inclusions $\mathcal{U}\cap\hat{\mathcal{D}}^c\subseteq j_!(\hat{\mathcal{X}}^c)$ and  $\mathcal{V}\cap\hat{\mathcal{D}}^c\subseteq i_*(\hat{\mathcal{Y}}^c)$, the inverse inclusions are obvious. For $U\in\mathcal{U}\cap\hat{\mathcal{D}}^c \subseteq \text{Im}(j_!)$ the adjunction map $j_!j^*(U)\longrightarrow U$ is an isomorphism. It follows that $U\in j_!(\hat{\mathcal{X}}^c)$, since $j^*(U)\in\hat{\mathcal{X}}^c$. The second inclusion is analogous.\end{proof}

\begin{corollary}
If in Theorem \ref{thm.lifting of recollements} we assume that $\mathcal{D}=\hat{\mathcal{D}}^c$, then assertion 2 of the theorem holds if and only if $\mathcal{Y}=\hat{\mathcal{Y}}^c$ and $\mathcal{X}=\hat{\mathcal{X}}^c$.
\begin{comment}
Let $\mathcal{D}:=\hat{\mathcal{D}}^c$ be the subcategory of compact objects of a compactly generated triangulated category $\hat{\mathcal{D}}$. If $(\mathcal{U}_0,\mathcal{V}_0)$ is a semi-orthogonal decomposition of $\mathcal{D}$, then there is a TTF triple $(\mathcal{U},\mathcal{V},\mathcal{W})$ in $\hat{\mathcal{D}}$ which satisfies the following properties:
\begin{enumerate}
\item The semi-orthogonal decompositions $(\mathcal{U},\mathcal{V})$ and $(\mathcal{V},\mathcal{W})$ of $\hat{\mathcal{D}}$ are compactly generated. 
\item $(\mathcal{U}\cap\mathcal{D},\mathcal{V}\cap\mathcal{D})=(\mathcal{U}_0,\mathcal{V}_0)=(\mathcal{U}^c,\mathcal{V}^c)$.
\end{enumerate}
\end{comment}
\end{corollary}
\begin{proof}
\begin{comment}
Applying the argument from the first paragraph of the proof of implication $3)\Longrightarrow 2)$ in  Theorem \ref{thm.lifting of recollements}, to any semi-orthogonal decomposition $(\mathcal{U}_0,\mathcal{V}_0)$ of $\hat{\mathcal{D}}^c$ instead of $(j_!(\hat{\mathcal{X}}^c),i_*(\hat{\mathcal{Y}}^c))$, we get that $(\mathcal{U},\mathcal{V},\mathcal{W}):=(\text{Loc}_{\hat{\mathcal{D}}}(\mathcal{U}_0),\text{Loc}_{\hat{\mathcal{D}}}(\mathcal{V}_0),\text{Loc}_{\hat{\mathcal{D}}}(\mathcal{V}_0)^\perp)$ is a TTF triple in $\hat{\mathcal{D}}$ such that its constituent torsion pairs $(\mathcal{U},\mathcal{V})$ and $(\mathcal{V},\mathcal{W})$ are compactly generated.  Moreover, by the first paragraph of the proof of Theorem 3.1, we also know that $(\mathcal{U}\cap\mathcal{D},\mathcal{V}\cap\mathcal{D})=(\mathcal{U}^c,\mathcal{V}^c)$. Since $\mathcal{U}_0$ is a  thick subcategory of a compactly generated triangulated category $\mathcal{U}$ which generates $\mathcal{U}$ and consists of compact objects of $\mathcal{U}$, we get $\mathcal{U}_0=\mathcal{U}^c$. Similarly, $\mathcal{V}_0=\mathcal{V}^c$. 
\end{comment}
Let us suppose that $\mathcal{D}=\hat{\mathcal{D}}^c$ in the recollement of Theorem \ref{thm.lifting of recollements}. If assertion 2 of the theorem holds so does assertion 3, and hence the functors $j_!$, $j^*$, $i^*$ and $i_*$ preserve compact objects. Hence, $\text{Im}(j^*)\subseteq\hat{\mathcal{X}}^c$  and $\text{Im}(i^*)\subseteq\hat{\mathcal{Y}}^c$. It follows that $\mathcal{Y}=\hat{\mathcal{Y}}^c$ and  $\mathcal{X}=\hat{\mathcal{X}}^c$, since the functors $i^*$ and $j^*$ are dense, for any recollement. Conversely, if $\mathcal{Y}=\hat{\mathcal{Y}}^c$ and $\mathcal{X}=\hat{\mathcal{X}}^c$, then $(\text{Im}(j_!),\text{Im}(i_*))=(j_!(\hat{\mathcal{X}}^c),i_*(\hat{\mathcal{Y}}^c))$ is a semi-orthogonal decomposition of $\mathcal{D}=\hat{\mathcal{D}}^c$ and by Lemma \ref{LemmaNew} there is a TTF triple $(\mathcal{U},\mathcal{V},\mathcal{W})=(\text{Loc}_{\hat{\mathcal{D}}}(\text{Im}(j_!)),\text{Loc}_{\hat{\mathcal{D}}}(\text{Im}(i_*)),\text{Loc}_{\hat{\mathcal{D}}}(\text{Im}(i_*))^\perp)$ in $\hat{\mathcal{D}}$,  with compactly generated constituent torsion pairs, such that $(\mathcal{U}\cap\mathcal{D},\mathcal{V}\cap\mathcal{D})=(\text{Im}(j_!),\text{Im}(i_*))$. Since $\mathcal{W}\cap\mathcal{D}=\mathcal{V}^\perp\cap\mathcal{D}=\text{Im}(i_*)^\perp\cap\mathcal{D}=\text{Im}(j_*)$, the TTF triple $(\mathcal{U},\mathcal{V},\mathcal{W})$ satisfies all the conditions of assertion 2 in Theorem \ref{thm.lifting of recollements}.
\end{proof}

We immediately get:

\begin{corollary}
Let $\hat{\mathcal{Y}}$,  $\hat{\mathcal{D}}$ and  $\hat{\mathcal{X}}$ be compactly generated triangulated categories and suppose that we have a recollement $$
\begin{xymatrix}{\hat{\mathcal{Y}}^c \ar[r]^{i_*}& \hat{\mathcal{D}}^c \ar@<3ex>[l]_{i^!}\ar@<-3ex>[l]_{i^*}\ar[r]^{j^*} & \hat{\mathcal{X}}^c \ar@<3ex>_{j_*}[l]\ar@<-3ex>_{j_!}[l]}
\end{xymatrix}
$$
Then the TTF triple $(\text{Im}(j_!),\text{Im}(i_*),\text{Im}(j_*))$ in $\hat{\mathcal{D}}^c$ lifts to a TTF triple $(\mathcal{U},\mathcal{V},\mathcal{W})$ in $\hat{\mathcal{D}}$, where the torsion pairs $(\mathcal{U},\mathcal{V})$ and $(\mathcal{V},\mathcal{W})$ are compactly generated. 
\end{corollary}

Here is a list of examples where the last corollary applies:

\begin{example}\label{ex repet} Let us consider one of the following  situations:
\begin{enumerate}
\item Let $A$, $B$ and $C$ be  finite dimensional algebras over a field. Let the triple $\mathcal{Y}\equiv \mathcal{D} \equiv \mathcal{X}$ be $\underline{mod}\text{-}\hat{B}\equiv \underline{mod}\text{-}\hat{A} \equiv \underline{mod}\text{-}\hat{C}$, where $\underline{mod}\text{-}\hat{A}$ is the stable category of the repetitive algebra $\hat{A}$ of $A$ (see \cite[Section 2.2]{H}) and let $\hat{\mathcal{D}}$ be $\underline{Mod}\text{-}\hat{A}$.

\item Let $A$, $B$ and $C$ be self-injective finite length algebras. Let the triple $\mathcal{Y}\equiv \mathcal{D} \equiv \mathcal{X}$ be $\underline{mod}\text{-}B\equiv \underline{mod}\text{-}A \equiv \underline{mod}\text{-}C$, where $\underline{mod}\text{-}A$ is the stable category of $A$ and let $\hat{\mathcal{D}}$ be $\underline{Mod}\text{-}A$.

\item Let $\mathbf{U}$, $\mathbf{X}$ and $\mathbf{Z}$ be separated Noetherian schemes. Let the triple $\mathcal{Y}\equiv \mathcal{D} \equiv \mathcal{X}$ be $\mathcal{D}^b(\text{coh}(\mathbf{U}))\equiv \mathcal{D}^b(\text{coh}(\mathbf{X})) \equiv \mathcal{D}^b(\text{coh}(\mathbf{Z}))$ and let $\hat{\mathcal{D}}$ be $\mathcal{K}(\text{Inj-}\mathbf{X})$ -- the homotopy category of injective objects of $\text{Qcoh}(\mathbf{X})$.

\item Let $A$, $B$ and $C$ be right Noetherian rings. Let the triple $\mathcal{Y}\equiv \mathcal{D} \equiv \mathcal{X}$ be $\Db B\equiv \Db A \equiv \Db C$ and let $\hat{\mathcal{D}}$ be the homotopy category of injectives $\mathcal{K}(\text{Inj-}A)$.
\end{enumerate}

Suppose there is a recollement 
$$
\begin{xymatrix}{\mathcal{Y} \ar[r]^{i_*}& \mathcal{D} \ar@<3ex>[l]_{i^!}\ar@<-3ex>[l]_{i^*}\ar[r]^{j^*} & \mathcal{X} \ar@<3ex>_{j_*}[l]\ar@<-3ex>_{j_!}[l]}
\end{xymatrix}
$$
and consider $\mathcal{D}$ as a full triangulated subcategory of $\hat{\mathcal{D}}$, then there exists  a TTF triple $(\mathcal{U},\mathcal{V},\mathcal{W})$ in $\hat{\mathcal{D}}$, with compactly generated constituent torsion pairs $(\mathcal{U},\mathcal{V})$ and $(\mathcal{V},\mathcal{W})$,  which restricts to the TTF triple $(\text{Im}(j_!),\text{Im}(i_*),\text{Im}(j_*))$ in $\mathcal{D}$.
\end{example}
\begin{proof}
In all the cases, it turns out that the three categories of the recollements are the subcategories of compact objects in the appropriate compactly generated triangulated categories.  

1) Let $\hat{\Lambda}$ denotes the repetitive algebra of $\Lambda $.
It is easy to see that the category $\text{Mod}\text{-}\hat{\Lambda}$ of unitary $\hat{\Lambda}$-modules (i.e. modules $M$ such that $M\hat{\Lambda}=M$) is a Frobenius category, so that its stable category $\underline{Mod}\text{-}\hat{\Lambda}$ is triangulated and compactly generated. Its subcategory of compact objects is precisely $\underline{mod}\text{-}\hat{\Lambda}$, where $\text{mod}\text{-}\hat{\Lambda}$ is the subcategory of $\text{Mod}\text{-}\hat{\Lambda}$ consisting of the finitely generated $\hat{\Lambda}$-modules, which coincide with the $\hat{\Lambda}$-modules of finite length.

2) It is well-known that if $\Lambda$ is a self-injective  Artin algebra, in particular a self-injective finite length algebra, then its module category $\text{Mod}\text{-}\Lambda$ is Frobenius and its associated triangulated stable category $\underline{Mod}\text{-}\Lambda$ has $\underline{mod}\text{-}\Lambda$ as its subcategory of compact objects. 

3) and 4)  By \cite[Theorem 1.1]{Kra}, we can identify $\mathcal{D}^b(\text{coh}(\mathbf{Y}))$ with the subcategory of compact objects of $\mathcal{K}(\text{Inj}\text{-}\mathbf{Y})$, for any separated Noetherian scheme $\mathbf{Y}$, and by \cite[Proposition 2.3]{Kra}, we can identify $\Db R$ with the subcategory of compact objects of $\mathcal{K}(\text{Inj-}R)$, for any right Noetherian ring $R$ (here $\text{mod-}R$ is the category of finitely generated $R$-modules, which coincides with that of Noetherian modules).

With all these considerations in mind, the result is now a direct consequence of the previous corollary.
\end{proof}

In order to provide some examples where condition 3 of  the last theorem implies condition 2, we introduce the following terminology.

\begin{term} \label{term}
Given any triangulated category  $\mathcal{D}$ and any class $\mathcal{X}$ of its objects, we denote by 
$\mathcal{D}_\mathcal{X}^-$ (resp. $\mathcal{D}_\mathcal{X}^+$ or  $\mathcal{D}_\mathcal{X}^b$) the (thick) subcategory of $\mathcal{D}$ 
consisting of objects $M$ such that, for each  $X\in\mathcal{X}$, one has $\text{Hom}_\mathcal{D}(X,M[k])=0$ for $k\gg 0$ (resp. $k\ll 0$ or 
$|k| \gg 0$). We denote by $\mathcal{D}_{\mathcal{X},fl}$  the (thick) subcategory of $\mathcal{D}$ consisting of objects $M$ such that, for 
each $X\in\mathcal{X}$ and each $k\in\mathbb{Z}$, the $K$-module $\text{Hom}_\mathcal{D}(X,M[k])$ is of finite length. We finally put 
$\mathcal{D}^\star_{\mathcal{X},\dagger}=\mathcal{D}_\mathcal{X}^\star\cap\mathcal{D}_{\mathcal{X},\dagger}$, for 
$\star\in\{\emptyset,+,-,b\}$ and $\dagger\in\{\emptyset,fl\}$. In the particular case when $\mathcal{D}$ is compactly generated and 
$\mathcal{X}=\mathcal{D}^c$, we will simply write $\mathcal{D}_\dagger^\star$ instead of $\mathcal{D}_{\mathcal{X},\dagger}^\star$. Note that,
 in order to define $\mathcal{D}_\dagger^\star$ in the latter case, one can replace $\mathcal{D}^c$ by any set $\mathcal{X}$ of compact 
generators of $\mathcal{D}$, since $\mathcal{D}^c=\text{thick}_\mathcal{D}(\mathcal{X})$. Let us denote by 
$\mathbf{P}_\dagger ^\star$  the property that defines the full subcategory ${\mathcal{D}}^*_\dagger$ of ${\mathcal{D}}$. For instance, 
if $\star = -$ and $\dagger =fl$, then, for a given $M\in{\mathcal{D}}$, we will say that $M$ satisfies property $\mathbf{P}^{\star}_\dagger$, 
for some $X\in{\mathcal{D}}^c$, when $\text{Hom}_{{\mathcal{D}}}(X,M[k])$ is zero, for $k\gg 0$, and is a $K$-module of finite length, 
for all $k\in\mathbb{Z}$. 
\end{term}

\begin{example} \label{ex.dg category}
If $\mathcal{A}$ is a small dg category, then $\mathcal{D}^-(\mathcal{A}):=\mathcal{D}(\mathcal{A})^-$ (resp.  $\mathcal{D}^+(\mathcal{A}):=\mathcal{D}(\mathcal{A})^+$ or $\mathcal{D}^b(\mathcal{A}):=\mathcal{D}(\mathcal{A})^b$) is the subcategory of $\mathcal{D}(\mathcal{A})$ consisting of dg $\mathcal{A}$-modules $M$ such that, for each  $A\in\mathcal{A}$, one has $H^kM(A)=0$ for $k\gg 0$ (resp. $k\ll 0$ or $|k| \gg 0$). Similarly, for $\star\in\{\emptyset,+, -,b\}$, one has $\mathcal{D}^\star_{fl}(\mathcal{A})$  consists of dg $\mathcal{A}$-modules $M\in\mathcal{D}^\star(\mathcal{A})$ such that $H^kM(A)$ is a $K$-module of finite length, for each $A\in\mathcal{A}$ and each $k\in\mathbb{Z}$.
\end{example}

\begin{remark} \label{rem.from lifting-and-restricting}
In \cite{NS2} $\mathcal{D}^-(\mathcal{A})$ was defined as the union $\bigcup_{k\geq 0}\mathcal{U}[k]$, where $\mathcal{U}=\mathcal{D}^{\leq 0}(\mathcal{A})$. Here $\mathcal{D}^{\leq 0}(\mathcal{A})=\text{Susp}_{\mathcal{D}(\mathcal{A})}(A^\wedge\text{: }A\in\mathcal{A})$, which is the aisle of a t-structure in $\mathcal{D}(\mathcal{A})$. That definition does not agree in general with the one given here, although they coincide when $\mathcal{A}=A$ is a dg algebra.  
\end{remark}

\begin{definition} \label{def.locally bounded cg category}
Let $\mathcal{D}$ be a compactly generated triangulated category. We will say that $\mathcal{D}$ is \emph{homologically locally bounded} when $\mathcal{D}^c\subseteq\mathcal{D}^b$ and $\mathcal{D}$ is  \emph{homologically locally finite length} when  $\mathcal{D}^c\subseteq\mathcal{D}^b_{fl}$. 
\end{definition}

\begin{example} \label{ex.homologically locally  bounded dg category}
If $\mathcal{A}$ is a small dg category and $\mathcal{D}=\mathcal{D}(\mathcal{A})$ is its derived category, 
then $\mathcal{D}$ is homologically locally bounded if and only if the set $\{k\in\mathbb{Z}\text{: }H^k\mathcal{A}(A,A')\neq 0\}$ is finite, for all $A,A'\in\mathcal{A}$. Moreover, $\mathcal{D}$ is homologically locally finite length if, in addition, $H^k\mathcal{A}(A,A')$ is a $K$-module of finite length, for all $k\in\mathbb{Z}$ and all $A,A'\in\mathcal{A}$. Slightly abusing the terminology, we will say in those cases that $\mathcal{A}$ is a \emph{homologically locally bounded} or a \emph{homologically locally finite length dg category}, respectively. When $\mathcal{A}=A$ is a dg algebra, we will simply say that $A$ is  \emph{homologically  bounded} if $H^k(A)=0$, for almost all $k\in\mathbb{Z}$, or that $A$ is  \emph{homologically finite length} if $H^*(A):=\oplus_{k\in\mathbb{Z}}H^k(A)$ is a $K$-module of finite length.
\end{example}

We are ready to give  examples where condition 3 of Theorem \ref{thm.lifting of recollements} implies condition 2.

\begin{corollary} \label{cor.Case D homologically loc. bounded}
Let $\hat{\mathcal{Y}}$, $\hat{\mathcal{D}}$ and $\hat{\mathcal{X}}$ be compactly generated  triangulated categories. For $\star\in\{\emptyset ,+,-,b\}$ and $\dagger\in\{\emptyset, fl\}$ let $$
\begin{xymatrix}{\hat{\mathcal{Y}}^{\star}_\dagger \ar[r]^{{i}_*}& \hat{\mathcal{D}}^{\star}_\dagger \ar@<3ex>[l]_{{i}^!}\ar@<-3ex>[l]_{{i}^*}\ar[r]^{{j}^*} & \hat{\mathcal{X}}^{\star}_\dagger \ar@<3ex>_{{j}_*}[l]\ar@<-3ex>_{{j}_!}[l]}
\end{xymatrix}$$
be a recollement, such that the subcategories involved contain the respective subcategories of compact objects and such that the functors $j_!,j^*,i^*,i_*$ preserve compact objects. If $\hat{\mathcal{D}}$ is homologically locally bounded (resp. homologically locally finite length), then the subcategory  $\hat{\mathcal{D}}^{\star}$ (resp. $\mathcal{D}^{\star}_{fl}$) cogenerates $\hat{\mathcal{D}}$, and hence assertion 2 of Theorem \ref{thm.lifting of recollements} holds for $\dagger =\emptyset$ (resp. $\dagger =fl$). 
\end{corollary}
\begin{proof}
Let us check that $\hat{\mathcal{D}}^b$ (resp. $\hat{\mathcal{D}}^b_{fl}$) cogenerates $\hat{\mathcal{D}}$. For this take a minimal injective cogenerator $E$ of $\text{Mod-}K$ and use the notion of Brown-Comenetz dual. Since compactly generated (or even well-generated) triangulated categories satisfy Brown representability theorem (see \cite[Proposition 8.4.2]{N}), for each $X\in\hat{\mathcal{D}}^c$, the functor $\text{Hom}_K(\text{Hom}_{\hat{\mathcal{D}}}(X,-),E):\hat{\mathcal{D}}^{op}\longrightarrow\text{Mod-}K$ is naturally isomorphic to the representable functor $\text{Hom}_{\hat{\mathcal{D}}}(-,D(X))$, for an object $D(X)$, uniquely determined up to isomorphism,  called the \emph{Brown-Comenetz dual} of $X$. It immediately follows that $\{D(X)\text{: }X\in\hat{\mathcal{D}}^c\}$ is a skeletally small cogenerating class of $\hat{\mathcal{D}}$. Our task  reduces to check that $D(X)\in\hat{\mathcal{D}}^b$ (resp.  $D(X)\in\hat{\mathcal{D}}^b_{fl}$), when $\hat{\mathcal{D}}$ is homologically locally bounded (resp. homologically locally finite length).    But this is clear since, given any $Y,X\in\hat{\mathcal{D}}^c$, we have that $\text{Hom}_{\hat{\mathcal{D}}}(Y,D(X)[k])\cong\text{Hom}_K(\text{Hom}_{\hat{\mathcal{D}}}(Y,X[-k]),E)$ and the homologically  locally bounded condition on $\hat{\mathcal{D}}$ implies that $\text{Hom}_{\hat{\mathcal{D}}}(Y,X[-k])=0$, for all but finitely many $k\in\mathbb{Z}$. When $\hat{\mathcal{D}}$ is homologically locally finite length, we have in addition that  each $\text{Hom}_{\hat{\mathcal{D}}}(Y,X[-k])$ is of finite length as $K$-module, which implies that the same is true for $\text{Hom}_{\hat{\mathcal{D}}}(Y,D(X)[k]) $.
\end{proof}
Our next result, inspired by the analogous results for derived categories of ordinary algebras \cite{AKLY}, says that 'restriction up to equivalence' and 'restriction' are equivalent concepts for recollements in some interesting cases.

\begin{proposition} \label{prop.restriction}
 Let $$
\begin{xymatrix}{\hat{\mathcal{Y}} \ar[r]^{\hat{i}_*}& \hat{\mathcal{D}} \ar@<3ex>[l]_{\hat{i}^!}\ar@<-3ex>[l]_{\hat{i}^*}\ar[r]^{\hat{j}^*} & \hat{\mathcal{X}} \ar@<3ex>_{\hat{j}_*}[l]\ar@<-3ex>_{\hat{j}_!}[l]}
\end{xymatrix}$$ be a recollement of compactly generated triangulated categories that is the upper part of a ladder of recollements of height two,  let $\star\in\{\emptyset, b,-,+\}$ and $\dagger\in\{\emptyset, fl\}$, and suppose that $\mathcal{E}^c\subseteq\mathcal{E}^\star_\dagger$, for  $\mathcal{E}=\hat{\mathcal{X}},\hat{\mathcal{D}},\hat{\mathcal{Y}}$. The following assertions are equivalent:

\begin{enumerate}
\item  The recollement restricts to $(-)^\star_\dagger$-level.
\item  The recollement restricts, up to equivalence, to $(-)^\star_\dagger$-level, i.e. the associated TTF triple $(\text{Im}(\hat{j}_!),(\text{Im}(\hat{i}_*),(\text{Im}(\hat{j}_*))$ restricts to $\hat{\mathcal{D}}^\star_\dagger$
\item $\hat{i}^*(\hat{\mathcal{D}}^\star_\dagger)\subseteq\hat{\mathcal{Y}}^\star_\dagger$.
\item  $\hat{j}_!(\hat{\mathcal{X}}^\star_\dagger)\subseteq\hat{\mathcal{D}}^\star_\dagger$
\end{enumerate}
\end{proposition}
\begin{proof}
The fact that the recollement is the upper part of a ladder of recollements of height two implies that the functors $\hat{i}^*,\hat{i}_*,\hat{j}_!,\hat{j}^*$ preserve compact objects, since their right adjoints preserve coproducts. Note that if in an adjoint pair of triangulated functors $(F:\mathcal{D}\longrightarrow\mathcal{E},G:\mathcal{E}\longrightarrow\mathcal{D})$   between compactly generated triangulated categories the functor $F$ preserves compact objects, then $G(\mathcal{E}^\star_\dagger)\subseteq\mathcal{D}^\star_\dagger$. Indeed, let  $E\in\mathcal{E}^\star_\dagger$ be any object. Then $G(E)\in\mathcal{D}^\star_\dagger$ if and only if $G(E)$ satisfies property $\mathbf{P}_\dagger^*$ for all $X\in\mathcal{D}^c$. Due to the natural isomorphism $\text{Hom}_\mathcal{E}(F(X),E[n])\cong\text{Hom}_\mathcal{D}(X,G(E)[n])$, for all $X\in\mathcal{D}^c$ and $n\in\mathbb{Z}$, and since $F(X)\in\mathcal{E}^c$ we get that $G(E)$ satisfies  property $\mathbf{P}_\dagger^*$ for all $X\in\mathcal{D}^c$, since  $E$ satisfies property  $\mathbf{P}_\dagger^*$, for all $Z\in\mathcal{E}^c$. This implies  that all the functors $\hat{i}_*$, $\hat{i}^!$, $\hat{j}^*$ and $\hat{j}_*$ restrict to the $(-)^\star_\dagger$-level. In particular, the triangle  $\hat{i}_*\hat{i}^!D\longrightarrow D\longrightarrow\hat{j}_*\hat{j}^*D\stackrel{+}{\longrightarrow}$, which is the truncation triangle with respect to $(\mathcal{V},\mathcal{W}):=(\text{Im}(\hat{i}_*),\text{Im}(\hat{j}_*))$, belongs to $\hat{\mathcal{D}}^\star_\dagger$ for each $D\in\hat{\mathcal{D}}^\star_\dagger$. Therefore the semi-orthogonal decomposition $(\mathcal{V},\mathcal{W})$ restricts to $\hat{\mathcal{D}}^\star_\dagger$. As a consequence,  assertion 2 holds if and only if the triangle  $\hat{j}_!\hat{j}^*D\longrightarrow D\longrightarrow\hat{i}_*\hat{i}^*D\stackrel{+}{\longrightarrow}$ (I), which is the truncation triangle with respect to  $(\mathcal{U},\mathcal{V}):=(\text{Im}(\hat{j}_!),\text{Im}(\hat{i}_*))$, belongs to $\hat{\mathcal{D}}^\star_\dagger$ for each $D\in\hat{\mathcal{D}}^\star_\dagger$.

Since $\hat{i}_*$ and $\hat{j}_*$ are fully faithful,  the counits of the adjoint pairs $(\hat{i}^*,\hat{i}_*)$ and $(\hat{j}^*,\hat{j}_*)$  are natural isomorphisms, and, hence $\hat{\mathcal{Y}}^\star_\dagger\subseteq\hat{i}^*(\hat{\mathcal{D}}^\star_\dagger)$ and  $\hat{j}^*(\hat{\mathcal{D}}^\star_\dagger)=\hat{\mathcal{X}}^\star_\dagger$. Using this, we get that the triangle (I) belongs to $\hat{\mathcal{D}}^\star_\dagger$, for each $D\in\hat{\mathcal{D}}^\star_\dagger$ if and only if $\hat{j}_!(\hat{\mathcal{X}}^\star_\dagger)\subseteq\hat{\mathcal{D}}^\star_\dagger$. Thus, assertions 2 and 4 are equivalent.

  On the other hand, we have  $\hat{i}^*(\hat{\mathcal{D}}^c)=\hat{\mathcal{Y}}^c$. Indeed, we only need to check the inclusion $\supseteq$, since $\hat{i}^*$ preserves compact objects. If $Y\in\hat{\mathcal{Y}}^c$ then  $Y\cong \hat{i}^*\hat{i}_*(Y)$, and  $\hat{i}_*Y\in\hat{\mathcal{D}}^c$. The equality  $\hat{i}^*(\hat{\mathcal{D}}^c)=\hat{\mathcal{Y}}^c$ implies, that for $Y\in\hat{\mathcal{Y}}$, one has that $\hat{i}_*Y\in\hat{\mathcal{D}}^\star_\dagger$ if and only if $Y\in\hat{\mathcal{Y}}^\star_\dagger$. Indeed  $\hat{i}_*Y\in\hat{\mathcal{D}}^\star_\dagger$ if and only if $\hat{i}_*Y$ satisfies property $\mathbf{P}^\star_\dagger$, for all $D_0\in\hat{\mathcal{D}}^c$, which, by adjunction, is equivalent to saying that $Y$ satisfies property $\mathbf{P}^\star_\dagger$ for all $\hat{i}^*(D_0)$, with $D_0\in\hat{\mathcal{D}}^c$. That is, if and only if $Y$ satisfies property $\mathbf{P}^\star_\dagger$, for all $Y_0\in\hat{\mathcal{Y}}^c$, if and only if $Y\in\hat{\mathcal{Y}}^\star_\dagger$. Thus,  triangle (I) belongs to $\hat{\mathcal{D}}^\star_\dagger$, for each $D\in\hat{\mathcal{D}}^\star_\dagger$, if and only if $\hat{i}_*\hat{i}^*D\in\hat{\mathcal{D}}^\star_\dagger$ if and only if $\hat{i}^*D\in\hat{\mathcal{Y}}^\star_\dagger$, if and only if $\hat{i}^*(\hat{\mathcal{D}}^\star_\dagger)\subseteq\hat{\mathcal{Y}}^\star_\dagger$. Therefore assertions 2 and 3 are also equivalent.
 
 Finally, taking into account the first paragraph of this proof, it is clear that if the equivalent assertions 3 and 4 hold, then assertion 1 holds. 
\end{proof}

\begin{theorem} \label{thm.Case Y homologically loc.bounded}
Let  $\hat{\mathcal{D}}$ and $\hat{\mathcal{X}}$ be compactly generated algebraic triangulated categories  and let  $\mathcal{B}$ be a small $K$-linear category.
 For $\star\in\{\emptyset ,+,-,b\}$ and $\dagger\in\{\emptyset, fl\}$ let $$
\begin{xymatrix}{\mathcal{D}^{\star}_\dagger (\mathcal{B}) \ar[r]^-{{i}_*}& \hat{\mathcal{D}}^{\star}_\dagger \ar@<3ex>[l]_-{{i}^!}\ar@<-3ex>[l]_-{{i}^*}\ar[r]^{j^*} & \hat{\mathcal{X}}^{\star}_\dagger \ar@<3ex>_{{j}_*}[l]\ar@<-3ex>_{{j}_!}[l]}
\end{xymatrix}$$ be a recollement, such that the categories  involved contain the respective subcategories of compact objects and such that the functors $j_!,j^*,i^*,i_*$ preserve compact objects (that is $j_!(\hat{\mathcal{X}}^c)\subseteq \hat{\mathcal{D}}^c$ and similarly for $j^*, i^*$ and $i_*$). Then $\text{Im}(i_*)$ cogenerates $\mathcal{V}=\text{Loc}_{\hat{\mathcal{D}}}(i_*(\mathcal{D}^c(\mathcal{B})))$ 
and  assertion 2 of Theorem \ref{thm.lifting of recollements} holds. 

If in addition $\hat{\mathcal{X}}=\mathcal{D}(\mathcal{C})$, for some small $K$-linear category $\mathcal{C}$, then the given recollement lifts up to equivalence to a recollement $$
\begin{xymatrix}{\mathcal{D}(\mathcal{B}) \ar[r]^-{\hat{i}_*}& \hat{\mathcal{D}} \ar@<3ex>[l]_-{\hat{i}^!}\ar@<-3ex>[l]_-{\hat{i}^*}\ar[r]^{\hat{j}^*} & \mathcal{D}(\mathcal{C}) \ar@<3ex>_{\hat{j}_*}[l]\ar@<-3ex>_{\hat{j}_!}[l]}
\end{xymatrix},$$ which is the upper part of a ladder of recollements of height two.
\end{theorem}
\begin{proof}
We are going to use the notation from the proof of Theorem \ref{thm.lifting of recollements}. To prove the first assertion, we are going to 
check that  $\mathcal{D}(\mathcal{B}) \stackrel{\cong}{\longrightarrow}\mathcal{V}$ and that this equivalence restricts to 
$\mathcal{D}^{\star}_\dagger (\mathcal{B})\stackrel{\cong}{\longrightarrow} \text{Im}(i_*)$. Note that $\mathcal{B}$ is homologically locally bounded considered
 as a dg category. When $\dagger =fl$, due to the inclusion $\mathcal{D}^c(\mathcal{B})\subseteq\mathcal{D}^{\star}_\dagger(\mathcal{B})$, the dg 
category $\mathcal{B}$ is also homologically locally finite length.  
By Lemma \ref{LemmaNew}, $(\mathcal{U},\mathcal{V},\mathcal{W})=(\text{Loc}_{\hat{\mathcal{D}}}(j_!(\hat{\mathcal{X}}^c)),\text{Loc}_{\hat{\mathcal{D}}}(i_*(\mathcal{D}^c(\mathcal{B}))),\text{Loc}_{\hat{\mathcal{D}}}(i_*(\mathcal{D}^c(\mathcal{B})))^\perp )$ is a TTF triple in $\hat{\mathcal{D}}$ and ${\mathcal{V}}^c=\mathcal{V}\cap\hat{\mathcal{D}}^c$.  By the proof of the implication $3)\Longrightarrow 2)$ of Theorem \ref{thm.lifting of recollements}, $\text{Im}(i_*)=\mathcal{V}\cap\hat{\mathcal{D}}^{\star}_\dagger$. Indeed, the additional condition on $\text{Im}(i_*)$  was used only to check that $\text{Im}(j_!)=\mathcal{U}\cap\hat{\mathcal{D}}^{\star}_\dagger$.

By hypothesis, we have $i_*(\mathcal{D}^c(\mathcal{B}))\subseteq\mathcal{V}\cap\hat{\mathcal{D}}^c$. For $V\in\mathcal{V}\cap\hat{\mathcal{D}}^c \subseteq \text{Im}(i_*)$ take $W$ such that $i_*W\simeq V$, since $i^*$ preserves compact objects $W\simeq i^*i_*W$ is compact and $V\in i_*(\mathcal{D}^c(\mathcal{B}))$. So ${\mathcal{V}}^c=\mathcal{V}\cap\hat{\mathcal{D}}^c=i_*(\mathcal{D}^c(\mathcal{B}))$. Hence, also $i_*i^*\hat{\mathcal{D}}^c=i_*\mathcal{D}^c(\mathcal{B}).$

Note that $\mathcal{V}$ is a quotient of an algebraic compactly generated triangulated category by a localizing subcategory generated by a set of compact objects. Then  
$\mathcal{V}$ is  compactly generated  by \cite[Theorem 2.1]{N2} (using the description of compact objects and the right adjoint to the localization functor),  and it is also algebraic as a triangulated subcategory of an algebraic triangulated category. Furthermore,  by \cite[Theorem 9.2]{K} there is a triangulated equivalence 
$F:\mathcal{D}(\mathcal{B})\stackrel{\cong}{\longrightarrow}\mathcal{V}$ such that $F(B^\wedge)\cong i_*(B^\wedge) $, for each 
$B\in\mathcal{B}$. Since $\mathcal{D}^c(\mathcal{B})=\text{thick}_{\mathcal{D}(\mathcal{B})}(B^\wedge\text{: }B\in\mathcal{B})$, 
we get $F(\mathcal{D}^c(\mathcal{B}))=  i_*(\mathcal{D}^c(\mathcal{B})) $.

Consider the canonical triangle  $j_!j^*Z\longrightarrow Z\longrightarrow i_*i^*Z\stackrel{+}{\longrightarrow}$, for each
 $Z\in\hat{\mathcal{D}}^c$. Since $j_!j^*(Z)$ is compact $\text{Hom}_{\hat{\mathcal{D}}}(j_!j^*(Z),-)$ vanishes on 
$\text{Im}(F)=\mathcal{V}$. For $Y\in\mathcal{D}(\mathcal{B})$ we get that $F(Y)$ belongs to $\hat{\mathcal{D}}^*_\dagger$  
iff it satisfies property $\mathbf{P}^\star_\dagger$, for all $Z\in\hat{\mathcal{D}}^c$, iff it satisfies property $\mathbf{P}^\star_\dagger$, 
for all $Z'\in\hat{\mathcal{D}}^c$ such that $Z'\cong i_*i^*(Z)$. Since $i_*i^*\hat{\mathcal{D}}^c=i_*\mathcal{D}^c(\mathcal{B})$, we get  
that $F(Y)$ is in $\hat{\mathcal{D}}^*_\dagger$  iff it satisfies property $\mathbf{P}^{\star}_\dagger$, for all $Z=i_*(M)\cong F(M')$, with
$M, M'\in\mathcal{D}^c(\mathcal{B})$.  The fact that $F$ is an equivalence implies that 
$F(Y)\in\hat{\mathcal{D}}^*_\dagger$ iff $Y$ satisfies property 
$\mathbf{P}^\star_\dagger$ in $\mathcal{D}(\mathcal{B})$ for all $M'\in \mathcal{D}^c (\mathcal{B})$. 
Thus, $F(Y)\in\hat{\mathcal{D}}^*_\dagger$ iff $Y\in\mathcal{D}^*_\dagger\mathcal{B}$. This means that 
$F(\mathcal{D}^*_\dagger(\mathcal{B}))=\hat{\mathcal{D}}^*_\dagger \cap\text{Im}(F)=\hat{\mathcal{D}}^*_\dagger\cap\mathcal{V}=
\text{Im}(i_*)$. Therefore $F$ induces an equivalence of triangulated categories $F:\mathcal{D}^*_\dagger (\mathcal{B})\stackrel{\cong}{\longrightarrow}\text{Im}(i_*)$. Note that, this equivalence need not be naturally isomorphic to the one induced by $i_*$. By the proof of Corollary \ref{cor.Case D homologically loc. bounded}, $\mathcal{D}^*_\dagger$ cogenerates $\mathcal{D}(\mathcal{B})$, and hence $\text{Im}(i_*)$ cogenerates $\mathcal{V}$.

Let us prove the second assertion of the proposition. From the TTF triple constructed above we get a recollement:
\begin{equation}\label{Recf}
\begin{xymatrix}{\mathcal{D}(\mathcal{B}) \ar[r]^-{\hat{i}_*}& \hat{\mathcal{D}} \ar@<3ex>[l]_-{\hat{i}^!}\ar@<-3ex>[l]_-{\hat{i}^*}\ar[r]^{\hat{j}^*} & \mathcal{W} \ar@<3ex>_{\hat{j}_*}[l]\ar@<-3ex>_{\hat{j}_!}[l]}
\end{xymatrix},
\end{equation}
where $\hat{j}_*:\mathcal{W}\hookrightarrow\hat{\mathcal{D}}$ is the inclusion, such that the associated TTF triple $(\text{Im}(\hat{j}_!),\text{Im}(\hat{i}_*),\text{Im}(\hat{j}_*))$ in $\hat{\mathcal{D}}$ restricts to the TTF triple $(\text{Im}(j_!),\text{Im}(i_*),\text{Im}(j_*))$ in $\hat{\mathcal{D}}^*_\dagger$. In particular, there is an equivalence of categories $\mathcal{U}\stackrel{\cong}{\longrightarrow}\mathcal{W}$ which restricts to the canonical equivalence $\text{Im}(j_!)\stackrel{\cong}{\longrightarrow}\text{Im}(j_*)$ given by $j_*j_!^{-1}$. As a quotient of an algebraic triangulated category $\mathcal{W}$ is algebraic. Since $\mathcal{W} \stackrel{\cong}{\longrightarrow}\mathcal{U}=\text{Loc}_{\hat{\mathcal{D}}}(j_!(\mathcal{D}^c(\mathcal{C})))=\text{Loc}_{ \hat{\mathcal{D}}}(j_!(C^\wedge)\text{: }C\in\mathcal{C})$, 
we conclude that $\mathcal{W}$ is compactly generated by $\{j_*(C^\wedge)\text{: }C\in\mathcal{C}\}$.

As before, by \cite[Theorem 9.2]{K}, there is a triangulated equivalence
 $G:\mathcal{D}(\mathcal{C})\stackrel{\cong}{\longrightarrow}\mathcal{W}$ such that $G(C^\wedge)\cong j_*(C^\wedge)$, for each 
$C\in\mathcal{C}$. Let  $X\in\mathcal{D}(\mathcal{C})$ be any object. We claim that 
$G(X)\in\text{Im}(j_*)=\mathcal{W}\cap\hat{\mathcal{D}}^*_\dagger$ iff $X\in\mathcal{D}_\dagger^*(\mathcal{C})$. Indeed,  
$G(X)\in\hat{\mathcal{D}}^*_\dagger$ iff $\text{Hom}_{\hat{\mathcal{D}}}(Z,G(X)[k])$ satisfies property $\mathbf{P}^{\star}_\dagger$, 
for all $Z\in\hat{\mathcal{D}}^c$. Using the triangle $i_*i^!Z\longrightarrow Z\longrightarrow j_*j^*Z\stackrel{+}{\longrightarrow}$ and the fact that $\text{Hom}_{\hat{\mathcal{D}}}(i_*i^!Z,-)$ vanishes on $\mathcal{W}=\text{Im}(G)$, we get that $G(X)$ is in $\hat{\mathcal{D}}^*_\dagger$ iff $\text{Hom}_{\hat{\mathcal{D}}}(j_*j^*Z,G(X)[k])$ satisfies 
$\mathbf{P}^{\star}_\dagger$, for all $Z\in\hat{\mathcal{D}}^c$. By hypothesis, we have an inclusion $j^*(\hat{\mathcal{D}}^c)\subseteq\mathcal{D}^c(\mathcal{C})$. Conversely, if $X'\in\mathcal{D}^c(\mathcal{C})$ then $j_!(X')\in\hat{\mathcal{D}}^c$, so that $X'\cong j^*j_!(X')\in j^*(\hat{\mathcal{D}}^c)$. Thus, when $Z$ runs through the objects of $\hat{\mathcal{D}}^c$, the object $j^*Z$ runs through the objects of $\mathcal{D}^c(\mathcal{C})$. Since $\mathcal{D}^c(\mathcal{C})=\text{thick}_{\mathcal{D}(\mathcal{C})}(C^\wedge\text{: }C\in\mathcal{C})$, we easily conclude that $G(X)\in\hat{\mathcal{D}}^*_\dagger$ iff $\text{Hom}_{\hat{\mathcal{D}}}(j_*(C^\wedge ),G(X)[k])$ 
satisfies $\mathbf{P}^{\star}_\dagger$, for all $C\in\mathcal{C}$. Since $G$ is an equivalence of categories and $j_*(C^\wedge )\cong G(C^\wedge )$, we have $G(X)\in\hat{\mathcal{D}}^*_\dagger$ iff $\text{Hom}_{\mathcal{D}(\mathcal{C})}(C^\wedge ,X[k])$ 
satisfies $\mathbf{P}^{\star}_\dagger$, for all $C\in\mathcal{C}$. That is, iff $X\in\mathcal{D}^*_\dagger(\mathcal{C})$. 

The previous paragraph yields an equivalence of categories $G:\mathcal{D}(\mathcal{C})\stackrel{\cong}{\longrightarrow}\mathcal{W}$ which 
induces by restriction another equivalence $\mathcal{D}^*_\dagger(\mathcal{C})\stackrel{\cong}{\longrightarrow}\text{Im}(j_*)$. This implies that we can 
replace $\mathcal{W}$ by $\mathcal{D}(\mathcal{C})$ in the recollement (\ref{Recf}), thus obtaining a recollement as in the final assertion of the
 proposition, which in turn restricts to a recollement    whose associated TTF triple is $(\text{Im}(j_!),\text{Im}(i_*),\text{Im}(j_*))$. This last recollement is then equivalent 
to the original one. 

It remains to prove that the obtained recollement $$
\begin{xymatrix}{\mathcal{D}(\mathcal{B}) \ar[r]^-{\hat{i}_*}& \hat{\mathcal{D}} \ar@<3ex>[l]_-{\hat{i}^!}\ar@<-3ex>[l]_-{\hat{i}^*}\ar[r]^-{\hat{j}^*} & \mathcal{D}(\mathcal{C}) \ar@<3ex>_-{\hat{j}_*}[l]\ar@<-3ex>_-{\hat{j}_!}[l]}
\end{xymatrix}$$ is the upper part of a ladder of recollements of height two.  This is a direct consequence of \cite[Proposition 3.4]{GP} since $\hat{i}_*$ preserves compact objects.
\end{proof}

\begin{remark} \label{rem.recollement usual derived categories}
 When $\mathcal{B}$ is a $K$-linear category,  the assumption $\mathcal{D}^c(\mathcal{B})\subseteq\mathcal{D}^*_\dagger(\mathcal{B})$ always holds when $\dagger=\emptyset$. 
When $\dagger =fl$ the assumption holds iff $\mathcal{B}(B,B')$ is a $K$-module of finite length, for all $B,B'\in\mathcal{B}$. 
In particular when $\mathcal{B}=B$ is an ordinary algebra, the inclusion
 $\mathcal{D}^c(\mathcal{B})\subseteq\mathcal{D}^*_{fl}(\mathcal{B})$ holds iff $B$ is  finite length. 
\end{remark}

For our next result, we shall use the following concept.

\begin{definition} \label{def.compact-detectable}
A compactly generated triangulated category $\mathcal{E}$ will be called  \emph{compact-detectable in finite length} when $\mathcal{E}^c$ consists of the objects $X\in\mathcal{E}^b_{fl}$ such that  $\text{Hom}_\mathcal{E}(X,E[k])=0 \mbox{ for } E\in\mathcal{E}^b_{fl}\mbox{ and }k\gg 0$. Note that such a category is homologically locally finite length. 
\end{definition}

\begin{example} \label{exs.intrinsic description of compacts}
The following triangulated $K$-categories $\hat{\mathcal{D}}$ are compact-detectable in finite length
 and have the property that $\hat{D}^b_{fl}$ is Hom-finite (i.e. $\text{Hom}_{\hat{\mathcal{D}}}(M,N)$  is a  $K$-module of finite length for any $M$, $N\in \hat{D}^b_{fl}$):

\begin{enumerate}
\item  $\hat{\mathcal{D}}=\mathcal{D}(\text{Qcoh}(\mathbb{X}))$, for a  projective  scheme  $\mathbb{X}$  over a perfect field $K$ \cite{Rou}.
\item $\hat{\mathcal{D}}=\mathcal{D}(A)$, where $A$ is a homologically non-positive homologically finite length dg $K$-algebra, 
where $K$ is any commutative ring.
\end{enumerate}
\end{example}
\begin{proof}
1) By \cite[Lemma 7.46]{Rou}, for an arbitrary projective scheme $\mathbb{X}$ over $K$, 
we have $\mathcal{D}^b(\text{coh}(\mathbb{X}))=\mathcal{D}(\text{Qcoh}(\mathbb{X}))^b_{fl}$,  which is well-known to be 
Hom-finite over $K$. 
By \cite[Lemma 7.49]{Rou}, $\mathcal{D}(\text{Qcoh}(\mathbb{X}))^c$ consists of $X\in\mathcal{D}(\text{Qcoh}(\mathbb{X}))$ such that, for each $M\in\mathcal{D}^b(\text{coh}(\mathbb{X}))$, the $K$-vector space $\oplus_{k\in\mathbb{Z}}\text{Hom}_{\mathcal{D}(\text{Coh}(\mathbb{X}))}(X,M[k])$ is finite dimensional. But if $X,M\in\mathcal{D}^b(\text{coh}(\mathbb{X}))$, then $\text{Hom}_{\mathcal{D}(\text{Qcoh}(\mathbb{X}))}(X,M[k])=0$ for $k\ll 0$. We also know that $\mathcal{D}^b(\text{coh}(\mathbb{X}))$ is $\text{Hom}$-finite. Thus, the subcategory consisting of $X\in\mathcal{D}^b(\text{coh}(\mathbb{X}))$ such that,  for each $M\in\mathcal{D}^b(\text{coh}(\mathbb{X}))$,  one has $\text{Hom}_{\mathcal{D}(\text{Qcoh}(\mathbb{X}))}(X,M[k])=0$ for $k\gg 0$, coincides with the subcategory of $X\in\mathcal{D}^b(\text{coh}(\mathbb{X}))$ such that $\oplus_{k\in\mathbb{Z}}\text{Hom}_{\mathcal{D}(\text{Qcoh}(\mathbb{X}))}(X,M[k])$ is finite dimensional. This subcategory 
is precisely $\mathcal{D}(\text{Qcoh}(\mathbb{X}))^c\cap\mathcal{D}^b(\text{coh}(\mathbb{X}))=\mathcal{D}(\text{Qcoh}(\mathbb{X}))^c$.

2) By \cite[Section 3.7]{NS2} or \cite[Example 6.1]{KN}, there is a \emph{canonical t-structure} $(\mathcal{D}^{\leq 0}A,\mathcal{D}^{\geq 0}A)$ in $\mathcal{D}(A)$, where $\mathcal{D}^{\leq 0}(A)$ (resp. $\mathcal{D}^{\geq 0}(A)$) consists of the dg $A$-modules $M$ such that $H^k(M)=0$, for all $k>0$ (resp. $k<0$). Moreover, $\mathcal{D}^{\leq 0}(A)=\text{Susp}_{\mathcal{D}(A)}(A)$. By Corollary \ref{cor.restriction of silting for dg categories} below for $\mathcal{T}=\{A\}$, this t-structure restricts to $\mathcal{D}^b_{fl}(A)$. 

Note that there is a dg subalgebra $\tilde{A}$ of $A$, given by $\tilde{A}^n=A^n$, for $n<0$, $\tilde{A}^0=Z^0(A)=\{0\text{-cycles of }A\}$, and $\tilde{A}^n=0$, for $n>0$. The inclusion $\lambda:\tilde{A}\hookrightarrow A$ is a quasi-isomorphism, and the associated restriction of scalars $\lambda_*:\mathcal{D}(A)\longrightarrow\mathcal{D}(\tilde{A})$ is a triangulated equivalence which takes $A$ to  $\tilde{A}$. As a consequence, this equivalence preserves the canonical t-structure, and hence it induces an equivalence between the corresponding hearts. The heart of $(\mathcal{D}^{\leq 0}(\tilde{A}),\mathcal{D}^{\geq 0}(\tilde{A}))$ is known to be equivalent to the category of modules over $H^0(\tilde{A})\cong H^0(A)$ (see \cite[Example 6.1]{KN}). This equivalence is given by $H^0:\mathcal{H}\stackrel{\cong}{\longrightarrow}\text{Mod-}H^0(A)$ ($M\rightsquigarrow H^0(M)$). Putting $\mathcal{H}_{fl}:=\mathcal{H}\cap\mathcal{D}^b_{fl}(A)$, which is the heart of the restricted t-structure $(\mathcal{D}^{\leq 0}(A)\cap\mathcal{D}^b_{fl}(A),\mathcal{D}^{\geq 0}(A)\cap\mathcal{D}^b_{fl}(A))$ in $\mathcal{D}^b_{fl}(A)$, we deduce an equivalence of categories $H^0:\mathcal{H}_{fl}\stackrel{\cong}{\longrightarrow}\text{mod-}H^0(A)$, bearing in mind that $H^0(A)$ is a finite length $K$-algebra.

Let now $X\in\mathcal{D}^b_{fl}(A)$ be any object. By the proof of \cite[Theorem 2]{NSZ} and the fact that 
$\text{Hom}_{\mathcal{D}(A)}(A[k],M)\cong H^{-k}M$  is a $K$-module of finite length, for each $M\in\mathcal{D}^b_{fl}(A)$ and each 
$k\in\mathbb{Z}$,  we know that $X$ 
is the Milnor colimit of a sequence 
$0=X_{-1}\stackrel{f_0}{\longrightarrow}X_0\stackrel{f_1}{\longrightarrow}X_1\longrightarrow\cdots\longrightarrow{X_n}\stackrel{f_n}\longrightarrow\cdots$ 
such that $\text{cone}(f_n)\in\text{add}(A)[n]$, for each $n\in\mathbb{N}$. Let $r>0$ be arbitrary and, for each $n> r$,  put 
$u_n:=f_n\circ \dots\circ f_{r+1}:X_{r}\longrightarrow X_n$ and $C_n=\text{cone}(u_n)$.  By Verdier's $3\times 3$ lemma 
(see \cite[Lemma 1.7]{M}), we have a commutative diagram, where all rows and columns are triangles 
$$\begin{xymatrix}@C=45pt{\coprod_{n>r}X_r \ar[r]^{\coprod_{n>r}u_n} \ar[d]^{1-\sigma}& \coprod_{n>r}X_n \ar[r] \ar[d]^{1-\sigma} & \coprod_{n>r}C_n \ar[d] \ar[r]&\\
\coprod_{n>r}X_r \ar[d]\ar[r]^{\coprod_{n>r}u_n} & \coprod_{n>r}X_n \ar[r] \ar[d] & \coprod_{n>r}C_n\ar[d]\ar[r]&\\
X_r\ar[r]& X\ar[r]& C\ar[r]&.}\end{xymatrix}$$ 
Since each $C_n$ is a finite iterated extension of objects in $\text{add}(A)[n]$, with $n>r$, 
we get $C_n\in\mathcal{D}^{<-r}(A)$, for each $n>r$. It follows that $C\in\mathcal{D}^{\leq -r}(A)$.
 But $C\in\mathcal{D}^b_{fl}(A)$, since the left two terms of the triangle in the bottom row of the diagram are in $\mathcal{D}^b_{fl}(A)$. 
We then get a triangle $X_r\longrightarrow X\longrightarrow C\stackrel{+}{\longrightarrow}$ in 
$\mathcal{D}^b_{fl}(A)$ such that $X_r\in\mathcal{D}^c(A)$ and $C\in\mathcal{D}^{\leq -r}(A)$. If now $Y\in\mathcal{D}^b_{fl}(A)$ is 
any object, then we know that $Y\in\mathcal{D}^{>-r}(A)$, for some integer $r>0$. For this integer, we then get a monomorphism 
$\text{Hom}_{\mathcal{D}(A)}(X,Y)\longrightarrow\text{Hom}_{\mathcal{D}(A)}(X_r,Y)$ whose target is a $K$-module of finite length. 
This proves
that $\mathcal{D}^b_{fl}(A)$ is Hom-finite. 

On the other hand, by \cite[Lemma 3.7]{Br}, we know that 
\begin{equation}\label{another}
\mathcal{D}^{\leq 0}(A)\cap\mathcal{D}^b_{fl}(A)=\bigcup_{n\in\mathbb{N}}\text{add}(\mathcal{H}_{fl}[n]\star\mathcal{H}_{fl}[n-1]\star\cdots\star\mathcal{H}_{fl}[0]).
\end{equation}
Let $X\in\mathcal{D}^b_{fl}(A)$ be an object such that, for each $M\in\mathcal{D}^b_{fl}(A)$, one has $\text{Hom}_{\mathcal{D}(A)}(X,M[k])=0$ for $k\gg 0$ (and, hence, also for $|k| \gg 0$). Let us choose a (necessarily finite) set $\mathcal{S}$ of representatives of the isoclasses of simple objects of $\mathcal{H}_{fl}$. If $m$ is the Loewy length of $H^0(A)$, then $\mathcal{H}_{fl}\subset\text{add}(\mathcal{S})\star\stackrel{m}{\dots}\star\text{add}(\mathcal{S})$. Fixing  $r\in\mathbb{N}$ such that $\text{Hom}_{\mathcal{D}(A)}(X,\mathcal{S}[k])=0$, for $k\geq r$, we get that $\text{Hom}_{\mathcal{D}(A)}(X,-[k])$ vanishes on $\mathcal{H}_{fl}$, for all $k\geq r$. By (\ref{another}) above,  $\text{Hom}_{\mathcal{D}(A)}(X,-)$ vanishes on $\mathcal{D}^{\leq -r}(A)\cap\mathcal{D}^b_{fl}(A)$. 
If for this $r$ we consider the triangle $X_r\longrightarrow X\longrightarrow C\stackrel{+}{\longrightarrow}$ constructed above, then
 the arrow $X\longrightarrow C$ is the zero map, and hence $X$ is isomorphic to a direct summand of $X_r$ and $X\in\mathcal{D}^c(A)$.
\end{proof}

\begin{remark} \label{rem.Neeman-approximability}
While preparing the manuscript we have learnt that Neeman has introduced the powerful tool of approximable triangulated categories. Using it, one can derive the compact-detectability in finite length for the categories from the last example, using the fact that they are approximable, with the equivalence class of the canonical t-structure as the preferred one (see \cite[Examples 3.3 and 3.6]{N3}). Although nontrivial, the only thing left to prove would be the fact that what is $\mathcal{T}^-_c$, in Neeman's terminology, coincides with $\hat{\mathcal{D}}^b_{fl}$ in our case. Once this is proved the compact-detectability in finite length follows from \cite[Theorem 0.3]{N4}. 
\end{remark}

\begin{proposition} \label{prop.homologically non-positive locally fd}
Let $\hat{\mathcal{Y}}$, $\hat{\mathcal{D}}$ and $\hat{\mathcal{X}}$ be compactly generated triangulated categories which are compact-detectable in finite length and let $$
\begin{xymatrix}{\hat{\mathcal{Y}}^{b}_{fl} \ar[r]^{{i}_*}& \hat{\mathcal{D}}^{b}_{fl} \ar@<3ex>[l]_{{i}^!}\ar@<-3ex>[l]_{{i}^*}\ar[r]^{{j}^*} & \hat{\mathcal{X}}^{b}_{fl} \ar@<3ex>_{{j}_*}[l]\ar@<-3ex>_{{j}_!}[l]}
\end{xymatrix}$$
be a recollement.   Then the functors $j_!$, $j^*$, $i^*$ and $i_*$ preserve compact objects. In particular,  the associated TTF triple $(\text{Im}(j_!),\text{Im}(i_*),\text{Im}(j_*))$ in $\hat{\mathcal{D}}^b_{fl}$ lifts to a TTF triple $(\mathcal{U},\mathcal{V},\mathcal{W})$ in $\hat{\mathcal{D}}$ such that the torsion pairs $(\mathcal{U},\mathcal{V})$ and $(\mathcal{V},\mathcal{W})$ are compactly generated and $j_!(\hat{\mathcal{X}}^c)=\mathcal{U}\cap\hat{\mathcal{D}}^c$ and $i_*(\hat{\mathcal{Y}}^c)=\mathcal{V}\cap\hat{\mathcal{D}}^c$. 

In the particular case when $\hat{\mathcal{Y}}=\mathcal{D}(B)$   and  $\hat{\mathcal{X}}=\mathcal{D}(C)$, for ordinary finite length $K$-algebras $B$ and $C$ (and hence $\hat{\mathcal{Y}}^b_{fl}\cong\mathcal{D}^b(\text{mod-}B)$ and  $\hat{\mathcal{X}}^b_{fl}\cong\mathcal{D}^b(\text{mod-}C)$), and $\hat{\mathcal{D}}$ is algebraic,  the given recollement lifts, up to equivalence, to a recollement $$
\begin{xymatrix}{\mathcal{D}(B) \ar[r]^-{\hat{i}_*}& \hat{\mathcal{D}}\ar@<3ex>[l]_-{\hat{i}^!}\ar@<-3ex>[l]_-{\hat{i}^*}\ar[r]^{\hat{j}^*} & \mathcal{D}(C), \ar@<3ex>_{\hat{j}_*}[l]\ar@<-3ex>_{\hat{j}_!}[l]} 
\end{xymatrix}$$ which is the upper part of a ladder of recollements of height two. 
\end{proposition}
\begin{proof}
It is clear that if $\mathcal{D}$ and $\mathcal{E}$ are triangulated categories which are compact-detectable in finite length and $F:\mathcal{D}^b_{fl}\longrightarrow\mathcal{E}^b_{fl}$ is a functor that has a right adjoint, then $F$ preserves compact objects. Therefore $j_!$, $j^*$, $i^*$ and $i_*$ preserve compact objects. Now Corollary \ref{cor.Case D homologically loc. bounded} says that 
assertion 2 of Theorem \ref{thm.lifting of recollements} holds. The last assertion of the proposition  is a direct consequence of the last 
assertion of Theorem \ref{thm.Case Y homologically loc.bounded}.
\end{proof}

\begin{remark} \label{rem.Artin}
Last proposition applies to any recollement 
$$
\begin{xymatrix}
{D^b(\text{mod-}B) \ar[r]^{i_*} & D^b(\text{mod-}A) \ar@<3ex>[l]_{i^!} \ar@<-3ex>[l]_{i^*} \ar[r]^{j^*} & D^b(\text{mod-}C), \ar@<3ex>_{j_*}[l] \ar@<-3ex>_{j_!}[l]}
\end{xymatrix}$$ where $A$, $B$ and $C$ are finite length algebras. 
\end{remark}

\section{Partial silting sets} \label{sect.partial silting sets}

Recall that a silting set  in  a triangulated category $\mathcal{D}$ is a non-positive set $\mathcal{T}$ such that 
$\text{thick}_\mathcal{D}(\mathcal{T})=\mathcal{D}$ (see \cite{AI}). In this paper, we will call a silting set with this property a \emph{classical silting set}. In \cite{NSZ} and \cite{PV} the authors introduced  the notion of a silting set in any triangulated category with coproducts. We  take the following definition, {given in \cite{NSZ} for triangulated categories with coproducts},  and  consider it in an arbitrary triangulated category $\mathcal{D}$.

\begin{definition} \label{def.partial silting object}
Let $\mathcal{D}$ be a triangulated category.  A set of objects $\mathcal{T}$ in $\mathcal{D}$ will be called \emph{partial silting}   when the following
conditions hold:
\begin{enumerate}
\item The t-structure generated by $\mathcal{T}$ exists in $\mathcal{D}$;
\item $\text{Hom}_\mathcal{D}(T,-[1])$ vanishes on the aisle of that t-structure, for all $T\in\mathcal{T}$.
\end{enumerate}
Such a set will be called a \emph{silting set}
when it generates $\mathcal{D}$.

A t-structure $(\mathcal{D}^{\leq 0},\mathcal{D}^{\geq 0})$ in $\mathcal{D}$ is called a \emph{(partial) silting t-structure} when it is generated by a (partial) silting set. 
\end{definition}

\begin{remark}
Note that this definition of a silting set from \cite{NSZ} coincides with the one from \cite{PV}, i.e. a set of objects $\mathcal{T}$ such that $(\mathcal{T}^{\perp_{>0}},\mathcal{T}^{\perp_{<0}})$ is a t-structure. Furthermore, 'silting set in $\mathcal{D}$ consisting of compact objects' and  'classical silting set in $\mathcal{D}^c$' are the same.  
\end{remark}

\begin{example}\label{ex.NSZ} 
\begin{enumerate}
\item[a)] If $\mathcal{D}$ has coproducts, then any non-positive set of compact
objects is partial silting (see \cite[Example 2(1)]{NSZ}).
\item[b)]   Let $A$ be an ordinary algebra and let $\mathcal{K}^b(\text{Proj-}A)$ denote the bounded homotopy category of complexes of projective modules. A complex $P^\bullet\in\mathcal{K}^b(\text{Proj-}A)$ is called a semi-tilting complex in 
\cite{Wei} if $\text{Hom}_{\mathcal{D}(A)}(P^\bullet ,P^{\bullet (I)}[k])=0$, for all sets $I$  and all integers $k>0$, and $\text{thick}_{\mathcal{D}(A)}(\text{Add}(P^\bullet ))=\mathcal{K}^b(\text{Proj-}A)$. In such a case $\mathcal{T}=\{P^\bullet\}$ is a silting  set in $\mathcal{D}(A)$ (see \cite[Example 2(2)]{NSZ}). 
\end{enumerate}
\end{example}

The following  gives a good source of examples of partial silting sets.

\begin{proposition} \label{prop.restricted partial silting}
Let $\mathcal{D}$ be a thick subcategory of a triangulated
category $\mathcal{E}$ and let $\mathcal{T}\subset\mathcal{D}$ be a
set of objects. If $\mathcal{T}$ is partial silting in $\mathcal{E}$ and the associated t-structure $\tau$ in $\mathcal{E}$ restricts to $\mathcal{D}$, then
 $\mathcal{T}$ is partial silting in $\mathcal{D}$. Moreover, when $\mathcal{E}$ has coproducts  and $\mathcal{T}$ is a silting set in 
$\mathcal{E}$ consisting of compact objects  that is partial silting  in $\mathcal{D}$, then $\tau$ restricts to the $t$-structure generated by  $\mathcal{T}$ in $\mathcal{D}$. 
\end{proposition}  
\begin{proof}
Assume $\mathcal{T}$ is partial silting in $\mathcal{E}$ and the associated t-structure $\tau$ in $\mathcal{E}$ restricts to $\mathcal{D}$. The restricted t-structure in $\mathcal{D}$ is $( {}^\perp
(\mathcal{T}^{\perp_{\leq
0}})\cap\mathcal{D},\mathcal{T}^{\perp_{<0}}\cap\mathcal{D})$, where
the orthogonals are taken in $\mathcal{E}$. Since
$\text{Hom}_\mathcal{E}(T,-)$ vanishes on $( {}^\perp \mathcal{T}^{\perp_{\leq 0}})[1]$, it vanishes on $( {}^\perp
(\mathcal{T}^{\perp_{\leq 0}})\cap\mathcal{D})[1]$. It remains to see that the
restricted t-structure is generated by $\mathcal{T}$ in
$\mathcal{D}$. That is, that $( {}^\perp
(\mathcal{T}^{\perp_{\leq 0}})\cap\mathcal{D},\mathcal{T}^{\perp_{<0}}\cap\mathcal{D})=(
{}^\perp(\mathcal{T}^{\perp_{\leq 0}}\cap\mathcal{D})\cap\mathcal{D},\mathcal{T}^{\perp_{<0}}\cap\mathcal{D})$.
Right parts of these pairs coincide and we clearly have the inclusion $\subseteq$ on the left parts. If $X\in
 {}^\perp(\mathcal{T}^{\perp_{\leq 0}}\cap\mathcal{D})\cap\mathcal{D}$ and
$U\stackrel{f}{\longrightarrow}
X\stackrel{g}{\longrightarrow}V\stackrel{+}{\longrightarrow}$ is the
truncation triangle with respect to the restricted t-structure, then
$g=0$ and hence $X$ is a direct summand of
$U\in  {}^\perp (\mathcal{T}^{\perp_{\leq 0}})\cap\mathcal{D}$. This
implies that $X$ belongs to 
 $ {}^\perp (\mathcal{T}^{\perp_{\leq
0}})\cap\mathcal{D}$ since this class is closed under direct
summands.

For the second part of the statement note that by \cite[Theorem 1]{NSZ}, the set $\mathcal{T}$ is partial silting in $\mathcal{E}$ and  
${}^\perp(\mathcal{T}^{\perp_{\leq 0}})=\mathcal{T}^{\perp_{>0}}$. On the other hand, 
the t-structure in $\mathcal{D}$ generated by $\mathcal{T}$ is 
$\tau'=({}^\perp (\mathcal{T}^{\perp_{\leq 0}}\cap\mathcal{D})\cap\mathcal{D},\mathcal{T}^{\perp_{<0}}\cap\mathcal{D})$, and 
the partial silting condition of $\mathcal{T}$ in $\mathcal{D}$ gives  
${}^\perp (\mathcal{T}^{\perp_{\leq 0}}\cap\mathcal{D})\cap\mathcal{D}\subseteq\mathcal{T}^{\perp_{>0}}$, so  
${}^\perp (\mathcal{T}^{\perp_{\leq 0}}\cap\mathcal{D})\cap\mathcal{D}\subseteq\mathcal{T}^{\perp_{>0}}\cap\mathcal{D}$. Thus, we get a
chain of inclusions ${}^\perp(\mathcal{T}^{\perp_{\leq 0}})\cap\mathcal{D}\subseteq {}^\perp (\mathcal{T}^{\perp_{\leq 0}}\cap\mathcal{D})\cap\mathcal{D}
\subseteq\mathcal{T}^{\perp_{>0}}\cap\mathcal{D}$, all of which must be  equalities. Hence, 
$\tau'=(\mathcal{T}^{\perp_{>0}}\cap\mathcal{D},\mathcal{T}^{\perp_{<0}}\cap\mathcal{D})$ is the restriction of $\tau$ to $\mathcal{D}$.
\end{proof}

We now address the question on the uniqueness of the partial silting set which generates a given partial silting t-structure. The following is a consequence of the results in  \cite[Section 4]{NSZ}. 

\begin{proposition} 
Let $\mathcal{D}$ be a triangulated category with coproducts. If $(\mathcal{D}^{\leq 0},\mathcal{D}^{\geq 0})$ is a partial silting 
t-structure in $\mathcal{D}$, then the partial silting set which generates the t-structure is uniquely determined up to $Add$-equivalence. 
\end{proposition}

When $\mathcal{D}$ is a subcategory of a category with coproducts and the t-structure is generated by a partial silting set of compact objects, we still have a 
certain kind of uniqueness, as the following result shows. 

\begin{proposition} \label{prop.compact-psilting-from coheart}
Let $\mathcal{D}$ be  a thick subcategory of a triangulated category with coproducts  $\hat{\mathcal{D}}$ such that 
$\hat{\mathcal{D}}^c\subseteq\mathcal{D}$ and $\hat{\mathcal{D}}^c$ is skeletally small.
Suppose that $(\mathcal{D}^{\leq 0},\mathcal{D}^{\geq 0})$ is a t-structure in $\mathcal{D}$ 
generated by a partial silting set  which consists of compact objects in $\hat{\mathcal{D}}$. There is a non-positive set 
$\mathcal{T}\subseteq\hat{\mathcal{D}}^c$, uniquely determined up to $\text{add}$-equivalence,  
such that the following two conditions hold:

\begin{enumerate}
\item[a)] $\mathcal{T}$ is partial silting in $\mathcal{D}$ and  it generates $(\mathcal{D}^{\leq 0},\mathcal{D}^{\geq 0})$.
\item[b)] If $\mathcal{T}'\subset\hat{\mathcal{D}}^c$ is any partial silting set in $\mathcal{D}$ which generates  
$(\mathcal{D}^{\leq 0},\mathcal{D}^{\geq 0})$, then $\text{add}(\mathcal{T}')\subseteq\text{add}(\mathcal{T})$.
\end{enumerate}
If $(\mathcal{D}^{\leq 0},\mathcal{D}^{\geq 0})$ is the restriction of a t-structure $(\hat{\mathcal{D}}^{\leq 0},\hat{\mathcal{D}}^{\geq 0})$ in $\hat{\mathcal{D}}$ generated by some non-positive set $\mathcal{T}_0\subset\hat{\mathcal{D}}^c$, then $\text{add}(\mathcal{T})=\text{add}(\mathcal{T}_0)$.
\end{proposition}
\begin{proof}
Let $\mathcal{C}:={}^\perp(\mathcal{D}^{\leq 0}[1])\cap\mathcal{D}^{\leq 0}$ be the co-heart of $(\mathcal{D}^{\leq 0},\mathcal{D}^{\geq 0})$ and let $\mathcal{T}' \subset \mathcal{D}^c$ be any partial silting set in $\mathcal{D}$ which generates this t-structure. Since $\text{Hom}_\mathcal{D}(T',-)$ vanishes on $\mathcal{D}^{\leq 0}[1]$, for all $T'\in\mathcal{T}'$, we have $\mathcal{T}'\subset\mathcal{C}$, and hence $\text{add}(\mathcal{T}')\subseteq\mathcal{C}\cap\hat{\mathcal{D}}^c$. Let $\mathcal{T}$ be a set of representatives of isomorphism classes of objects of $\mathcal{C}\cap \hat{\mathcal{D}}^c$. Let us check that $\mathcal{T}$ generates $(\mathcal{D}^{\leq 0},\mathcal{D}^{\geq 0})$. 

Without loss of generality, we can assume that $\mathcal{T}'\subseteq\mathcal{T}$. This implies that $\mathcal{T}^{\perp_{\leq 0}}\cap\mathcal{D}\subseteq\mathcal{T}'^{\perp_{\leq 0}}\cap\mathcal{D}=\mathcal{D}^{>0}:=\mathcal{D}^{\geq 0}[-1]$. Since $\mathcal{T}\subset\mathcal{D}^{\leq 0}$, we have that $\text{Hom}_\mathcal{D}(T,Y)=0$, for all $T\in\mathcal{T}$ and $Y\in \mathcal{D}^{>0}$. Hence, the inclusion $\mathcal{D}^{>0}=\mathcal{T}'^{\perp_{\leq 0}}\cap\mathcal{D}\subseteq\mathcal{T}^{\perp_{\leq 0}}\cap\mathcal{D}$ also holds and $\mathcal{T}$ generates the t-structure $(\mathcal{D}^{\leq 0},\mathcal{D}^{\geq 0})$.

Let us prove the last assertion of the proposition. Suppose  that $(\mathcal{D}^{\leq 0},\mathcal{D}^{\geq 0})=({}^\perp (\mathcal{T}_0^{\perp_{\leq 0}})\cap\mathcal{D},\mathcal{T}_0^{\perp_{<0}}\cap\mathcal{D})$, for some non-positive set $\mathcal{T}_0\subset\hat{\mathcal{D}}^c$. Recall that ${}^\perp (\mathcal{T}_0^{\perp_{\leq 0}})=\text{Susp}_{\hat{\mathcal{D}}}(\mathcal{T}_0)$ (see \cite[Theorem 2]{NSZ}). If $C\in\mathcal{C}\cap\hat{\mathcal{D}}^c$, then $\text{Hom}_{\hat{\mathcal{D}}}(C,-)$ vanishes on $\text{Susp}_{\hat{\mathcal{D}}}(\mathcal{T}_0)[1]={}^\perp (\mathcal{T}_0^{\perp_{\leq 0}})[1]$. Indeed, $\bigcup_{k>0}\mathcal{T}_0[k]\subset\mathcal{D}^{\leq 0}[1]$, $\text{Hom}_{\hat{\mathcal{D}}}(C,-)$ vanishes on $\mathcal{D}^{\leq 0}[1]$ and $C$ is compact. Hence, $\mathcal{C}\cap\hat{\mathcal{D}}^c$ belongs to the co-heart $\hat{\mathcal{C}}:={}^\perp\text{Susp}_{\hat{\mathcal{D}}}(\mathcal{T}_0)[1]\cap\text{Susp}_{\hat{\mathcal{D}}}(\mathcal{T}_0)$ of the t-structure $({}^\perp (\mathcal{T}_0^{\perp_{\leq 0}}),\mathcal{T}_0^{\perp_{<0}})$ in $\hat{\mathcal{D}}$. By \cite[Lemma 6]{NSZ}, we conclude that $\mathcal{C}\cap\hat{\mathcal{D}}^c\subseteq\text{Add}(\mathcal{T}_0)$ and, since $\mathcal{C}\cap\hat{\mathcal{D}}^c$ consists of compact objects, $\mathcal{C}\cap\hat{\mathcal{D}}^c\subseteq\text{add}(\mathcal{T}_0)$. On the other hand, by 
Proposition \ref{prop.restricted partial silting}, $\mathcal{T}_0$ is a partial silting set in $\mathcal{D}$ which generates $(\mathcal{D}^{\leq 0},\mathcal{D}^{\geq 0})$. By the first paragraph of the proof, $\text{add}(\mathcal{T}_0)\subseteq\mathcal{C}\cap\hat{\mathcal{D}}^c$, and hence $\text{add}(\mathcal{T}_0)=\text{add}(\mathcal{T})$.
\end{proof}

Recall the notation and terminology of \ref{term}.

\begin{corollary} \label{cor.relative partial silting}
Let $\mathcal{D}$ be a triangulated category and let $\mathcal{T}$ be a partial silting set in $\mathcal{D}$. 
 For any $\star\in\{\emptyset,+,-,b\}$ and $\dagger\in\{\emptyset,fl\}$ the t-structure $\tau_\mathcal{T}=({}^\perp (\mathcal{T}^{\perp_{\leq 0}}),\mathcal{T}^{\perp_{<0}})$ restricts to $\mathcal{D}^\star_{\mathcal{T},\dagger}$. In particular,  if $\mathcal{T}$ is contained in $\mathcal{D}^\star_{\mathcal{T},\dagger}$, then   $\mathcal{T}$ is a partial silting set in $\mathcal{D}^\star_{\mathcal{T},\dagger}$. 
\end{corollary}
\begin{proof}
 For $M\in\mathcal{D}$,  let us consider the truncation triangle with respect to $\tau_\mathcal{T}$ $$U\longrightarrow M\longrightarrow V\stackrel{+}{\longrightarrow}. $$ Then $V\in\mathcal{T}^{\perp_{\leq 0}}$ and, since $\text{Hom}_\mathcal{D}(\mathcal{T},-)$ vanishes on ${}^\perp (\mathcal{T}^{\perp_{\leq 0}})[1]$, we get $U\in\mathcal{T}^{\perp_{>0}}$. This gives induced isomorphisms $\text{Hom}_\mathcal{D}(T,U[k])\cong\text{Hom}_\mathcal{D}(T,M[k])$, for $k\leq 0$, and $\text{Hom}_\mathcal{D}(T,M[k])\cong\text{Hom}_\mathcal{D}(T,V[k])$, for $k>0$. It immediately follows that $\tau_\mathcal{T}$ restricts to $\mathcal{D}^*_{\mathcal{T},\dagger}$, for any choices $\star\in\{\emptyset,+,-,b\}$ and $\dagger\in\{\emptyset,fl\}$.
\end{proof}

Using that for any generating set $\mathcal{X}$ of $\mathcal{D}$, 
consisting of compact objects, $\mathcal{D}^*_{\mathcal{X},\dagger}=\mathcal{D}^*_\dagger$, we get: 
\begin{corollary} \label{cor.restriction of silting for dg categories}
Let $\mathcal{D}$ be a compactly generated triangulated category and let $\mathcal{T}$ be a classical silting set in  $\mathcal{D}^c$. For any  $\star\in\{\emptyset,+,-,b\}$ and $\dagger\in\{\emptyset,fl\}$  the t-structure $\tau_\mathcal{T}=(^{\perp}(\mathcal{T}^{\perp_{\leq 0}}),\mathcal{T}^{\perp_{<0}})=(\mathcal{T}^{\perp_{>0}},\mathcal{T}^{\perp_{<0}})$ restricts to $\mathcal{D}^\star_\dagger$. And if $\mathcal{T}\subset\mathcal{D}^\star_\dagger$ (equivalently, if $\mathcal{D}^c\subset\mathcal{D}^\star_\dagger$), then $\mathcal{T}$ is a  silting  set of $\mathcal{D}^\star_\dagger$. 
\end{corollary}

\section{(Pre)envelopes and their constructions}\label{SecPreenv}

Recall that in any category $\mathcal{C}$, a morphism $f:C\longrightarrow C'$  is \emph{left (resp. right) minimal} when any endomorphism $g\in\text{End}_\mathcal{C}(C')$ (resp.  $g\in\text{End}_\mathcal{C}(C)$) such that $g\circ f=f$ (resp. $f\circ g=f$) is an isomorphism.  When  $\mathcal{X}$ is a subcategory, a morphism $f:C\longrightarrow X_C$, with $X_C\in\mathcal{X}$, is called an \emph{$\mathcal{X}$-preenvelope} or \emph{left $\mathcal{X}$-approximation} of $C$ if each morphism $g:C\longrightarrow X$, with $X\in\mathcal{X}$, factors through $f$. The dual concept is that of \emph{$\mathcal{X}$-precover} or \emph{right $\mathcal{X}$-approximation}.  An \emph{$\mathcal{X}$-envelope} (resp. \emph{$\mathcal{X}$-cover}) or \emph{minimal left $\mathcal{X}$-approximation} (resp. \emph{minimal right $\mathcal{X}$-approximation})  is an $\mathcal{X}$-preenvelope (resp. $\mathcal{X}$-precover) which is a left (resp. right) minimal morphism.   The subcategory $\mathcal{X}$ is called \emph{(pre)enveloping} (resp. \emph{(pre)covering}) when each object of $\mathcal{C}$ has an $\mathcal{X}$-(pre)envelope (resp. $\mathcal{X}$-(pre)cover). In this section we show some relationship between (pre)enveloping subcategories and t- and co-t-structures in a triangulated category $\mathcal{D}$. 

The following result is folklore and follows from
\cite[Corollary 1.4]{KS}.

\begin{lemma} \label{lem.preenvelopes in Krull-Schmidt}
Let $\mathcal{V}$ be a full subcategory of $\mathcal{D}$ such that
$\mathcal{V}$ is Krull-Schmidt. If an object $M$ of $\mathcal{D}$ has a $\mathcal{V}$-preenvelope (resp. $\mathcal{V}$-precover), then
it has a $\mathcal{V}$-envelope (resp. $\mathcal{V}$-cover).
\end{lemma}

\begin{lemma} \label{lem.envelope versus co-t-structures}
Let $\mathcal{V}$ be a full subcategory of $\mathcal{D}$ closed
under extensions, let $f:M\longrightarrow
V$ be a morphism with $V\in\mathcal{V}$.  Consider the following
assertions:
\begin{enumerate}
\item $f$ is a $\mathcal{V}$-envelope
\item the object $U$ in the triangle $U\longrightarrow M\stackrel{f}{\longrightarrow} V\stackrel{+}{\longrightarrow}$ belongs to  $ {}^{\perp}\mathcal{V}$
\item $f$ is a $\mathcal{V}$-preenvelope.
\end{enumerate}
Then $1)\Longrightarrow 2)\Longrightarrow 3)$ holds.
\end{lemma}

\begin{proof}
$1)\Longrightarrow 2)$ Adapt the proof of \cite[Lemma 1.3]{Bu}.

$2)\Longrightarrow 3)$ 
Applying the functor $\text{Hom}_\mathcal{D}(-,V')$ to the triangle
from assertion 2, we get that
$\text{Hom}_\mathcal{D}(f,V'):\text{Hom}_\mathcal{D}(V,V')\longrightarrow\text{Hom}_\mathcal{D}(M,V')$
is an epimorphism for any $V'\in\mathcal{V}$, thus $f$ is a
$\mathcal{V}$-preenvelope.
\end{proof}

\begin{lemma} \label{lem.inductive step construction envelopes}
Let $\mathcal{E}$ and $\mathcal{F}$ be full subcategories of
$\mathcal{D}$. Consider the following homotopy pushout diagram,
where the rows are triangles.
$$
\begin{xymatrix} {C \ar[d]^{g} \ar[r]^{u}&M \ar[r]^h \ar[d]^{f}  &F \ar@{=}[d]\\
E \ar[r]^{u'} &X \ar[r]^{h'} &F}
\end{xymatrix}
$$
\begin{enumerate}
\item If $h$ is an $\mathcal{F}$-preenvelope and $g$ is an $\mathcal{E}$-preenvelope, then $f$ is an $\mathcal{E}\star\mathcal{F}$-preenvelope.
\item Suppose that $\mathcal{E}$ and $\mathcal{F}$ are closed under extensions, and that the inclusion $\mathcal{F}\subseteq\mathcal{E}[1]$ holds. If $g$ is an $\mathcal{E}$-envelope and $h$ is an $\mathcal{F}$-envelope, then $f$ is an $\mathcal{E}\star\mathcal{F}$-envelope (and hence an $\text{add}(\mathcal{E}\star\mathcal{F})$-envelope), provided that  one of the following conditions hold:
\begin{enumerate}
 \item  $\mathcal{D}$ is Krull-Schmidt.
\item $\text{Hom}_\mathcal{D}(E,F)=0$, for all $E\in \mathcal{E}$ and $F\in\mathcal{F }$.
\end{enumerate}
\end{enumerate}
\end{lemma}
\begin{proof}
1) Let $f':M\longrightarrow X'$ be any morphism, where
$X'\in\mathcal{E}\star\mathcal{F}$ and fix a triangle
$E'\stackrel{\gamma}{\longrightarrow}
X'\stackrel{\delta}{\longrightarrow}
F'\stackrel{+}{\longrightarrow}$, with $E'\in\mathcal{E}$ and
$F'\in\mathcal{F}$. The $\mathcal{F}$-preenveloping condition on $h$
gives a morphism $\rho :F\longrightarrow F'$ such that $\rho\circ
h=\delta\circ f'$. We then get a morphism $g':C\longrightarrow E'$
making commutative the following diagram:
$$
\begin{xymatrix} {C \ar[d]^{g'} \ar[r]^{u}&M \ar[r]^h \ar[d]^{f'}  &F \ar[d]^{\rho} \\
E' \ar[r]^{\gamma} &X' \ar[r]^{\delta} &F'}
\end{xymatrix}
$$
The $\mathcal{E}$-preenveloping condition of $g$ gives a morphism
$\lambda :E\longrightarrow E'$ such that $g'=\lambda\circ g$. Thus $\gamma\circ\lambda\circ g=\gamma\circ g'=f'\circ u$ and there exists $\mu :X\longrightarrow X'$ such that $\mu\circ
u'=\gamma\circ\lambda$ and $\mu\circ f=f'$, since the diagram we started from is
a homotopy pushout. In particular, $f'$
factors through $f$ so that $f$ is an
$\mathcal{E}\star\mathcal{F}$-preenvelope.

2) Since any $\mathcal{E}\star\mathcal{F}$-envelope is  an
$\text{add}(\mathcal{E}\star\mathcal{F})$-envelope, we only need to check the left minimality
of $f$.

2.a) When $\mathcal{D}$ is Krull-Schmidt, there is a 
decomposition $f=\begin{pmatrix} f'&
0\end{pmatrix}^t :M\longrightarrow X_1\oplus X_2=X$, where
$f':M\longrightarrow X_1$ is left minimal. Thus we can assume that the
triangle $\text{C}(f)\longrightarrow
M\stackrel{f}{\longrightarrow}X\stackrel{+}{\longrightarrow}$, coincides with the
triangle $$\text{C}(f')\oplus X_2[-1]\stackrel{\begin{pmatrix}
\gamma & 0\end{pmatrix}}{\longrightarrow}M\stackrel{\begin{pmatrix}
f'& 0\end{pmatrix}^t}{\longrightarrow}X_1\oplus
X_2\stackrel{+}{\longrightarrow}.$$  Since homotopy pushout squares are also homotopy pullback we get a triangle ($*$) $$C(g)=\text{C}(f')\oplus X_2[-1]\stackrel{\begin{pmatrix}
\alpha & \beta\end{pmatrix}}{\longrightarrow}C\stackrel{g}{\longrightarrow}E\stackrel{+}{\longrightarrow},$$
so $u\circ \begin{pmatrix}\alpha &
\beta \end{pmatrix}=\begin{pmatrix}\gamma & 0 \end{pmatrix}$ and $u\circ\beta=0$. Thus $\beta$ admits a factorization
$\beta :X_2[-1]\longrightarrow F[-1]\longrightarrow C$. But 
$F[-1]\in\mathcal{F}[-1]\subseteq\mathcal{E}$ and $X_2[-1]$ is a
direct summand of $\text{C}(g)$. By Lemma \ref{lem.envelope versus
co-t-structures},  $\text{C}(g)\in
{}^\perp\mathcal{E}$, which implies $\beta =0$. Hence, the triangle ($*$) is isomorphic to $\text{C}(f')\oplus
X_2[-1]\stackrel{\begin{pmatrix} \alpha & 0
\end{pmatrix}}{\longrightarrow} C\stackrel{\begin{pmatrix}g'& 0
\end{pmatrix}^t}{\longrightarrow}E'\oplus
X_2\stackrel{+}{\longrightarrow},$ where
$E\simeq E'\oplus X_2$. The left minimality of $g$ implies $X_2=0$
and, hence, that $f$ is left minimal.

2.b) Assume now that
$\text{Hom}_\mathcal{D}(\mathcal{E},\mathcal{F})=0$. Let
$\alpha\in\text{End}_\mathcal{D}(X)$ be  such that $\alpha\circ
f=f$. Since $h'\circ\alpha\circ u'\in\text{Hom}_\mathcal{D}(E,F)=0$,
there are $\alpha_1:E\longrightarrow E$ and
$\alpha_2:F\longrightarrow F$ making the following diagram
commutative:
$$
\begin{xymatrix} {F[-1] \ar[d] \ar[r]^{\lambda}&E \ar[d]^{\alpha_1} \ar[r]^{u'}&X \ar[r]^{h'} \ar[d]^{\alpha}  &F \ar[d]^{\alpha_2}\\
F[-1] \ar[r]^{\lambda}&E \ar[r]^{u'} &X \ar[r]^{h'} &F}
\end{xymatrix}
$$
Then $u'\circ\alpha_1\circ g=\alpha\circ u'\circ
g=\alpha\circ f\circ u=f\circ u=u'\circ g,$ which implies 
$u'\circ (g-\alpha_1\circ g)=0$ and, hence,  $g -\alpha_1\circ
g$ factors in the form $C\stackrel{t}{\longrightarrow}F[-1]\stackrel{\lambda}{\longrightarrow}E$.
But  $F[-1]\in\mathcal{F}[-1]\subseteq\mathcal{E}$  and
since $g$ is an $\mathcal{E}$-envelope, there is a morphism
$\pi :E\longrightarrow F[-1]$ such that $t=\pi\circ g$. It follows
that $g-\alpha_1\circ g=\lambda\circ\pi\circ g$ and $g=(\alpha_1+\lambda\circ\pi )\circ g$. The left minimality of
$g$ implies that $\alpha_1+\lambda\circ\pi$ is an isomorphism. But
 $u'\circ (\alpha_1+\lambda\circ\pi)=u'\circ\alpha_1$ since
$u'\circ\lambda =0$. This means that we can replace $\alpha_1$ by
$\alpha_1+\lambda\circ\pi$ (and $\alpha_2$ by by some new
$\alpha_2$) and assume that $\alpha_1$ is an
isomorphism.

Note now that $\alpha_2\circ h=\alpha_2\circ h'\circ
f=h'\circ\alpha\circ f= h'\circ f=h$. Then the left minimality of
$h$ implies that $\alpha_2$ is an isomorphism and, as a consequence,
 $\alpha$ is an isomorphism.
\end{proof}

\begin{corollary} \label{cor.E-star-F enveloping}
Let $\mathcal{E}$ and $\mathcal{F}$ be 
enveloping subcategories  of the triangulated category
$\mathcal{D}$ closed under extensions and such that $\mathcal{F}\subseteq\mathcal{E}[1]$. If
either $\mathcal{D}$ is Krull-Schmidt or
$\text{Hom}_\mathcal{D}(\mathcal{E},\mathcal{F})=0$, then
$\mathcal{E}\star\mathcal{F}$ is an enveloping subcategory of
$\mathcal{D}$ and,  in particular, it is closed under direct summands. If,
moreover, $\text{Hom}_\mathcal{D}(\mathcal{E},\mathcal{F}[1])=0$ then
$\mathcal{E}\star\mathcal{F}$ is also closed under extensions in
$\mathcal{D}$.
\end{corollary}
\begin{proof}
The enveloping condition on $\mathcal{E}\star\mathcal{F}$ is a
direct consequence of Lemma \ref{lem.inductive step construction
envelopes}, and it is well-known that any enveloping subcategory is
closed under direct summands. The final statement  follows from
\cite[Lemma 8]{NSZ}.
\end{proof}

By \cite[Lemma 2.15]{AI} we have the following:

\begin{lemma} \label{lem.thick-suspended}
Let $\mathcal{T}$ be a non-positive set of objects of $\mathcal{D}$.
Then
\begin{enumerate}
\item $\text{thick}_\mathcal{D}(\mathcal{T})=\bigcup_{r\leq s}\text{add}(\text{add}(\mathcal{T})[r]\star\text{add}(\mathcal{T})[r+1]\star \cdots\star\text{add}(\mathcal{T})[s])$. Moreover, if $\text{add}(\mathcal{T})$ is an enveloping subcategory of $\mathcal{D}$, then $$\text{thick}_\mathcal{D}(\mathcal{T})=\bigcup_{r\leq s}(\text{add}(\mathcal{T})[r]\star\text{add}(\mathcal{T})[r+1]\star \cdots\star\text{add}(\mathcal{T})[s])$$
\item $\text{susp}_\mathcal{D}(\mathcal{T})=\bigcup_{r\geq 0}\text{add}(\text{add}(\mathcal{T})\star\text{add}(\mathcal{T})[1]\star \cdots\star\text{add}(\mathcal{T})[r])$. If $\text{add}(\mathcal{T})$ is an enveloping subcategory of $\mathcal{D}$, then $\text{susp}_\mathcal{D}(\mathcal{T})=\bigcup_{r\geq 0}(\text{add}(\mathcal{T})\star\text{add}(\mathcal{T})[1]\star \cdots\star\text{add}(\mathcal{T})[r])$.
\end{enumerate}
\end{lemma}
\begin{proof}
First equality in the assertion (1) is \cite[Lemma 2.15]{AI}, first equality in the assertion (2) is proved analogously. When $\text{add}(\mathcal{T})$ is enveloping, the assertions follow by an iterative application of Corollary \ref{cor.E-star-F
enveloping}.
\end{proof}

For an object $M$  and a subcategory
$\mathcal{T}$ in 
$\mathcal{D}$, we shall use the notation 
$$s(M,\mathcal{T}):= Sup\{k\in\mathbb{N}\mid \text{Hom}_\mathcal{D}(M,-[k])_{| \mathcal{T}}\neq
0 \}\in\mathbb{N}\cup\{ \infty\}, $$ when this subset of natural numbers is nonempty.
When this subset 
is empty, by
convention, we put 
$\text{add}(\mathcal{T})\star\text{add}(\mathcal{T})[1]\star
\cdots\star\text{add}(\mathcal{T})[s(M,\mathcal{T})]:=0$.

\begin{lemma} \label{lem.description of U-envelope}
Let $\mathcal{T}$ be a nonpositive set of objects of $\mathcal{D}$ and  $\mathcal{U}:=\text{susp}_\mathcal{D}(\mathcal{T})$.  The following assertions are equivalent for an
object $M\in \mathcal{D}$:

\begin{enumerate}
\item $M$ has a $\mathcal{U}$-(pre)envelope.
\item $\text{Hom}_\mathcal{D}(M,-[k])_{| \mathcal{T}}=0$ for $k>>0$, and $M$ has an $\text{add}(\text{add}(\mathcal{T})\star\text{add}(\mathcal{T})[1]\star \cdots\star\text{add}(\mathcal{T})[s])$-(pre)envelope, where $s=s(M,\mathcal{T})$.
\end{enumerate}
\end{lemma}
\begin{proof}
$1)\Longrightarrow 2)$ 
Let us check that $\text{Hom}_\mathcal{D}(M,-[k])_{| \mathcal{T}}=0$ for $k>>0$. Let
 $f:M\longrightarrow U$ be a $\mathcal{U}$-(pre)envelope. By Lemma \ref{lem.thick-suspended}  there exists an $r\in\mathbb{N}$ such that
$U\in\text{add}(\text{add}(\mathcal{T})\star\text{add}(\mathcal{T})[1]\star
\cdots\star\text{add}(\mathcal{T})[r])$. If  $k>r$ and
$g:M\longrightarrow T[k]$ is a morphism, with $T\in\mathcal{T}$,
then $g$  factors in the form
$g:M\stackrel{f}{\longrightarrow}U\stackrel{h}{\longrightarrow}T[k]$,
where the second arrow is zero since
$\text{Hom}_\mathcal{D}(-,T[k])$ vanishes on
$\text{add}(\mathcal{T})[j]$, for $j=0,1,\dots,r$.

There is a triangle
$U'\stackrel{\begin{pmatrix}v_1& v_2 \end{pmatrix}^t}{\longrightarrow}U\oplus Z \stackrel{\begin{pmatrix} p_1 & p_2\end{pmatrix}}{\longrightarrow}U''\stackrel{+}{\longrightarrow}$,
where
$U'\in\text{add}((\mathcal{T})\star\text{add}(\mathcal{T})[1]\star
\cdots\star\text{add}(\mathcal{T})[s]$ and
$U''\in\text{add}(\mathcal{T})[s+1]\star
\cdots\star\text{add}(\mathcal{T})[r]$. By definition of
$s=s(M,\mathcal{T})$, we have that
$\text{Hom}_\mathcal{C}(M,U'')=0$, and  so  $0=p_1\circ f=\begin{pmatrix}p_1 & p_2 \end{pmatrix}\circ\begin{pmatrix} f& 0\end{pmatrix}^t$. This
implies that $\begin{pmatrix}f& 0 \end{pmatrix}^t:M\longrightarrow U\oplus Z$ admits a factorization
$\begin{pmatrix} f& 0\end{pmatrix}^t:M\stackrel{f'}{\longrightarrow}U'\stackrel{\begin{pmatrix} v_1& v_2\end{pmatrix}^t}{\longrightarrow}U\oplus Z$, and so $f=v_1\circ f'$. Then
 $f'$ is clearly the desired preenvelope. If $f$ was an envelope, then $U$ is a summand of $U'$ and $f$ is the desired envelope.

$2)\Longrightarrow 1)$ 
Let $0\neq f:M\longrightarrow X$ be any morphism with
$X\in\mathcal{U}$, then
$X\in\text{add}(\text{add}(\mathcal{T})\star\text{add}(\mathcal{T})[1]\star
\cdots\star\text{add}(\mathcal{T})[r])$, for some $r\in\mathbb{N}$.
Without loss of generality, we assume that
$X\in\text{add}(\mathcal{T})\star\text{add}(\mathcal{T})[1]\star
\cdots\star\text{add}(\mathcal{T})[r]$. There is a triangle
$X'\stackrel{v}{\longrightarrow}X\stackrel{p}{\longrightarrow}X''\stackrel{+}{\longrightarrow}$,
where
$X'\in\text{add}((\mathcal{T})\star\text{add}(\mathcal{T})[1]\star
\cdots\star\text{add}(\mathcal{T})[s]$ and
$X''\in\text{add}(\mathcal{T})[s+1]\star
\cdots\star\text{add}(\mathcal{T})[r]$. As before,
$\text{Hom}_\mathcal{C}(M,X'')=0$, and  so  $p\circ f =0$. This
implies that $f$ admits a factorization
$f:M\stackrel{f'}{\longrightarrow}X'\stackrel{v}{\longrightarrow}X$. Thus any morphism from $M$ to $\mathcal{U}$ factors through $\text{add}(\text{add}(\mathcal{T})\star\text{add}(\mathcal{T})[1]\star \cdots\star\text{add}(\mathcal{T})[s])$ and we are done.
\end{proof}

\begin{proposition} \label{prop.U-envelope}
Let $\mathcal{D}$ be a triangulated category and $\mathcal{T}$ be a
non-positive set of objects in $\mathcal{D}$. Consider the following
assertions:
\begin{enumerate}
\item[1] $\text{Hom}_\mathcal{D}(M,-[k])_{| \mathcal{T}}=0$ for any $M\in \mathcal{D}$, $k>>0$ and $M$ has an $\text{add}(\mathcal{T})[s(M,\mathcal{T})]$-envelope.
\item[1'] $\text{Hom}_\mathcal{D}(M,-[k])_{| \mathcal{T}}=0$ for any $M\in \mathcal{D}$, $k>>0$ and $M$ has an $\text{add}(\mathcal{T})[s(M,\mathcal{T})]$-preenvelope.
\item[2] $\text{Hom}_\mathcal{D}(M,-[k])_{| \mathcal{T}}=0$ for any $M\in \mathcal{D}$, $k>>0$ and $M$ has an $add(\text{add}(\mathcal{T})\star\text{add}(\mathcal{T})[1]\star \cdots\star\text{add}(\mathcal{T})[s(M,\mathcal{T})])$-envelope.
\item[2'] $\text{Hom}_\mathcal{D}(M,-[k])_{| \mathcal{T}}=0$ for any $M\in \mathcal{D}$, $k>>0$ and $M$ has an $add(\text{add}(\mathcal{T})\star\text{add}(\mathcal{T})[1]\star \cdots\star\text{add}(\mathcal{T})[s(M,\mathcal{T})])$-preenvelope.
\item[3] $\text{susp}_\mathcal{D}(\mathcal{T})$ is an enveloping class in $\mathcal{D}$.
\item[3'] $\text{susp}_\mathcal{D}(\mathcal{T})$ is a preenveloping class in $\mathcal{D}$.
\item[4] $( {}^\perp\text{susp}_\mathcal{D}(\mathcal{T})[1],\text{susp}_\mathcal{D}(\mathcal{T}))$ is a co-t-structure in $\mathcal{D}$.
\end{enumerate}
Then implications \begin{center}$
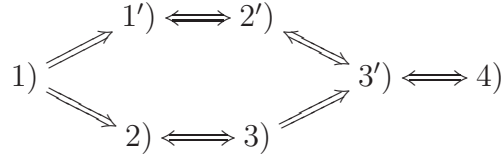
\begin{xymatrix}@R=.3pc {&1')\ar@{<=>}[r]&2') \ar@{<=>}[dr]&&\\
1) \ar@{=>}[ur] \ar@{=>}[dr]&&& 3') \ar@{<=>}[r]& 4)\\
&2) \ar@{<=>}[r]&3) \ar@{=>}[ur]&&}
\end{xymatrix}
$\end{center} hold and if $\mathcal{D}$ is Krull-Schmidt, then all the assertions
are equivalent.
Moreover, when assertion 1 holds,  the envelope $M\longrightarrow U$
from assertion 2, which is also a
$\text{susp}_\mathcal{D}(\mathcal{T})$-envelope, can be constructed
inductively.
\end{proposition}
\begin{proof}
The implications $2)\Longleftrightarrow 3)\Longrightarrow
4)\Longrightarrow 3')$  follow from Lemma
\ref{lem.description of U-envelope} and Lemma \ref{lem.envelope versus
co-t-structures}. The equivalence $2)'\Longleftrightarrow 3')$ also
follows from Lemma \ref{lem.description of U-envelope}, and the
implications $1)\Longrightarrow 1')$ and  $3)\Longrightarrow 3')$
are clear. Apart from  the statement about inductive
construction, it is enough to prove implications $1)\Longrightarrow
2)$ and $2')\Longrightarrow 1')\Longrightarrow 4)$,  then the
equivalence of all assertions when $\mathcal{D}$ is Krull-Schmidt
will follow from Lemma \ref{lem.preenvelopes in Krull-Schmidt}.

$1)\Longrightarrow 2)$ Without loss of generality, we only consider
$M$ such that
$\text{Hom}_\mathcal{C}(M,-[k])_{| \mathcal{T}}\neq 0$, for some
$k\in\mathbb{N}$. Let us prove by induction on $r\geq 0$ that if $M$ is
an object such that  $0\leq s:=s(M,\mathcal{T})\leq r$, then  $M$ has
an $\text{add}(\mathcal{T})\star\text{add}(\mathcal{T})[1]\star
\cdots\star\text{add}(\mathcal{T})[r]$-envelope.  Note that if
$s:=s(M,\mathcal{T})<r$, then by the induction hypothesis,
there is an
$\text{add}(\mathcal{T})\star\text{add}(\mathcal{T})[1]\star
\cdots\star\text{add}(\mathcal{T})[s]$-envelope, which is easily seen
to be an
$\text{add}(\mathcal{T})\star\text{add}(\mathcal{T})[1]\star
\cdots\star\text{add}(\mathcal{T})[r]$-envelope. Assume
 $r=s$, and fix an $\text{add}(\mathcal{T})[s]$-envelope
$h:M\longrightarrow T_M[s]$, which we complete to a triangle
$C\stackrel{u}{\longrightarrow}M\stackrel{h}{\longrightarrow}T_M[s]\stackrel{+}{\longrightarrow}$
($*$). Since $\text{add}(\mathcal{T})[s]$ is closed
under extensions, by Lemma \ref{lem.envelope versus co-t-structures},
 $\text{Hom}_\mathcal{C}(C,-[s])_{| \mathcal{T}}=0$.
Applying $\text{Hom}_\mathcal{C}(-,T[k])$ to the triangle ($*$), we
 see that $\text{Hom}_\mathcal{C}(C,-[k])_{| \mathcal{T}}=0$,
for $k\geq s$. Then  $s(C,\mathcal{T})<s$. By the induction
hypothesis there is an
$\text{add}(\mathcal{T})\star\text{add}(\mathcal{T})[1]\star
\cdots\star\text{add}(\mathcal{T})[r-1]$-envelope $g:C\longrightarrow
E$. Then, for
$\mathcal{E}=\text{add}(\mathcal{T})\star\text{add}(\mathcal{T})[1]\star
\cdots\star\text{add}(\mathcal{T})[r-1]$,
$\mathcal{F}=\text{add}(\mathcal{T})[r]$ and $F=T_M[r]$,  Lemma
\ref{lem.inductive step construction envelopes} implies that $M$
has an 
$\mathcal{E}\star\mathcal{F}=\text{add}(\mathcal{T})\star\text{add}(\mathcal{T})[1]\star
\cdots\star\text{add}(\mathcal{T})[r]$-envelope.

$2')\Longrightarrow 1')$ Let $f:M\longrightarrow U$ be an
$\text{add}(\mathcal{T})\star\text{add}(\mathcal{T})[1]\star
\cdots\star\text{add}(\mathcal{T})[s]$-preenvelope, where
$s=s(M,\mathcal{T})\geq 0$.  There is a
triangle $X'\longrightarrow
U\stackrel{g}{\longrightarrow}T[s]\stackrel{+}{\longrightarrow}$,
where
$X'\in\text{add}(\mathcal{T})\star\text{add}(\mathcal{T})[1]\star
\cdots\star\text{add}(\mathcal{T})[s-1]$ and
$T\in\text{add}(\mathcal{T})$. Let $h:M\longrightarrow T'[s]$ be
any morphism, where $T'\in\mathcal{T}$. Then  there is a morphism $\eta :U\longrightarrow T'[s]$ such
that $\eta\circ f=h$. Since $\text{Hom}_\mathcal{C}(X',T'[s])=0$, there is a
morphism $\mu :T[s]\longrightarrow T'[s]$ such that $\mu\circ
g=\eta$. Thus, $h=\eta\circ f=\mu\circ g\circ f$ and $g\circ f$ is an
$\text{add}(\mathcal{T})[s]$-preenvelope of $M$.

$1')\Longrightarrow 4)$ Put
$\mathcal{U}:=\text{susp}_\mathcal{D}(\mathcal{T})$. Let us prove that any object $M$ fits into a triangle
$V_M\longrightarrow M\longrightarrow
U_M\stackrel{+}{\longrightarrow}$, where $U_M\in\mathcal{U}$ and
$V_M\in {}^\perp\mathcal{U}$. If $M\in {}^\perp\mathcal{U}$ there
is nothing to prove. We then assume that $M\not\in
{}^\perp\mathcal{U}$, so that $s(M,\mathcal{T})\geq 0$. Let us prove the statement by
induction on $s(M,\mathcal{T})$. Assume $s(M,\mathcal{T})=0$ and consider  the triangle $V_M\longrightarrow
M\stackrel{f}{\longrightarrow}T_0\stackrel{+}{\longrightarrow}$,
where $f$ is an $\text{add}(\mathcal{T})$-preenvelope, which exists
by the hypothesis. It follows that the map
$f^*:\text{Hom}_\mathcal{D}(T_0,T)\longrightarrow\text{Hom}_\mathcal{D}(M,T)$
is an epimorphism and
$\text{Hom}_\mathcal{D}(V_M,T[k])=0$,   for all $T\in\mathcal{T}$ and
all integers $k\geq 0$.  Given the description of
$\mathcal{U}$ from Lemma \ref{lem.thick-suspended}, we conclude
that $V_M\in {}^\perp\mathcal{U}$.

Suppose  $s:=s(M,\mathcal{T})>0$ and that all
$N\in\mathcal{D}$ such that $s(N,\mathcal{T})<s$ admit the desired
triangle. Consider a triangle $X\longrightarrow
M\stackrel{g}{\longrightarrow}T_s[s]\stackrel{+}{\longrightarrow}$,
where $g$ is an $\text{add}(T)[s]$-preenvelope. Applying the functor $\text{Hom}_\mathcal{D}(-,T[k])$, for
$T\in\mathcal{T}$, to this
triangle we  see that
$\text{Hom}_\mathcal{D}(X,T[k])=0$, for all $T\in\mathcal{T}$ and
 $k\geq s$. It follows that  $s(X,\mathcal{T})<s$. By the induction
hypothesis  $X\in {}^\perp\mathcal{U}\star\mathcal{U}$ and  $T_s[s]\in
 {}^\perp\mathcal{U}\star\mathcal{U}$. It follows that $M\in
{}^\perp\mathcal{U}\star\mathcal{U}$ since
${}^\perp\mathcal{U}\star\mathcal{U}$ is closed under extensions
(see \cite[Lemma 8]{NSZ}).

Finally, the proof of implication $1)\Longrightarrow 2)$ shows how
to construct 
$\text{susp}_\mathcal{D}(\mathcal{T})$-envelopes inductively.
\end{proof}

\begin{definition} \label{def.weakly preenveloping}
We shall say that a non-positive set $\mathcal{T}$ in $\mathcal{D}$
is \emph{weakly preenveloping} when it satisfies condition (1') of
Proposition \ref{prop.U-envelope}. The notion  of a \emph{weakly precovering} nonpositive set of objects is defined dually. 
\end{definition}

Recall that an object $G$ of a triangulated category $\mathcal{D}$ is called a \emph{classical generator} when $\text{thick}_\mathcal{D}(G)=\mathcal{D}$. Recall also that if a pair $(\mathcal{X},\mathcal{Y})$ is a t-structure or a co-t-structure in $\mathcal{D}$, it is called \emph{left (resp. right) bounded} when $\mathcal{D}=\bigcup_{k\in\mathbb{Z}}\mathcal{X}[k]$ (resp.  $\mathcal{D}=\bigcup_{k\in\mathbb{Z}}\mathcal{Y}[k]$). The pair is called \emph{bounded} when it is left and right bounded.

The following proposition together with its dual generalizes \cite[Proposition 3.2]{IY}.

\begin{proposition} \label{prop.bijection psilting-cotstructures}
Let $\mathcal{D}$ be a skeletally small triangulated category
with split idempotents. The assignment
$\mathcal{T}\rightsquigarrow(
{}^\perp\text{susp}_\mathcal{D}(\mathcal{T})[1],\text{susp}_\mathcal{D}(\mathcal{T}))$
gives a one-to-one correspondence between (add-)equivalence classes of
weakly preenveloping non-positive sets and left bounded
co-t-structures in $\mathcal{D}$. Its inverse associates to 
such a co-t-structure a set of representatives of the isomorphism classes of
the objects of its co-heart.

This correspondence restricts to a bijection between equivalence
classes of classical silting sets and bounded co-t-structures in
$\mathcal{D}$. When $\mathcal{D}$ has a classical generator, this
 induces a bijection between equivalence classes of silting
objects and bounded co-t-structures in $\mathcal{D}$.
\end{proposition}
\begin{proof}
By Proposition \ref{prop.U-envelope},  $ \tau
(\mathcal{T}):=(
{}^\perp\mathcal{U}[1],\mathcal{U}):= (
{}^\perp\text{susp}_\mathcal{D}(T)[1],\text{susp}_\mathcal{D}(\mathcal{T}))$
is a co-t-structure in $\mathcal{D}$. Moreover, for each $M$ in
$\mathcal{D}$, there exists  $r\in\mathbb{N}$ such that
$\text{Hom}_{\mathcal{D}}(M,-[k])_{| \mathcal{T}}=0$, for  $k\geq
r$. It follows that $M\in  {}^\perp\mathcal{U}[r]$ and $\tau (\mathcal{T})$ is left bounded.

By \cite{Bo},
the co-heart of any co-t-structure $\tau$ is a non-positive class of
objects. In our case it is skeletally small, so  we can chose a
set $\mathcal{T}(\tau)$ of representatives of isomorphism classes of its
objects. We claim that $\mathcal{T}$ and
$\mathcal{T}(\tau (\mathcal{T}))$ are equivalent non-positive sets.
The inclusion $\mathcal{T}\subset\mathcal{C}$,
where $\mathcal{C}$ is the co-heart of $\tau (\mathcal{T})$, clearly holds, so we
need to prove that $\mathcal{C}\subseteq\text{add}(\mathcal{T})$.
For $0\neq C\in\mathcal{C}$ we get $s(C,\mathcal{T})=0$, since 
$\text{Hom}_\mathcal{D}(C,-)$ vanishes on $\mathcal{U}[1]$. Since $\mathcal{T}$ is weakly preenveloping, there is an
$\text{add}(\mathcal{T})$-preenvelope  $f:C\longrightarrow T_C$,
let us consider a triangle
$V_C\stackrel{g}{\longrightarrow}C\stackrel{f}{\longrightarrow}T_C\stackrel{+}{\longrightarrow}$.
As before (see the proof of implication $1')\Longrightarrow 4)$ in Proposition
\ref{prop.U-envelope}), $V_C \in {}^\perp\mathcal{U}$ and, hence, $g=0$. It
follows that $f$ is a section and
$C\in\text{add}(\mathcal{T})$.

Let $\tau=( {}^\perp\mathcal{U}_\tau [1] ,\mathcal{U}_\tau)$ be
any left bounded co-t-structure in $\mathcal{D}$, let $\mathcal{C}:=
{}^\perp\mathcal{U}_\tau [1]\cap\mathcal{U}_\tau$ be its co-heart and
let $\mathcal{T}$ be a set of representatives of its isomorphism classes. The left boundedness of $\tau$ implies that $\text{Hom}_\mathcal{D}(M,-[k])_{| \mathcal{T}}=0$ for any $M\in\mathcal{D}$ for
$k>>0$.  Clearly $s:=s(M,\mathcal{U})\geq s(M,\mathcal{T})$.
We claim that the inverse inequality also holds, provided
$s(M,\mathcal{U})\geq 0$. Let us consider the triangle coming from the co-t-structure $\tau$:
$V\longrightarrow M[-s]\stackrel{f}{\longrightarrow}
U\stackrel{+}{\longrightarrow}$. For an arbitrary
$U'\in\mathcal{U}_\tau$  applying $\text{Hom}_\mathcal{D}(-,U')$ to this triangle gives $\text{Hom}_\mathcal{D}(U,U'[k])=0$, for all $k>0$.
Thus, $U\in {}^\perp\mathcal{U}_\tau
[1]\cap\mathcal{U}_\tau =\mathcal{C}$ and $U\in\text{add}(\mathcal{T})$. Clearly, the map
$f:M[-s]\longrightarrow\mathcal{U}$ is an
$\text{add}(\mathcal{T})$-preenvelope. This in turn implies that
$f[s]:M\longrightarrow U[s]$ is an
$\text{add}(\mathcal{T})[s]$-preenvelope. Note that $f$ is a
nonzero map, since, otherwise  $M\in
{}^\perp\mathcal{U}_\tau$, contradicting the hypothesis. This implies that
$s(M,\mathcal{T})=s$ and that $M$ has an
$\text{add}(\mathcal{T})[s(M,\mathcal{T})]$-preenvelope. Hence,
$\mathcal{T}$ is  weakly preenveloping and the map from the set of left bounded
co-t-structures to weakly preenveloping non-positive sets is
well-defined.

The last paragraph shows that if $M\not\in
{}^\perp\mathcal{U}_\tau$, then there exists a nonzero morphism
$f:M\longrightarrow T[s]$, for some $T\in\mathcal{T}$, where
$s=s(M,\mathcal{U})=s(M,\mathcal{T})$. It follows that $
{}^\perp\mathcal{U}_\tau = {}^\perp (\bigcup_{k\geq
0}\mathcal{T}[k])$ and, hence, that $ {}^\perp\mathcal{U}_\tau=
{}^\perp\text{susp}_\mathcal{D}(\mathcal{T})$. Due to the weak
preenveloping condition on $\mathcal{T}$,  Proposition
\ref{prop.U-envelope} provides a co-t-structure $\tau':=(
{}^\perp\text{susp}_\mathcal{D}(\mathcal{T})[1],\text{susp}_\mathcal{D}(\mathcal{T}))$
 in $\mathcal{D}$. Clearly, $\tau '=\tau$. Since $\tau=\tau
(\mathcal{T}(\tau))$, the assignments
$\mathcal{T}\rightsquigarrow\tau (\mathcal{T})$ and
$\tau\rightsquigarrow\mathcal{T}(\tau )$ define mutually inverse maps.

As for the last statement, note that the dual version of the result above gives the
bijection $\mathcal{T}\rightsquigarrow
(\text{cosusp}_\mathcal{D}(\mathcal{T}),\text{cosusp}_\mathcal{D}(\mathcal{T})^\perp[-1])$
 between the equivalence classes of weakly
precovering non-positive sets and right bounded co-t-structures
in $\mathcal{D}$, the inverse of this map takes any such
co-t-structure $\tau$ to a set of representatives of
isomorphism classes  of objects of the co-heart of $\tau$. If  $\tau$
is a bounded co-t-structure in $\mathcal{D}$ and $\mathcal{T}_\tau$
is a set of representatives of
isomorphism classes  of  objects of its co-heart, then
we deduce from the bijections and from the construction of the triangle with respect to $\tau$ that $\tau
=(\text{cosusp}_\mathcal{D}(\mathcal{T}),\text{susp}_\mathcal{D}(\mathcal{T}))$.
In particular, any object $M\in\mathcal{D}$ fits into a triangle
$V\longrightarrow M\longrightarrow U\stackrel{+}{\longrightarrow}$,
where
$V\in\text{cosusp}_\mathcal{D}(\mathcal{T})[-1]\subset\text{thick}_\mathcal{D}(\mathcal{T})$
and
$U\in\text{susp}_\mathcal{D}(\mathcal{T})\subset\text{thick}_\mathcal{D}(\mathcal{T})$.
It follows that $\mathcal{D}=\text{thick}_\mathcal{D}(\mathcal{T})$,
so that $\mathcal{T}$ is a classical silting set. The fact that if
$\mathcal{T}$ is a classical silting set in $\mathcal{D}$, then
$(\text{cosusp}_\mathcal{D}(\mathcal{T}),\text{susp}_\mathcal{D}(\mathcal{T}))$
is a bounded co-t-structure is well-known (see \cite[Theorem 4.3.2
(II.1)]{Bo}).

Finally, if $\mathcal{D}$ has a classical generator $G$ and
$\mathcal{T}$ is a silting set in $\mathcal{D}$, then 
$G\in\text{thick}_\mathcal{D}(\mathcal{T})$, which implies the
existence of a finite subset $\mathcal{T}_0\subseteq\mathcal{T}$
such that $G\in\text{thick}_\mathcal{D}(\mathcal{T}_0)$, so that  $\mathcal{D}=\text{thick}_\mathcal{D}(\mathcal{T}_0)$, and
hence $\mathcal{T}_0$ is a classical silting set. By \cite[Theorem 2.18]{AI},
we conclude that $\mathcal{T}_0=\mathcal{T}$ and 
$\hat{T}:=\coprod_{T\in\mathcal{T}}T$ is a classical silting object.
\end{proof}

 We point out the following consequence of the proof of last Proposition.

\begin{corollary} \label{cor.silting=weaklyprecovpreenv}
Let $\mathcal{D}$ be a skeletally small triangulated category with split idempotents and let $\mathcal{T}$ be  a classical silting set.  Then it  is weakly precovering and weakly preenveloping in $\mathcal{D}$. 
\end{corollary}

\begin{proposition} \label{prop.envelope in the aisle}
Let $\mathcal{D}$ be a  triangulated category with coproducts, let
$\mathcal{T}$ be a nonpositive set of compact objects and let
$(\mathcal{U}_\mathcal{T},\mathcal{U}_\mathcal{T}^\perp [1])=(
{}^\perp (\mathcal{T}^{\perp_{\leq 0}}),\mathcal{T}^{\perp_{<0}})$ be
the associated t-structure in $\mathcal{D}$. The following
assertions are equivalent:

\begin{enumerate}
\item  $\mathcal{T}$ is a weakly preenveloping set in $\mathcal{D}^c$
\item $({}^\perp\text{susp}_\mathcal{D}(\mathcal{T})[-1]\cap\mathcal{D}^c,\text{susp}_\mathcal{D}(\mathcal{T}))$ is a co-t-structure in $\mathcal{D}^c$.
\item For each $M\in\mathcal{D}^c$,  there is a triangle $V_M\longrightarrow M\longrightarrow U_M\stackrel{+}{\longrightarrow}$, where $U_M\in\mathcal{U}_\mathcal{T}$ and $V_M\in {}^\perp\mathcal{U}_\mathcal{T}$.
 \end{enumerate}
\end{proposition}
\begin{proof}

$1)\Longleftrightarrow 2)$ is just the equivalence
$1')\Longleftrightarrow 4)$ of Proposition \ref{prop.U-envelope}
applied to $\mathcal{D}^c$.

$2)\Longrightarrow 3)$ Let $M$ be compact and fix a triangle
$V_M\longrightarrow M\longrightarrow
U_M\stackrel{+}{\longrightarrow}$, where
$U_M\in\text{susp}_\mathcal{D}(\mathcal{T})$ and $V_M\in
{}^\perp\text{susp}_\mathcal{D}(\mathcal{T})$. We clearly have that
$U_M\in\mathcal{U}_\mathcal{T}$.  It remains to prove that
$\text{Hom}_\mathcal{D}(V_M,-)$ vanishes on
$\mathcal{U}_\mathcal{T}[1]$. But, by the proof of \cite[Theorem
12.2]{KN} (see also \cite[Theorem 2]{NSZ}), we know that if
$U\in\mathcal{U}_\mathcal{T}$ then it is the Milnor colimit
$U=\text{Mcolim}U_n$ of a sequence
$$U_0\stackrel{h_1}{\longrightarrow}U_1\stackrel{h_2}{\longrightarrow}\cdots\stackrel{h_n}{\longrightarrow}U_n\stackrel{h_{n+1}}{\longrightarrow},$$
where $U_0\in\text{Sum}(\mathcal{T})$ and
$\text{cone}(h_n)\in\text{Sum}(\mathcal{T})[n]$, for all $n>0$. The
compactness of $V_M$ gives 
$\varinjlim\text{Hom}_\mathcal{D}(V_M,U_n)\cong\text{Hom}_\mathcal{D}(V_M,U)=0$.

$3)\Longrightarrow 1)$  Let $M\in\mathcal{D}^c$ be arbitrary and let
$V_M\longrightarrow M\stackrel{f}{\longrightarrow}
U_M\stackrel{+}{\longrightarrow}$ be the triangle given by assertion
3. As mentioned above, we have a sequence of morphisms
$$U_0\stackrel{h_1}{\longrightarrow}U_1\stackrel{h_2}{\longrightarrow}\cdots\stackrel{h_n}{\longrightarrow}U_n\stackrel{h_{n+1}}{\longrightarrow},$$
where $U_0\in\text{Sum}(\mathcal{T})$ and
$\text{cone}(h_n)\in\text{Sum}(\mathcal{T})[n]$ for all $n>0$, such
that $U_M\cong\text{Mcolim}U_n$. Due to compactness of $M$, the canonical morphism $\varinjlim\text{Hom}_\mathcal{D}(M,U_n)\longrightarrow\text{Hom}_\mathcal{D}(M,U_M)$ is an isomorphism. Thus, there exists  $g:M\longrightarrow U_t$, for some $t\in\mathbb{N}$, such that $f$ factors in the form $f:M\stackrel{g}{\longrightarrow}U_t\stackrel{u_t}{\longrightarrow}U_M$, where $u_t$ is the canonical morphism into the Milnor colimit. It immediately follows that $g$ is  a $\mathcal{U}_\mathcal{T}$-preenvelope since so is $f$. But $\text{Hom}_\mathcal{D}(U_t,-)$ vanishes on $\text{Sum}(\mathcal{T}[k])$, for all $k>t$, hence,  $\text{Hom}_\mathcal{D}(M,-[k])_{| \mathcal{T}}=0$ for $k>t$.

Let us consider the sequences of morphisms $U_{s}\stackrel{1}{\longrightarrow}U_{s}\stackrel{1}{\longrightarrow}\cdots \longrightarrow U_{s}\stackrel{1}{\longrightarrow}\cdots$ and $U_{s}\stackrel{h_{s+1}}{\longrightarrow}U_{s+1}\stackrel{h_{s+2}}{\longrightarrow}\cdots\stackrel{h_{n}}{\longrightarrow}U_{n}\stackrel{h_{n+1}}{\longrightarrow}\cdots$.   There is a morphism of sequences $(U_s,1)\longrightarrow (U_n,h_n)$ that for $n\geq s$ is the map $h'_n:=h_n\circ \dots\circ h_{s+1}:U_s\longrightarrow U_n$ and for $n=s$ is $h'_s=1_{U_s}$. Thus, there is a triangle $U_s\stackrel{h'_n}{\longrightarrow}U_n\longrightarrow U'_n\stackrel{+}{\longrightarrow}$, for each $n\geq s$, where $U'_n\in\text{Add}(\mathcal{T}[s+1])\star\text{Add}(\mathcal{T}[s+2])\star \cdots\star\text{Add}(\mathcal{T}[n])$, for each $n\geq s$. Using  Verdier's $3\times 3$ lemma and \cite[Lemmas 1.6.6 and 7.1.1]{N}, we get a triangle $U_s=\text{Mcolim}(U_s,1)\stackrel{f'}{\longrightarrow} U_M = \text{Mcolim}(U_n,h_n) \longrightarrow U_{>s}\stackrel{+}{\longrightarrow}$, where  $U_{>s}$ fits into a triangle $\coprod_{n\geq s}U'_n\longrightarrow\coprod_{n\geq s}U'_n\longrightarrow U_{>s}\stackrel{+}{\longrightarrow}$. In particular,  $\text{Hom}_\mathcal{D}(M,U_{>s})=0$, and hence $f'_*:\text{Hom}_\mathcal{D}(M,U_s)\longrightarrow \text{Hom}_\mathcal{D}(M,U_M)$ is surjective, since $M$ is compact and $\text{Hom}_\mathcal{D}(M,U'_n[k])=0$ for each $k\geq 0$ and $n\geq s$. It follows that there exists a factorization $f:M\stackrel{f'}{\longrightarrow}U_s\longrightarrow U_M$ of the map $f$ of the previous paragraph, $f'$ is a $\mathcal{U}_\mathcal{T}$-preenvelope since so is $f$.

Let us consider finally the triangle
$U_{s-1}\stackrel{h_s}{\longrightarrow}U_s\stackrel{p}{\longrightarrow}\coprod_{i\in
I}T_i[s]\stackrel{+}{\longrightarrow}$. An argument similar to that
of the proof of implication $2')\Longrightarrow 1')$ in Proposition
\ref{prop.U-envelope} shows that the composition $p\circ
f':M\longrightarrow \coprod_{i\in
I}T_i[s]$ is an
$\text{Add}(\mathcal{T})[s]$-preenvelope. The compactness of $M$
gives a factorization  $p\circ
f':M\stackrel{\alpha}{\longrightarrow}\coprod_{i\in
F}T_i[s]\stackrel{\iota_F}{\longrightarrow}\coprod_{i\in I}T_i[s]$,
for some finite subset $F\subseteq I$, where $\iota_F$ is the
canonical section. It follows that $\alpha
:M\longrightarrow\coprod_{i\in F}T_i[s]$ is also an
$\text{Add}(\mathcal{T})[s]$-preenvelope and, hence, an
$\text{add}(\mathcal{T})[s]$-preenvelope.
\end{proof}

We can now deduce the following  consequence.

\begin{corollary} \label{cor.silting set is preenveloping}
Let $\mathcal{D}$ be a triangulated category with coproducts, and let $\mathcal{T}$ be a silting set in $\mathcal{D}$ consisting of compact objects, i.e. a classical silting set in $\mathcal{D}^c$. The following assertions hold:

\begin{enumerate}
\item $\mathcal{T}$ is a weakly preenveloping and weakly precovering set in $\mathcal{D}^c$.
\item The associated t-structure in $\mathcal{D}$ is $\tau_\mathcal{T}=(\mathcal{T}^{\perp_{>0}},\mathcal{T}^{\perp_{<0}})=(\text{Susp}_\mathcal{D}(\mathcal{T}),\mathcal{T}^{\perp_{<0}})$, and it has a left adjacent co-t-structure $({}^\perp\text{Susp}_\mathcal{D}(\mathcal{T})[1],\text{Susp}_\mathcal{D}(\mathcal{T}))$ which restricts to $\mathcal{D}^c$. 
\item $\text{Susp}_\mathcal{D}(\mathcal{T})\cap\mathcal{D}^c=\mathcal{T}^{\perp_{>0}}\cap\mathcal{D}^c=\text{susp}_\mathcal{D}(\mathcal{T})$ and ${}^\perp\text{Susp}_\mathcal{D}(\mathcal{T})[1]\cap\mathcal{D}^c=\text{cosusp}_\mathcal{D}(\mathcal{T})$.
\end{enumerate}
\end{corollary}
\begin{proof}
Assertion 1 is  a particular case of Corollary \ref{cor.silting=weaklyprecovpreenv}, and assertion 2 follows from   \cite[Corollary 7]{Pauk}.
We just need to prove assertion 3. Indeed the pair $({}^\perp\text{Susp}_\mathcal{D}(\mathcal{T})\cap\mathcal{D}^c,\text{Susp}_\mathcal{D}(\mathcal{T})\cap\mathcal{D}^c)$ is a pair of orthogonal subcategories of $\mathcal{D}^c$. But, by Proposition \ref{prop.bijection psilting-cotstructures} and its proof we know that $(\text{cosusp}_\mathcal{D}(\mathcal{T})[-1],\text{susp}_\mathcal{D}(\mathcal{T}))$ is a torsion pair in $\mathcal{D}^c$. Since we have inclusions $\text{cosusp}_\mathcal{D}(\mathcal{T})[-1]\subseteq {}^\perp\text{Susp}_\mathcal{D}(\mathcal{T})\cap\mathcal{D}^c$ and $\text{susp}_\mathcal{D}(\mathcal{T})\subseteq \text{Susp}_\mathcal{D}(\mathcal{T})\cap\mathcal{D}^c$, these inclusions are necessarily equalities.
\end{proof}

\section{Gluing partial silting sets} \label{sect.gluing partial silting}

In this section, we will give criteria for the gluing of partial silting t-structures to be a partial silting t-structure. 

\subsection{Sufficient condition}

\begin{lemma} \label{lem.generators of glued t-structure}

Let
\begin{equation}\label{recrec}
\begin{xymatrix}{\mathcal{Y} \ar[r]^{i_*}& \mathcal{D} \ar@<3ex>[l]_{i^!}\ar@<-3ex>[l]_{i^*}\ar[r]^{j^*} & \mathcal{X} \ar@<3ex>_{j_*}[l]\ar@<-3ex>_{j_!}[l]}
\end{xymatrix}
\end{equation}
be a recollement of triangulated categories, let $(\mathcal{X}',\mathcal{X}'')$ and $(\mathcal{Y}',\mathcal{Y}'')$ be torsion pairs in $\mathcal{X}$ and $\mathcal{Y}$, respectively, generated by classes of objects $\mathcal{S}_X\subseteq\mathcal{X}'$ and $\mathcal{S}_Y\subseteq\mathcal{Y}'$. The glued torsion pair $(\mathcal{D}',\mathcal{D}'')$ is generated by $j_!(\mathcal{S}_X)\cup i_*(\mathcal{S}_Y)$. 
\end{lemma}
\begin{proof}
An object $Z\in\mathcal{D}$ belongs to the class $(j_!(\mathcal{S}_X)\cup i_*(\mathcal{S}_Y))^{\perp}$ iff
$\text{Hom}_\mathcal{D}(j_!(S),Z)=0=\text{Hom}_\mathcal{D}(i_*(S'),Z)$,
for all objects $S\in\mathcal{S}_X$ and
$S'\in\mathcal{S}_Y$ iff
$\text{Hom}_\mathcal{X}(S,j^*(Z))=0=\text{Hom}_\mathcal{Y}(S',i^!(Z))$
for all   $S\in\mathcal{S}_X$ and $S'\in\mathcal{S}_Y$ iff $j^*(Z)\in\mathcal{X}''$ and
$i^!(Z)\in\mathcal{Y}''$ iff 
$Z\in\mathcal{D}''$.
\end{proof}

\begin{remark} \label{rem.aisle of X is aisle of D}
In the recollement (\ref{recrec}) if $(\mathcal{X}^{\leq
0},\mathcal{X}^{\geq 0})$ is a t-structure in $\mathcal{X}$, then
$j_!(\mathcal{X}^{\leq 0})$ is the aisle of a t-structure in
$\mathcal{D}$. Indeed $j_!(\mathcal{X}^{\leq 0})=\{D\in\mathcal{D} \mid j^*D\in\mathcal{X}^{\leq 0}, i^*D=0\}$ and  
so it is the aisle of the t-structure in
$\mathcal{D}$ glued from $(\mathcal{X}^{\leq
0},\mathcal{X}^{\geq 0})$ and the trivial t-structure
$(0,\mathcal{Y})$ in $\mathcal{Y}$.
\end{remark}

We are ready to prove the technical criteria on which all further results of this section are based. The main applications will be given in Corollaries \ref{cor.gluing wrt bounded subcategories} and  \ref{cor.Liu-Vitoria-Yang-generalized}.

\begin{theorem} \label{thm.gluing of silting t-structures}
Let

$$
\begin{xymatrix}{\mathcal{Y} \ar[r]^{i_*}& \mathcal{D} \ar@<3ex>[l]_{i^!}\ar@<-3ex>[l]_{i^*}\ar[r]^{j^*} & \mathcal{X} \ar@<3ex>_{j_*}[l]\ar@<-3ex>_{j_!}[l]}
\end{xymatrix}
$$ be a recollement of  triangulated categories,  let $\mathcal{T}_X$
and $\mathcal{T}_Y$ be  (partial) silting sets in
$\mathcal{X}$ and $\mathcal{Y}$, let  $(\mathcal{X}^{\leq
0},\mathcal{X}^{\geq 0})$, $(\mathcal{Y}^{\leq 0},\mathcal{Y}^{\geq
0})$ be the associated t-structures in $\mathcal{X}$ and
$\mathcal{Y}$ and let $(\mathcal{D}^{\leq 0},\mathcal{D}^{\geq 0})$
be the glued t-structure. Suppose that the following condition
holds:

$(\star)$ For each object $T_Y\in\mathcal{T}_Y$, there is
a triangle $\tilde{T}_Y\longrightarrow
i_*T_Y\stackrel{f_{T_Y}}{\longrightarrow}U_{T_Y}[1]\stackrel{+}{\longrightarrow}$
such that $U_{T_Y}\in j_!(\mathcal{X}^{\leq 0})$ and $\tilde{T}_Y\in
{}^\perp j_!(\mathcal{X}^{\leq 0})[1]$.

Then for $\tilde{\mathcal{T}_Y}:=\{\tilde{T}_Y\text{:
}T_Y\in\mathcal{T}_Y\}$ the set $j_!(\mathcal{T}_X)\cup\tilde{\mathcal{T}_Y}$ is a  (partial) silting set in
$\mathcal{D}$ which generates $(\mathcal{D}^{\leq
0},\mathcal{D}^{\geq 0})$. 
\end{theorem}
\begin{proof}
Let us prove that
$\mathcal{T}:=j_!(\mathcal{T}_X)\cup\tilde{T}_Y$ generates
$(\mathcal{D}^{\leq 0},\mathcal{D}^{\geq 0})$. By Lemma
\ref{lem.generators of glued t-structure}, the
class  $j_!\mathcal{X}^{{\leq 0}}\cup i_*(\mathcal{T}_Y)$ generates
$(\mathcal{D}^{\leq 0},\mathcal{D}^{\geq 0})$, so 
$\mathcal{D}^{>0}=(j_!\mathcal{X}^{{\leq 0}}\cup
i_*(\mathcal{T}_Y))^{\perp_{\leq 0}}$. Condition $(\star)$ implies that $(j_!\mathcal{X}^{\leq 0}\cup
i_*(\mathcal{T}_Y))^{\perp_{\leq 0}}=(j_!\mathcal{X}^{\leq
0}\cup\tilde{\mathcal{T}}_Y)^{\perp_{\leq 0}}$. Note that  $(j_!\mathcal{X}^{\leq 0})^\perp\subseteq
(j_!(\mathcal{T}_X))^{\perp_{\leq 0}}$. Conversely, if $Z\in
(j_!(\mathcal{T}_X))^{\perp_{\leq 0}}$ then
$\text{Hom}_\mathcal{X}(T_X[k],j^*Z)\cong\text{Hom}_\mathcal{D}(j_!T_X[k],Z)=0$,
for all $T_X\in\mathcal{T}_X$ and all integers $k\geq 0$. Then $j^*Z\in\mathcal{X}^{>0}$, and so
$\text{Hom}_\mathcal{D}(j_!X,Z)\cong\text{Hom}_\mathcal{X}(X,j^*Z)=0$,
for all $X\in\mathcal{X}^{{\leq 0}}$ since $\mathcal{T}_X$
generates $(\mathcal{X}^{\leq 0},\mathcal{X}^{\geq 0})$. Thus, there is an equality $(j_!\mathcal{X}^{\leq 0})^\perp
=(j_!(\mathcal{T}_X))^{\perp_{\leq 0}}$. It follows that
$\mathcal{D}^{>0}=(j_!\mathcal{X}^{\leq
0}\cup\tilde{\mathcal{T}}_Y)^{\perp_{\leq 0}}=(j_!\mathcal{X}^{\leq
0})^{\perp_{\leq 0}}\cap\tilde{\mathcal{T}}_Y^{\perp_{\leq
0}}=(j_!\mathcal{X}^{\leq
0})^\perp\cap\tilde{\mathcal{T}}_Y^{\perp_{\leq 0}}=(j_!(\mathcal{T}_X))^{\perp_{\leq
0}}\cap\tilde{\mathcal{T}}_Y^{\perp_{\leq
0}}=(j_!(\mathcal{T}_X)\cup\tilde{\mathcal{T}}_Y)^{\perp_{\leq 0}}$.
Therefore $\mathcal{T}:=j_!(\mathcal{T}_X)\cup\tilde{\mathcal{T}}_Y$
generates $(\mathcal{D}^{\leq 0},\mathcal{D}^{\geq 0})$.

In order to prove that $\mathcal{T}$ is partial silting, we
need to check that the functors $\text{Hom}_\mathcal{D}(j_!T_X,-)$
and $\text{Hom}_\mathcal{D}(\tilde{T}_Y,-)$ vanish on
$\mathcal{D}^{<0}:=\mathcal{D}^{\leq 0}[1]$, for all
$T_X\in\mathcal{T}_X$ and $T_Y\in\mathcal{T}_Y$. Consider
$D\in\mathcal{D}^{\leq 0}$. Due to adjunction, 
$\text{Hom}_\mathcal{D}(j_!T_X,D[1])\cong\text{Hom}_\mathcal{X}(T_X,j^*D[1])$,
for all $T_X\in\mathcal{T}_X$. But 
$j^*D\in\mathcal{X}^{\leq 0}$, so the partial silting
condition on $\mathcal{T}_X$ gives 
$\text{Hom}_\mathcal{X}(T_X,j^*D[1])=0$, for all $T_X\in\mathcal{T}_X$.
On the other hand, $j^*D\in\mathcal{X}^{\leq 0}$ and,
by definition of $\tilde{T}_Y$, we have $\tilde{T}_Y\in
{}^\perp (j_!\mathcal{X}^{\leq 0}[1])$, for each
$T_Y\in\mathcal{T}_Y$.  Then
$\text{Hom}_\mathcal{D}(\tilde{T}_Y,D[1])\cong\text{Hom}_\mathcal{D}(\tilde{T}_Y,i_*i^*D[1])\cong\text{Hom}_\mathcal{Y}(i^*\tilde{T}_Y,i^*D[1])$.
The equality  $i^*\circ j_!=0$ implies that  $i^*$ vanishes on
$j_!\mathcal{X}^{\leq 0}[k]$,  for all $k\in\mathbb{Z}$, and hence $i^*\tilde{T}_Y\cong i^*i_*T_Y\cong T_Y$. Thus,
$\text{Hom}_\mathcal{D}(\tilde{T}_Y,D[1])\cong\text{Hom}_\mathcal{Y}(T_Y,i^*D[1])$,
for each $T_Y\in\mathcal{T}_Y$. Since, by definition of
$\mathcal{D}^{\leq 0}$, we know that $i^*D\in\mathcal{Y}^{\leq 0}$,
the partial silting condition on $\mathcal{T}_Y$ implies that
$\text{Hom}_\mathcal{D}(\tilde{T}_Y,D[1])\cong\text{Hom}_\mathcal{Y}(T_Y,i^*D[1])=0$,
for all $T_Y\in\mathcal{T}_Y$. Therefore $\mathcal{T}:=j_!(\mathcal{T}_X)\cup\tilde{\mathcal{T}}_Y$ is 
partial silting in $\mathcal{D}$.

Assume now that $\mathcal{T}_X$ and $\mathcal{T}_Y$ generate $\mathcal{X}$ and $\mathcal{Y}$, respectively. Let us prove that
$\mathcal{T}=j_!(\mathcal{T}_X)\cup\tilde{T}_Y$ generates
$\mathcal{D}$. Let $Z\in\mathcal{D}$ be an object such that
$\text{Hom}_\mathcal{D}(T[k],Z)=0$, for all $T\in\mathcal{T}$, $k
\in \mathbb{Z}$. Then 
$\text{Hom}_\mathcal{X}(T_X[k],j^*Z)\cong\text{Hom}_\mathcal{D}(j_!T_X[k],Z)=0$,
for all $k\in\mathbb{Z}$ and  $T_X\in\mathcal{T}_X$. Since $\mathcal{T}_X$ generates $\mathcal{X}$, we get $j^*Z=0$. In
particular, $\text{Hom}_\mathcal{D}(j_!X[k],Z)=0$, for
all $X\in\mathcal{X}^{\leq 0}$. Condition $(\star)$ implies
$\text{Hom}_\mathcal{Y}(T_Y[k],i^!Z)\cong\text{Hom}_\mathcal{D}(i_*T_Y[k],Z)=0$,
for all $k\in\mathbb{Z}$ and all $T_Y\in\mathcal{T}_Y$. The
generating condition on $\mathcal{T}_Y$ gives $i^!Z=0$.
From the canonical triangle $i_*i^!Z\longrightarrow
Z\longrightarrow j_*j^*Z\stackrel{+}{\longrightarrow}$, we conclude
that $Z=0$. Therefore $j_!(\mathcal{T}_X)\cup\tilde{T}_Y$ generates $\mathcal{D}$, and hence is a  silting
set in $\mathcal{D}$.
\end{proof}

It is natural to ask when condition $(\star)$ from Theorem \ref{thm.gluing of silting t-structures} holds. The following result is useful.

\begin{lemma}\label{lem.adjacent gluing}
Let $\mathcal{Y}, \mathcal{D}, \mathcal{X}$ be triangulated
categories and let $\mathcal{L}$ be a ladder of recollements of
height two

\begin{equation}\label{ladder2}
\begin{xymatrix}{
\mathcal{Y} \ar[r]^{i_*} \ar@<6ex>[r]^{i_{\sharp}}& \mathcal{D} \ar@<3ex>[l]_{i^!} \ar@<-3ex>[l]_{i^*} \ar@<6ex>[r]^{j^{\sharp}} \ar[r]^{j^*} & \mathcal{X} \ar@<3ex>[l]_{j_*} \ar@<-3ex>[l]_{j_!}\\
}
\end{xymatrix}
\end{equation}
Let $(\mathcal{X}_1,\mathcal{X}_1')$ and $(\mathcal{X}_2,\mathcal{X}_2')$ be adjacent torsion pairs in $\mathcal{X}$ and let $(\mathcal{Y}_1,\mathcal{Y}_1')$ and $(\mathcal{Y}_2,\mathcal{Y}_2')$ be adjacent torsion pairs in $\mathcal{Y}$. Consider the torsion pairs $(\mathcal{D}_i,\mathcal{D}'_i)$ ($i=1,2$) defined as follows:

\begin{enumerate}
\item   $(\mathcal{D}_1,\mathcal{D}_1')$ is the torsion pair in $\mathcal{D}$  glued from $(\mathcal{X}_1,\mathcal{X}_1')$ and $(\mathcal{Y}_1,\mathcal{Y}_1')$
 with respect to the upper recollement of the ladder.
 \item  $(\mathcal{D}_2,\mathcal{D}_2')$ is the torsion pair in $\mathcal{D}$  glued from  $(\mathcal{X}_2,\mathcal{X}_2')$ and $(\mathcal{Y}_2,\mathcal{Y}_2')$
 with respect to the lower recollement of the ladder.
\end{enumerate}

Then $(\mathcal{D}_1,\mathcal{D}_1')$ and $(\mathcal{D}_2,\mathcal{D}_2')$ are adjacent torsion pairs in $\mathcal{D}$. 
\end{lemma}

\begin{proof}
By the gluing procedure we get: $\mathcal{D}_2=\{Z \in
\mathcal{D} \mid j^*Z \in\mathcal{X}_2, i^*Z \in
\mathcal{Y}_2 \}=\{Z \in \mathcal{D}
\mid i^*Z \in \mathcal{Y}'_1, j^*Z \in \mathcal{X}'_1
\}=\mathcal{D}'_1$.
\end{proof}

\begin{corollary} \label{cor.upper layer of the recollement}
Let $\mathcal{L}$ be a ladder of recollements as in Lemma \ref{lem.adjacent gluing}, and let $\mathcal{T}_X$ and $\mathcal{T}_Y$ be partial silting sets which generate t-structures $\tau_X=(\mathcal{X}^{\leq 0},\mathcal{X}^{\geq 0})$ and $\tau_Y=(\mathcal{Y}^{\leq 0},\mathcal{Y}^{\geq 0})$ in $\mathcal{X}$ and $\mathcal{Y}$, respectively. If $\tau_X$ has a left adjacent co-t-structure in $\mathcal{X}$, then condition $(\star)$ of Theorem \ref{thm.gluing of silting t-structures} holds, so  $j_!(\mathcal{T}_X)\cup\tilde{\mathcal{T}}_Y$ is a partial silting set which generates the  t-structure $(\mathcal{D}^{\leq 0},\mathcal{D}^{\geq 0})$ in $\mathcal{D}$ glued with respect to the lower recollement of the ladder. 
\end{corollary}
\begin{proof}
Gluing the co-t-structures $({}^\perp(\mathcal{X}^{\leq 0})[-1],\mathcal{X}^{\leq 0})$ and $(\mathcal{Y},0)$ with respect to the upper recollement of the ladder, we obtain a co-t-structure in $\mathcal{D}$ whose right component is $j_!(\mathcal{X}^{\leq 0})$. Then condition $(\star)$ of Theorem \ref{thm.gluing of silting t-structures} holds.
\end{proof}

\begin{example} \label{ex.}
In the situation of the last corollary, let $\mathcal{X}$ have coproducts,  let $\mathcal{T}_X$ be a silting  set in  $\mathcal{X}$ with the associated t-structure $\tau_X:=(\mathcal{X}^{\leq 0},\mathcal{X}^{\geq 0})$. If either of the following  conditions holds, then $\tau_X$ has a left adjacent co-t-structure:

\begin{enumerate}
\item $\mathcal{X}$ is the stable category of an
efficient Frobenius exact category with coproducts in the terminology of \cite{SS};
\item $\mathcal{T}_X$ consists of compact objects. 
\end{enumerate}
(See \cite[Theorem 4.3.1]{Bo2} for a more general condition than 2 where the argument below also  works).
\end{example}
\begin{proof}
By \cite[Theorem 1]{NSZ} (see also \cite{PV}), we have $(\mathcal{X}^{\leq 0},\mathcal{X}^{\geq 0})=(\mathcal{T}_X^{\perp_{>0}},\mathcal{T}_X^{\perp_{<0}})$. Then the proof reduces to check that  $({}^\perp (\mathcal{T}_X^{\perp_{> 0}}),\mathcal{T}_X^{\perp_{>0}})$  is a torsion pair in $\mathcal{X}$. Under condition 1, this follows from  \cite[Corollary 3.5]{SS}. Under condition 2 it follows from \cite[Theorem 4.3]{AI} or from Corollary \ref{cor.silting set is preenveloping}.
\end{proof}

\begin{remark} \label{rem.results dualize}
The following diagram is a recollement of triangulated categories \\
$
\begin{xymatrix}{\mathcal{Y} \ar[r]^{i_*}& \mathcal{D} \ar@<3ex>[l]_{i^!}\ar@<-3ex>[l]_{i^*}\ar[r]^{j^*} & \mathcal{X} \ar@<3ex>_{j_*}[l]\ar@<-3ex>_{j_!}[l]}
\end{xymatrix}
$ if and only if   so is $
\begin{xymatrix}{\mathcal{Y}^{op} \ar[r]^{i_*}& \mathcal{D}^{op} \ar@<3ex>[l]_{i^*}\ar@<-3ex>[l]_{i^!}\ar[r]^{j^*} & \mathcal{X}^{op}. \ar@<3ex>_{j_!}[l]\ar@<-3ex>_{j_*}[l]}
\end{xymatrix}
 $ As a consequence, after defining  \emph{(partial) cosilting  set}  as the dual of (partial) silting  set, 
 many results in this section admit a dualization. We leave their statement to the reader. 
\end{remark}

\subsection{Gluing partial silting sets of compact objects}

When some of the functors in a recollement preserve compact objects, we can approach condition $(\star)$ of Theorem \ref{thm.gluing of silting t-structures} on the compact level.

\begin{setup} \label{setup}
 In this subsection we consider:

\begin{enumerate}
 \item A recollement \begin{equation}\label{rec}
\begin{xymatrix}{\mathcal{Y} \ar[r]^{i_*}& \mathcal{D} \ar@<3ex>[l]_{i^!}\ar@<-3ex>[l]_{i^*}\ar[r]^{j^*} & \mathcal{X} \ar@<3ex>_{j_*}[l]\ar@<-3ex>_{j_!}[l]},
\end{xymatrix}
\end{equation}   where $\mathcal{Y}$, $\mathcal{D}$ and $\mathcal{X}$ are thick subcategories of triangulated categories with coproducts $\hat{\mathcal{Y}}$, $\hat{\mathcal{D}}$ and $\hat{\mathcal{X}}$ which contain the corresponding subcategories of compact objects.
\item  Partial silting sets $\mathcal{T}_X$ and $\mathcal{T}_Y$ in $\mathcal{X}$ and $\mathcal{Y}$, respectively, consisting of compact objects, and the t-structures  $(\mathcal{X}^{\leq 0},\mathcal{X}^{\geq 0})$ and $(\mathcal{Y}^{\leq 0},\mathcal{Y}^{\geq 0})$ in $\mathcal{X}$ and $\mathcal{Y}$, generated by $\mathcal{T}_X$ and $\mathcal{T}_Y$.
\item $j_!(\mathcal{T}_X)\cup i_*(\mathcal{T}_Y)\subset\hat{\mathcal{D}}^c$ and $j_!(\mathcal{T}_X)$ weakly preenveloping in $\hat{\mathcal{D}}^c$.
\end{enumerate}

\end{setup}

The following are examples of weakly preenveloping sets of compact objects.

\begin{example} \label{ex.finite-is-weaklypreenveloping}
Let $\hat{\mathcal{D}}$ be a compactly generated triangulated category. Under either of the
following conditions, the set $\mathcal{T}$ is weakly preenveloping in $\hat{\mathcal{D}}^c$:
\begin{enumerate}
\item $\hat{\mathcal{D}}$ is homologically  locally bounded,  $\text{Hom}_{\hat{\mathcal{D}}}(M,N)$ is a finitely generated
$K$-module, for all $M,N\in\hat{\mathcal{D}}^c$, and  $\mathcal{T}$ is a finite non-positive set in $\hat{\mathcal{D}^{c}. }$
\item $\mathcal{T}$ is classical silting in $\hat{\mathcal{D}}^c$. 
\end{enumerate}
\end{example}
\begin{proof}
Example 2 follows from Corollary \ref{cor.silting=weaklyprecovpreenv}. As for example 1, using the
  finiteness of $\mathcal{T}$, we can assume that $\mathcal{T}=\{T\}$. The homological local boundedness of $\hat{\mathcal{D}}$ implies that, for each $M\in\hat{\mathcal{D}}^c$, one 
has $\text{Hom}_{\hat{\mathcal{D}}}(M,T[k])=0$, for $k\gg 0$. Moreover, if $s=s(M,T)=Sup\{k\in \mathbb{N}\text{: }\text{Hom}_{\hat{\mathcal{D}}}(M,T[k])\neq 0\}$, then $M$ has an 
$\text{add}(T)[s]$-preenvelope because $\text{Hom}_{\hat{\mathcal{D}}}(M,T[s])$ is finitely generated as a $K$-module.\end{proof}

In Setup \ref{setup}, there exists a triangle $(\dagger):$ 
${\tilde{T}_Y\longrightarrow i_*T_Y\longrightarrow U_{T_Y}[1]\stackrel{+}{\longrightarrow}}$  in $\hat{\mathcal{D}}^c$, where 
$U_{T_Y}\in\text{susp}_{\hat{\mathcal{D}}}(j_!(\mathcal{T}_X))$ and $\tilde{T}_Y\in {}^\perp\text{Susp}_{\hat{\mathcal{D}}}(j_!(\mathcal{T}_X))[1]$,
 for each $T_Y\in\mathcal{T}_Y$ (see Proposition \ref{prop.envelope in the aisle}). The natural question is: 
does condition $(\star)$ of Theorem \ref{thm.gluing of silting t-structures} hold? Our next result gives a partial answer. 

\begin{theorem} \label{thm.gluing at compact level}
In  Setup \ref{setup}, if  $j_!(\mathcal{X}^{\leq 0})\subseteq\text{Susp}_{\hat{\mathcal{D}}}(j_!(\mathcal{T}_X))$, then condition $(\star)$ of  Theorem \ref{thm.gluing of silting t-structures} holds and, for $\tilde{\mathcal{T}}_Y=\{ \tilde{T}_Y\text{: }T_Y\in\mathcal{T}_Y\}$,  the set  $\mathcal{T}:=j_!(\mathcal{T}_X)\cup\tilde{\mathcal{T}}_Y \subseteq \hat{\mathcal{D}}^c$ is partial silting in $\mathcal{D}$ and generates the glued t-structure $(\mathcal{D}^{\leq 0},\mathcal{D}^{\geq 0})$. 

The inclusion $j_!(\mathcal{X}^{\leq 0})\subseteq\text{Susp}_{\hat{\mathcal{D}}}(j_!(\mathcal{T}_X))$ holds if, in addition to Setup \ref{setup}, either one of the following conditions hold:

\begin{enumerate}
\item Recollement (\ref{rec}) is the restriction of a recollement 
$$
\begin{xymatrix}{\hat{\mathcal{Y}} \ar[r]^{\hat{i}_*}& \hat{\mathcal{D}} \ar@<3ex>[l]_{\hat{i}^!}\ar@<-3ex>[l]_{\hat{i}^*}\ar[r]^{\hat{j}^*} & \hat{\mathcal{X}} \ar@<3ex>_{\hat{j}_*}[l]\ar@<-3ex>_{\hat{j}_!}[l]}
\end{xymatrix},$$ and $(\mathcal{X}^{\leq 0},\mathcal{X}^{\geq 0})$ is the restriction of the t-structure in $\hat{\mathcal{X}}$  generated by $\mathcal{T}_X$. 
\item The triangulated categories $\hat{\mathcal{Y}}$, $\hat{\mathcal{D}}$ and $\hat{\mathcal{X}}$ are compactly generated,  
$\mathcal{T}_X$ is a classical silting set in $\hat{\mathcal{X}}^c$,  the functors $j_!$, $j^*$, $i_*$ and $i^*$ preserve compact objects 
and either $\text{Im}(i_*)$ cogenerates $\text{Loc}_{\hat{\mathcal{D}}}(i_*(\hat{\mathcal{Y}}^c))$ or $\mathcal{D}$ cogenerates 
$\hat{\mathcal{D}}$.
\end{enumerate}
\end{theorem}
\begin{proof}
Let us consider the triangle $(\dagger):$ $\tilde{T}_Y\longrightarrow i_*T_Y\longrightarrow U_{T_Y}[1]\stackrel{+}{\longrightarrow}$. 
 The inclusion  $j_!(\mathcal{X}^{\leq 0})\subseteq\text{Susp}_{\hat{\mathcal{D}}}(j_!(\mathcal{T}_X))$ implies  that 
$\tilde{T}_Y\in {}^\perp j_!(\mathcal{X}^{\leq 0})[1]$, and hence condition $(\star)$ of Theorem \ref{thm.gluing of silting t-structures} 
holds.

Let us check that $j_!(\mathcal{X}^{\leq 0})\subseteq\text{Susp}_{\hat{\mathcal{D}}}(j_!(\mathcal{T}_X))$ under conditions $1$ or $2$.

1) Since $\hat{j}_!:\hat{\mathcal{X}}\longrightarrow\hat{\mathcal{D}}$ has a right adjoint, it preserves Milnor colimits. Moreover, since $(\mathcal{X}^{\leq 0},\mathcal{X}^{\geq 0})=(\hat{\mathcal{X}}^{\leq 0}\cap\mathcal{X},\hat{\mathcal{X}}^{\geq 0}\cap\mathcal{X})$, where $(\hat{\mathcal{X}}^{\leq 0},\hat{\mathcal{X}}^{\geq 0})=({}^\perp (\mathcal{T}_X^{\perp_{\leq 0}}),\mathcal{T}_X^{\perp_{<0}})$,  each object $X$ of $\mathcal{X}^{\leq 0}$ 
is the Milnor colimit of a sequence 
$$X_0\stackrel{f_1}{\longrightarrow}X_1\stackrel{f_2}{\longrightarrow}\cdots\stackrel{f_n}{\longrightarrow}X_n\stackrel{f_{n+1}}{\longrightarrow}\cdots,$$ 
where $X_0\in\text{Add}(\mathcal{T}_X)$ and $\text{cone}(f_n)\in\text{Add}(\mathcal{T}_X)[n]$, for each $n>0$ (see \cite[Theorem 2]{NSZ}). 
Thus, $j_!(X)\in\text{Susp}_{\hat{\mathcal{D}}}(j_!(\mathcal{T}_X))$, for each $X\in\mathcal{X}^{\leq 0}$. 

2) By Theorem \ref{thm.lifting of recollements},  $(\text{Loc}_{\hat{\mathcal{D}}}(j_!(\mathcal{X}^c)),\text{Loc}_{\hat{\mathcal{D}}}(i_*(\mathcal{Y})^c),\text{Loc}_{\hat{\mathcal{D}}}(i_*(\mathcal{Y}^c))^\perp ) =:(\mathcal{U},\mathcal{V},\mathcal{W})$ is a TTF triple in $\hat{\mathcal{D}}$ such that $(\mathcal{U}\cap\mathcal{D},\mathcal{V}\cap\mathcal{D},\mathcal{W}\cap\mathcal{D})=(\text{Im}(j_!),\text{Im}(i_*),\text{Im}(j_!))$ and it satisfies conditions 2.a and 2.b of that theorem. We then get an associated recollement $$
\begin{xymatrix}{\mathcal{V} \ar[r]^{\hat{i}_*}& \hat{\mathcal{D}} \ar@<3ex>[l]_{\hat{i}^!}\ar@<-3ex>[l]_{\hat{i}^*}\ar[r]^{\hat{j}^*} & \mathcal{U} \ar@<3ex>_{\hat{j}_*}[l]\ar@<-3ex>_{\hat{j}_!}[l]}
\end{xymatrix},$$ which restricts up to equivalence to the recollement 
\begin{equation}\label{Recd}
\begin{xymatrix}{\text{Im}(i_*) \ar[r]^-{i_*}& \mathcal{D} \ar@<3ex>[l]_-{i^!}\ar@<-3ex>[l]_-{i^*}\ar[r]^-{j^*} & \text{Im}(j_!) \ar@<3ex>_-{j_*}[l]\ar@<-3ex>_-{j_!}[l]}
\end{xymatrix}.
\end{equation}
 Recollement (\ref{Recd}) is equivalent to the one in Setup \ref{setup}
 via the equivalences of triangulated categories $i_*:\mathcal{Y}\stackrel{\cong}{\longrightarrow}\text{Im}(i_*)$ and $j_!:\mathcal{X}\stackrel{\cong}{\longrightarrow}\text{Im}(j_!)$. 

{By Proposition \ref{prop.restricted partial silting}, we know that the t-structure $(\mathcal{T}_X^{\perp_{>0}},\mathcal{T}_X^{\perp_{<0}})$ 
in $\hat{\mathcal{X}}$ restricts to $\mathcal{X}$, and so 
$(\mathcal{X}^{\leq 0},\mathcal{X}^{\geq 0})=(\mathcal{T}_X^{\perp_{>0}}\cap\mathcal{X},\mathcal{T}_X^{\perp_{<0}}\cap\mathcal{X}):=
(\mathcal{T}_X^{\perp_{>0} (\mathcal{X})},\mathcal{T}_X^{\perp_{<0} (\mathcal{X})})$. Using the 
equivalence $j_!:\mathcal{X}\stackrel{\cong}{\longrightarrow}\text{Im}(j_!)=\mathcal{U}\cap\mathcal{D}$, we get that 
$(j_!(\mathcal{X}^{\leq 0}),j_!(\mathcal{X}^{\geq 0}))=
(j_!(\mathcal{T}_X)^{\perp_{>0} (\mathcal{U}\cap\mathcal{D})},j_!(\mathcal{T}_X)^{\perp_{<0} (\mathcal{U}\cap\mathcal{D})})$, which is the 
restriction to $\mathcal{U}\cap\mathcal{D}$ of the pair 
$(j_!(\mathcal{T}_X)^{\perp_{>0} (\mathcal{U})},j_!(\mathcal{T}_X)^{\perp_{<0} (\mathcal{U})})$. The equality 
$\hat{\mathcal{X}}^c=\text{thick}_\mathcal{D}(\mathcal{T}_X)$ gives the equality 
$\mathcal{U}^c=\mathcal{U}\cap\hat{\mathcal{D}}^c=j_!(\hat{\mathcal{X}}^c)=\text{thick}_{\mathcal{U}\cap\mathcal{D}}(j_!(\mathcal{T}_X))=
\text{thick}_\mathcal{U}(j_!(\mathcal{T}_X))$, so  $j_!(\mathcal{T}_X)$ is a silting set of compact objects in $\mathcal{U}$ and  
$(j_!(\mathcal{T}_X)^{\perp_{>0} },j_!(\mathcal{T}_X)^{\perp_{<0}})$ is a t-structure in $\mathcal{U}$ which restricts to 
$(j_!(\mathcal{X}^{\leq 0}),j_!(\mathcal{X}^{\geq 0}))$.}

We have checked that condition (1) holds for the recollement (\ref{Recd}), and hence the inclusion $j_!(\mathcal{X}^{\leq 0})\subseteq\text{Susp}_{\hat{\mathcal{D}}}(j_!(\mathcal{T}_X))$ holds by the first part of the proof, since the corresponding functor in (\ref{Recd}) is the inclusion.
\end{proof}

Several consequences of the last theorem (condition 2) and the results in Section \ref{section.some points} can be obtained. For the sake of brevity, we provide two of them.

\begin{corollary} \label{cor.gluing wrt bounded subcategories}
Let $\hat{\mathcal{Y}}$, $\hat{\mathcal{D}}$ and $\hat{\mathcal{X}}$ be compactly generated  triangulated categories, where 
$\hat{\mathcal{D}}$ is homologically locally bounded and such that $\text{Hom}_{\hat{D}}(M,N)$ is finitely generated (resp. of finite length)
 as a $K$-module, for all $M,N\in\hat{\mathcal{D}}^c$. Let $$
\begin{xymatrix}{\hat{\mathcal{Y}}^{\star}_\dagger \ar[r]^{{i}_*}& \hat{\mathcal{D}}^{\star}_\dagger \ar@<3ex>[l]_{{i}^!}\ar@<-3ex>[l]_{{i}^*}\ar[r]^{{j}^*} & \hat{\mathcal{X}}^{\star}_\dagger \ar@<3ex>_{{j}_*}[l]\ar@<-3ex>_{{j}_!}[l]}
\end{xymatrix}$$ be a recollement, such that the subcategories involved contain the respective subcategories of compact objects and the 
functors $j_!$, $j^*$, $i^*$ and $i_*$ preserve compact objects. Let $\mathcal{T}_X$ and $\mathcal{T}_Y$ be partial silting sets in 
$\hat{\mathcal{X}}^{\star}_\dagger$ and $\hat{\mathcal{Y}}^{\star}_\dagger$ consisting of compact objects, with $\mathcal{T}_X$  finite and  classical silting in 
$\hat{\mathcal{X}}^c$,  and let $(\mathcal{X}^{\leq 0},\mathcal{X}^{\geq 0})$ and $(\mathcal{Y}^{\leq 0},\mathcal{Y}^{\geq 0})$ be the 
associated t-structures in $\hat{\mathcal{X}}^{\star}_\dagger$ and $\hat{\mathcal{Y}}^{\star}_\dagger$, respectively. Then condition $(\star)$ of Theorem 
\ref{thm.gluing of silting t-structures} holds for   $\star\in\{ \emptyset,+,-,b\}$  
and $\dagger =\emptyset$ (resp. $\dagger =fl$). Therefore $j_!(\mathcal{T}_X)\cup\tilde{\mathcal{T}}_Y$ is a partial silting set in $\hat{\mathcal{D}}^{\star}_\dagger$, consisting of compact objects,  which generates the glued t-structure.
\end{corollary}
\begin{proof} Let us check that the conditions from Setup \ref{setup} hold. For this we only need to check that $j_!(\mathcal{T}_X)$ is weakly 
preenveloping in 
$\hat{\mathcal{D}}^c$. This follows from Example \ref{ex.finite-is-weaklypreenveloping}. 
Thus, the result is a direct consequence of Theorem \ref{thm.gluing at compact level} (condition 2) and 
Corollary \ref{cor.Case D homologically loc. bounded}.
\end{proof}

\begin{corollary} \label{cor.Liu-Vitoria-Yang-generalized}
Let $\hat{\mathcal{Y}}$, $\hat{\mathcal{D}}$ and $\hat{\mathcal{X}}$ be compactly generated algebraic triangulated categories and let 
\begin{equation}\label{Rece}
\begin{xymatrix}{\hat{\mathcal{Y}}^b_{fl} \ar[r]^{{i}_*}& \hat{\mathcal{D}}^b_{fl} \ar@<3ex>[l]_{{i}^!}\ar@<-3ex>[l]_{{i}^*}\ar[r]^{{j}^*} & \hat{\mathcal{X}}^b_{fl} \ar@<3ex>_{{j}_*}[l]\ar@<-3ex>_{{j}_!}[l]}
\end{xymatrix}
\end{equation}
 be a recollement, such that the subcategories involved contain the respective subcategories of compact 
objects and  $\hat{\mathcal{D}}$ is compact-detectable in finite length.
 Let $T_X\in\hat{\mathcal{X}}^c$ and $T_Y\in\hat{\mathcal{Y}}^c$ be  classical silting in $\hat{\mathcal{X}}^c$ and $\hat{\mathcal{Y}}^c$, respectively. Let $(\mathcal{X}^{\leq 0},\mathcal{X}^{\geq 0})$ and
 $(\mathcal{Y}^{\leq 0},\mathcal{Y}^{\geq 0})$ be the corresponding t-structures in $\hat{\mathcal{X}}^b_{fl}$ and $\hat{\mathcal{Y}}^b_{fl}$. There is a triangle $\tilde{T}_Y\longrightarrow i_*(T_Y)\longrightarrow U_{T_Y}[1]\stackrel{+}{\longrightarrow}$ in 
$\hat{\mathcal{D}}^c$ such that $U_{T_Y}\in\text{susp}_{\hat{\mathcal{D}}}(j_!(T_X))$ and $\tilde{T}_Y\in {}^\perp j_!(\mathcal{X}^{\leq 0})[1]$. 
In particular $T:=j_!(T_X)\oplus\tilde{T}_Y$ is a silting  object in $\hat{\mathcal{D}}^c$, uniquely determined up to add-equivalence, which generates the glued t-structure in $\hat{\mathcal{D}}^b_{fl}$. 
\end{corollary}
\begin{proof} 
The object $T_X$ generates $\hat{\mathcal{X}}$. By \cite[Theorem 4.3]{K}, there is a dg algebra $C$ and a triangulated equivalence $F:\mathcal{D}(C)\stackrel{\cong}{\longrightarrow}\hat{\mathcal{X}}$ which takes $C$ to $T_X$. In particular, there is an isomorphism $H^kC\cong\text{Hom}_{\mathcal{D}(C)}(C,C[k])\stackrel{\cong}{\longrightarrow}\text{Hom}_{\hat{\mathcal{X}}}(T_X,T_X[k])$, for each $k\in\mathbb{Z}$. Hence, $C$ is homologically non-positive and homologically finite length. In addition, $F$ restricts to an equivalence $\mathcal{D}^b_{fl}(C)\stackrel{\cong}{\longrightarrow}\hat{\mathcal{X}}^b_{fl}$, since $\mathcal{D}^b_{fl}(C)=\mathcal{D}(C)^b_{fl}$. Similarly, there is a homologically non-positive homologically finite length dg algebra $B$ and a triangulated equivalence $G:\mathcal{D}(B)\stackrel{\cong}{\longrightarrow}\hat{\mathcal{Y}}$ which takes $B$ to $T_Y$ and restricts to a triangulated equivalence $\mathcal{D}^b_{fl}(B)\stackrel{\cong}{\longrightarrow}\hat{\mathcal{Y}}^b_{fl}$.

Thus, we can assume, that the recollement (\ref{Rece}) is of the form $$
\begin{xymatrix}{\mathcal{D}^b_{fl}(B) \ar[r]^-{{i}_*}& \hat{\mathcal{D}}^b_{ fl} \ar@<3ex>[l]_-{{i}^!}\ar@<-3ex>[l]_-{{i}^*}\ar[r]^-{{j}^*} & \mathcal{D}^b_{fl}(C) \ar@<3ex>_-{{j}_*}[l]\ar@<-3ex>_-{{j}_!}[l]}
\end{xymatrix},$$ where $B$ and $C$ are homologically non-positive homologically finite length dg algebras. Moreover, by the 
proof of Corollary \ref{cor.gluing wrt bounded subcategories}, we know that we are in the situation of Setup \ref{setup}. Then, 
by Proposition \ref{prop.homologically non-positive locally fd} (see Example \ref{exs.intrinsic description of compacts}(2)) and Theorem \ref{thm.gluing at compact level} (condition 2) with 
$\mathcal{T}_X=\{T_X\}$ and $\mathcal{T}_Y=\{T_Y\}$, the result follows,  except for the uniqueness of $T$. This uniqueness is a 
consequence of  Propositions \ref{prop.restricted partial silting} and \ref{prop.compact-psilting-from coheart}.\end{proof}

\subsection{Gluing with respect to t-structures versus gluing with respect to co-t-structures over finite length algebras}
 If $A$ is a finite length $K$-algebra there is a triangulated equivalence 
$\mathcal{D}^b_{fl}(A)\cong\mathcal{D}^b(\text{mod-}A)$ and $\mathcal{D}^c(A)$ may be identified with $\mathcal{K}^b(\text{proj-}A)$, the 
homotopy category of finitely generated projective $A$-modules. The 
following result  is  a direct consequence of Proposition \ref{prop.homologically non-positive locally fd} (see Remark \ref{rem.Artin})  and 
Corollary \ref{cor.Liu-Vitoria-Yang-generalized}, except for its last sentence which follows from  Proposition \ref{prop.U-envelope}.

\begin{corollary} \label{cor.gluing Artin case}

Let $A$, $B$ and $C$ be finite length algebras, let 

\begin{equation}\label{recD}
\begin{xymatrix}{\Db{B}\ar[r]^{i_*}& \Db{A}\ar@<3ex>[l]_{i^!}\ar@<-3ex>[l]_{i^*}\ar[r]^{j^*} & \Db{C}\ar@<3ex>_{j_*}[l]\ar@<-3ex>_{j_!}[l]}
\end{xymatrix}
\end{equation} 
be a recollement, and let $T_C$ and $T_B$ be silting complexes in $\mathcal{K}^b(\text{proj-}C)$ and $\mathcal{K}^b(\text{proj-}B)$ 
which generate t-structures $(\mathcal{X}^{\leq 0},\mathcal{X}^{\geq 0})$ in $\Db{C}$ and $(\mathcal{Y}^{\leq 0},\mathcal{Y}^{\geq 0})$ in 
$\Db{B}$. There is a triangle $\tilde{T}_B\longrightarrow i_*(T_B) \stackrel{f}{\longrightarrow}  U_{T_B}[1] \stackrel{+}{\longrightarrow}$  in 
$\mathcal{K}^b(\text{proj-}A)$ such that $U_{T_B}\in\text{susp}_{\mathcal{D}(A)}(j_!(T_C))$ and 
$\tilde{T}_B\in {}^\perp j_!(\mathcal{X}^{\leq 0})[1]$. The object $T:=j_!(T_C)\oplus\tilde{T}_B$ is a silting complex in 
$\mathcal{K}^b(\text{proj-}A)$, uniquely determined up to $add$-equivalence, which generates the glued t-structure 
$(\mathcal{D}^{\leq 0},\mathcal{D}^{\geq 0})$ in $\Db{A}$. 

Moreover, the map $f$ can be taken to be a $\text{susp} (j_!T_C)[1]$-envelope, which can be calculated inductively. 
\end{corollary}

In this section we compare the gluing of silting objects via the co-t-structures \cite{LVY} and the one given by the last corollary. Conceptually,  t-structures and co-t-structures corresponding to silting objects are adjacent, hence by Lemma \ref{lem.adjacent gluing} gluing with respect to two recollements in a ladder of recollements should give the same result. We provide the details below.

\begin{proposition} \label{prop.gluing-via-(co-)t-structures}
Let $A,B,C$ be finite length algebras  and suppose that there are recollements \begin{equation}\label{recD}
\begin{xymatrix}{\Db{B}\ar[r]^{i_*}& \Db{A}\ar@<3ex>[l]_{i^!}\ar@<-3ex>[l]_{i^*}\ar[r]^{j^*} & \Db{C}\ar@<3ex>_{j_*}[l]\ar@<-3ex>_{j_!}[l]}
\end{xymatrix}
\end{equation}  and \begin{equation}\label{recK}
\begin{xymatrix}{\mathcal{K}^b(\text{proj-}{C})\ar[r]^{j_!}& \mathcal{K}^b(\text{proj-}{A})\ar@<3ex>[l]_{j^*}\ar@<-3ex>[l]_{j^\sharp}\ar[r]^{i^*} & \mathcal{K}^b(\text{proj-}{B})\ar@<3ex>_{i_*}[l]\ar@<-3ex>_{i^\sharp}[l]}
\end{xymatrix}
\end{equation} where the functors $j_!$, $j^*$, $i^*$ and $i_*$ of the second recollement are the restrictions of the corresponding functors in the first one.   Let $T_C$ and $T_B$ be silting objects in $\mathcal{K}^b(\text{proj-}C)$ and $\mathcal{K}^b(\text{proj-}B)$,  let $(\mathcal{X}^{\leq 0},\mathcal{X}^{\geq 0})$ and  $(\mathcal{Y}^{\leq 0},\mathcal{Y}^{\geq 0})$ be the associated t-structures in $\mathcal{D}^b(\text{mod-}C)$ and $\mathcal{D}^b(\text{mod-}B)$ and  $(\mathcal{X}_{\geq 0},\mathcal{X}_{\leq 0})$ and  $(\mathcal{Y}_{\geq 0},\mathcal{Y}_{\leq 0})$ the associated co-t-structures in $\mathcal{K}^b(\text{proj-}C)$ and $\mathcal{K}^b(\text{proj-}B)$, respectively. 

If $(\mathcal{D}^{\leq 0},\mathcal{D}^{\geq 0})$ is the glued t-structure in $\mathcal{D}^b(\text{mod-}A)$ with respect to the recollement (\ref{recD}) and $(\mathcal{D}_{\geq 0},\mathcal{D}_{\leq 0})$ is the glued co-t-structure in $\mathcal{K}^b(\text{proj-}A)$ with respect to the recollement (\ref{recK}), then there is a silting complex $T\in\mathcal{K}^b(\text{proj-}A)$, uniquely determined up to $add$-equivalence, such that $(\mathcal{D}^{\leq 0},\mathcal{D}^{\geq 0})$ is the t-structure in $\mathcal{D}^b(\text{mod-}A)$ associated to $T$ and $(\mathcal{D}_{\geq 0},\mathcal{D}_{\leq 0})$ is the co-t-structure in $\mathcal{K}^b(\text{proj-}A)$ associated to $T$. 
\end{proposition}
\begin{proof}
By Corollary \ref{cor.gluing Artin case}, there is a triangle  $\tilde{T}_B\longrightarrow i_*(T_B) \stackrel{f}{\longrightarrow}  U_{T_B}[1] \stackrel{+}{\longrightarrow}$ ($\star$) in 
$\mathcal{K}^b(\text{proj-}A)$ such that $U_{T_B}\in\text{susp}_{\mathcal{D}(A)}(j_!(T_C))$ and 
$\tilde{T}_B\in {}^\perp j_!(\mathcal{X}^{\leq 0})[1]$ and  $(\mathcal{D}^{\leq 0},\mathcal{D}^{\geq 0})$ is the t-structure in $\mathcal{D}^b(\text{mod}-A)$ associated to $T=j_!T_C\oplus\tilde{T}_B$.  Note that, by construction, this is a triangle associated to the co-t-structure $({}^\perp\text{susp}_{\mathcal{D}(A)}(j_!T_C))[1]\cap\mathcal{K}^b(\text{proj-}A),\text{susp}_{\mathcal{D}(A)}(j_!T_C))[1])$ in $\mathcal{K}^b(\text{proj-}A)$ (see the paragraph immediately before Theorem \ref{thm.gluing at compact level}).

On the other hand,  $(\mathcal{X}_{\geq 0},\mathcal{X}_{\leq 0})=(\text{cosusp}_{\mathcal{D}(C)}(T_C),\text{susp}_{\mathcal{D}(C)}(T_C))$ and  $(\mathcal{Y}_{\geq 0},\mathcal{Y}_{\leq 0})=(\text{cosusp}_{\mathcal{D}(B)}(T_B),\text{susp}_{\mathcal{D}(B)}(T_B))$ are the co-t-structures in $\mathcal{K}^b(\text{proj-}C)$ and  $\mathcal{K}^b(\text{proj-}B)$ cogenerated by $T_C$ and $T_B$, respectively. By the dual of Lemma \ref{lem.generators of glued t-structure} (see Remark  \ref{rem.results dualize}), the glued co-t-structure  $(\mathcal{D}_{\geq 0},\mathcal{D}_{\leq 0})$ with respect to recollement (\ref{recK}) is precisely the one cogenerated by $S=j_!T_C\oplus i_*T_B$. That is,   $(\mathcal{D}_{\geq 0},\mathcal{D}_{\leq 0})=({}^{\perp_{>0}}S\cap\mathcal{K}^b(\text{proj-}A), ({}^{\perp_{\geq 0}}S\cap\mathcal{K}^b(\text{proj-}A))^\perp\cap\mathcal{K}^b(\text{proj-}A))$. By the triangle ($\star$) above, we clearly have that ${}^{\perp_{>0}}S\subseteq {}^{\perp_{>0}}\tilde{T}_B\cap {}^{\perp_{>0}}j_!T={}^{\perp_{>0}}(\tilde{T}_B\oplus j_!T_C)={}^{\perp_{>0}}T$. The same triangle also shows that any object   $M\in {}^{\perp_{>0}}T$ is also in ${}^{\perp_{>0}}i_*T_B$, which in turn implies that $M\in  {}^{\perp_{>0}}(j_!T_C\oplus i_*T_B)= {}^{\perp_{>0}}S$. Therefore  $(\mathcal{D}_{\geq 0},\mathcal{D}_{\leq 0})$ is the co-t-structure in $\mathcal{K}^b(\text{proj-}A)$ cogenerated by $T$, i.e. $(\mathcal{D}_{\geq 0},\mathcal{D}_{\leq 0})=(\text{cosusp}_{\mathcal{D}(A)}(T),\text{susp}_{\mathcal{D}(A)}(T))$.

The uniqueness of $T$ up to $add$-equivalence follows from 
 Corollary \ref{cor.gluing Artin case} for   $(\mathcal{D}^{\leq 0},\mathcal{D}^{\geq 0})$ and  from Proposition \ref{prop.bijection psilting-cotstructures} for  $(\mathcal{D}_{\geq 0},\mathcal{D}_{\leq 0})$. Alternatively, it follows from the Koenig-Yang bijection for the case of a finite-dimensional algebra over a field (see \cite{KY}).
\end{proof}

A natural subsequent question is when the 'overlap' of recollements like in the last corollary exists. Our next result gives two situations for which it is the case:

\begin{proposition} \label{prop.recollement apt for bygluing}
Let $A,B,C$ be finite length $K$-algebras and let \begin{equation}\label{recDD}
\begin{xymatrix}{\Db{B}\ar[r]^{i_*}& \Db{A}\ar@<3ex>[l]_{i^!}\ar@<-3ex>[l]_{i^*}\ar[r]^{j^*} & \Db{C}\ar@<3ex>_{j_*}[l]\ar@<-3ex>_{j_!}[l]}
\end{xymatrix}
\end{equation}
be a recollement. The following assertions are equivalent:

\begin{enumerate}
\item There is a recollement \begin{equation}\label{recKK}
\begin{xymatrix}{\mathcal{K}^b(\text{proj-}{C})\ar[r]^{j_!}& \mathcal{K}^b(\text{proj-}{A})\ar@<3ex>[l]_{j^*}\ar@<-3ex>[l]_{j^\sharp}\ar[r]^{i^*} & \mathcal{K}^b(\text{proj-}{B})\ar@<3ex>_{i_*}[l]\ar@<-3ex>_{i^\sharp}[l]}
\end{xymatrix}
\end{equation} where the functors $j_!$, $j^*$, $i^*$ and $i_*$  are the restrictions of the corresponding functors in the recollement (\ref{recDD}) (see Proposition \ref{prop.homologically non-positive locally fd}).
\item The induced functor $j_!:\mathcal{K}^b(\text{proj-}C)\longrightarrow\mathcal{K}^b(\text{proj-}A)$ has a left adjoint.
\end{enumerate}

Under either one of the following two conditions, the assertions above hold:

\begin{enumerate}
\item[(a)]  $K$ is a field (and hence $A$, $B$ and $C$ are finite dimensional algebras);
\item[(b)] The algebra $C$ has finite global dimension.
\end{enumerate}
\end{proposition}
\begin{proof}
$(1)\Longrightarrow (2)$ is clear. 

$(2)\Longrightarrow (1)$ For each $D\in\mathcal{K}^b(\text{proj-}A)$, the triangle $j_!j^*D\longrightarrow D\longrightarrow i_*i^*D\stackrel{+}{\longrightarrow}$ given by the recollement (\ref{recDD}) belongs to $\mathcal{K}^b(\text{proj-}A)$ by Proposition \ref{prop.homologically non-positive locally fd}, implying that $(j_!(\mathcal{K}^b(\text{proj-}C)),i_*(\mathcal{K}^b(\text{proj-}B)))$ is a semi-orthogonal decomposition in $\mathcal{K}^b(\text{proj-}A)$. Let $j^\sharp:\mathcal{K}^b(\text{proj-}A)\longrightarrow\mathcal{K}^b(\text{proj-}C)$ be the left adjoint to $j_!$. Since $j_!$ is fully faithful,  the counit $\delta: j^\sharp\circ j_!\longrightarrow 1_{\mathcal{K}^b(\text{proj-}C)}$ is an isomorphism. Let us consider the triangle $K_D\longrightarrow D\stackrel{\eta_D}{\longrightarrow}j_!j^\sharp D\stackrel{+}{\longrightarrow}$, for any $D\in\mathcal{K}^b(\text{proj-}A)$, where $\eta_D$ is the unit of the adjunction. For each $X\in\mathcal{K}^b(\text{proj-}C)$, there is a composition of morphisms  $\text{Hom}_{\mathcal{K}^b(\text{proj-}C)}(j^\sharp D,X)\stackrel{j_!}{\longrightarrow}\text{Hom}_{\mathcal{K}^b(\text{proj-}A)}(j_!j^\sharp D,j_!X)\longrightarrow\text{Hom}_{\mathcal{K}^b(\text{proj-}A)}(D,j_!X)$, which is easily identified with the adjunction isomorphism. The first arrow of this composition is an isomorphism, since $j_!$ is fully faithful. Hence, $\eta_D^*=\text{Hom}_{\mathcal{K}^b(\text{proj-}A)}(\eta_D,j_!X)$ is an isomorphism, for all $X\in\mathcal{K}^b(\text{proj}-C)$, yielding $K_D\in {}^\perp j_!(\mathcal{K}^b(\text{proj-}C))$. So $({}^\perp j_!(\mathcal{K}^b(\text{proj-}C)),j_!(\mathcal{K}^b(\text{proj-}C)))$ is a semi-orthogonal decomposition of $\mathcal{K}^b(\text{proj-}A)$, and  $(\mathcal{S}_0,\mathcal{U}_0,\mathcal{V}_0):=({}^\perp j_!(\mathcal{K}^b(\text{proj-}C)),j_!(\mathcal{K}^b(\text{proj-}C)),i_*(\mathcal{K}^b(\text{proj-}B)))$ is a TTF triple in $\mathcal{K}^b(\text{proj-}A)$. One then obtains the recollement of assertion (1) by standard methods. In particular, the functor $i^\sharp :\mathcal{K}^b(\text{proj-}B)\longrightarrow\mathcal{K}^b(\text{proj-}A)$  is the composition $\mathcal{K}^b(\text{proj-}B)\stackrel{i_*}{\longrightarrow}\mathcal{V}_0\stackrel{\cong}{\longrightarrow}\mathcal{S}_0\stackrel{incl}{\hookrightarrow}\mathcal{K}^b(\text{proj-}A)$, where the central arrow is the canonical equivalence induced by the TTF triple.

We next prove that, under either one of conditions (a) or (b), assertion (1) holds.  By Theorem \ref{thm.Case Y homologically loc.bounded} and Proposition \ref{prop.restriction}, there is a recollement \begin{equation}\label{Recg}
\begin{xymatrix}{\mathcal{D}(B) \ar[r]^{\hat{i}_*}& \mathcal{D}(A) \ar@<3ex>[l]_{\hat{i}^!}\ar@<-3ex>[l]_{\hat{i}^*}\ar[r]^{\hat{j}^*} & \mathcal{D}(C) \ar@<3ex>_{\hat{j}_*}[l]\ar@<-3ex>_{\hat{j}_!}[l]}.
\end{xymatrix}
\end{equation} that restricts  to the $\mathcal{D}^b(\text{mod})$-level and whose restriction is equivalent to the recollement \ref{recDD}.  When $K$ is a field, using \cite[Proposition 3.2(b)]{AKLY}, we get that $\hat{j}_!$ has a left adjoint $\hat{j}^\sharp$. The functor $\hat{j}^\sharp$ preserves compact objects since its right adjoint preserves coproducts. The induced functors $(\hat{j}^\sharp:\mathcal{K}^b(\text{proj-}A)\longrightarrow\mathcal{K}^b(\text{proj-}C),\hat{j}_!:\mathcal{K}^b(\text{proj-}C)\longrightarrow\mathcal{K}^b(\text{proj-}A))$ form an adjoint pair. Bearing in mind that  $j_!(\mathcal{K}^b(\text{proj-}C))=\mathcal{K}^b(\text{proj-}A)\cap\text{Im}(j_!)=\mathcal{K}^b(\text{proj-}A)\cap\text{Im}(\hat{j}_!)=\hat{j}_!(\mathcal{K}^b(\text{proj-}C))$, we get a triangulated autoequivalence $\varphi :\mathcal{K}^b(\text{proj-}C)\stackrel{\cong}{\longrightarrow}\mathcal{K}^b(\text{proj-}C)$ such that $(\hat{j}_!)|_{ \mathcal{K}^b(\text{proj-}C)}\circ\varphi \cong (j_!)|_{ \mathcal{K}^b(\text{proj-}C)}$. Therefore assertion 2 holds.

When $C$ has finite global dimension, we consider, for any $M\in\mathcal{K}^b(\text{proj-}A)$,  the homological functor $H:=\text{Hom}_{\mathcal{K}^b(\text{proj-}A)}(M,j_!(-))$. Note that, when $R$ is a finite length $K$-algebra,  all $\text{Hom}$ spaces in $\mathcal{K}^b(\text{proj-}R)$ are modules over $K/I_R$, where $I_R=\{\lambda\in K\text{: }\lambda R=0\}$ is the annihilator of $R$ in $K$. Since $K/I_R$ is a $K$-submodule of $R$,  it is of finite length. Putting $I:=I_A\cap I_B\cap I_C$ in our case, one gets a commutative ring $K/I$ that has finite length as a $K$-module, and hence is an Artinian (whence Noetherian) ring. Moreover it is clear that all the functors in the recollement are $K/I$-linear.  Therefore, replacing $K$ by $K/I$  if necessary,we can assume that $K$ is an Artinian (whence Noetherian) commutative ring. Then the dual of \cite[Corollary 4.18]{Rou}, obtained first in \cite[Theorem 1.3]{BVdB} when $K$ is a field,  applies to our case. Indeed, in the terminology of \cite{Rou},  $\mathcal{K}^b(\text{proj-}C)^{op}$ is $Ext$-finite and strongly finitely generated (see \cite[Proposition 7.25]{Rou}) and the cohomological  functor $H:(\mathcal{K}^b(\text{proj-}C)^{op})^{op}\longrightarrow\text{Mod}-K$ is locally finite. Then $H$ is representable, so that we get an object $X_M\in\mathcal{K}^b(\text{proj-}C)^{op}$ such that $H\cong\text{Hom}_{\mathcal{K}^b(\text{proj-}C)^{op}}(-,X_M)\cong\text{Hom}_{\mathcal{K}^b(\text{proj-}C)}(X_M,-)$. It is routine to check that the assignment $M\rightsquigarrow X_M$ is the definition on objects of a functors $j^\sharp :\mathcal{K}^b(\text{proj-}A)\longrightarrow\mathcal{K}^b(\text{proj-}C)$ which is left adjoint to  $j_! :\mathcal{K}^b(\text{proj-}C)\longrightarrow\mathcal{K}^b(\text{proj-}A)$. Therefore assertion 2 holds.
\end{proof}

We do not know of any recollement of derived categories of finite length algebras where the equivalent assertions 1 and 2 of last proposition do not hold.

We would like to finish this section with examples of gluing computed applying the algorithm from Corollary \ref{cor.gluing Artin case} and Proposition \ref{prop.U-envelope}. Let $A$ be a finite dimensional algebra over a field. Recall the following example of a recollement of the unbounded derived category $\mathcal{D}(A)$. Let $e\in A$ be an idempotent element  such that $eA(1-e)=0$.  
It induces a homological ring epimorphism $i:A\longrightarrow A/AeA$ (see \cite[Example 4.1]{NS1}). Moreover, since $eA=eAe$, the left $eAe$-module $eA$ is projective; $A/AeA\cong A(1-e)$ is  projective as a left $A$-module as well. 
We then get a recollement \begin{equation}\label{lastRec1}
\begin{xymatrix}{\D{(A/AeA)} \ar[r]^-{i_*}& \D{(A)} \ar@<3ex>[l]_-{i^!}\ar@<-3ex>[l]_-{i^*}\ar[r]^-{j^*} & \D{(eAe)} \ar@<3ex>_-{j_*}[l]\ar@<-3ex>_-{j_!}[l]},
\end{xymatrix}
\end{equation}  
 where  $i_*=-\otimes_{A/AeA}^\mathbb{L}(A/AeA)=-\otimes_{A/AeA}(A/AeA)$, $i^*=-\otimes_A^\mathbb{L}(A/AeA)=-\otimes_A (A/AeA)$, $i^!=\mathbb{R}\text{Hom}_A(A/AeA,-))$, $j^*=-\otimes_{A}^\mathbb{L}(Ae)=-\otimes_AAe$, $j_!=-\otimes_{eAe}^\mathbb{L}eA=-\otimes_{eAe}eA$  and $j_*=\mathbb{R}\text{Hom}_{eAe}(Ae,-)$. It is clear that the functors $i^*, i_*,j_!,j^*$ restrict to the $\mathcal{D}^b(\text{mod})$-level. So by \cite[Theorem 4.6]{AKLY} the recollement restricts to the $\mathcal{D}^b(\text{mod})$-level if and only if $i_*(A/AeA)\in {K}^b(\text{proj-}A)$. That is, if and only if $A/AeA$ (equivalently $AeA$)  has finite projective dimension as a right $A$-module.  So if the idempotent $e$ satisfies the following two conditions
 \begin{enumerate}
\item[(a)] $eA(1-e)=0$;
\item[(b)] the projective dimension of $AeA$ as a right $A$-module is finite;
\end{enumerate}
we get an induced recollement
 \begin{equation}\label{lastRec2}
\begin{xymatrix}{\Db{A/AeA} \ar[r]^-{i_*}& \Db{A} \ar@<3ex>[l]_-{i^!}\ar@<-3ex>[l]_-{i^*}\ar[r]^-{j^*} & \Db{eAe} \ar@<3ex>_-{j_*}[l]\ar@<-3ex>_-{j_!}[l]}.
\end{xymatrix}
\end{equation}

\begin{example} \label{exs.inductive example}
 Let $A=KA_3=K(1\xrightarrow{a}2\xrightarrow{b}3)$ be the path algebra of the linear orientation of $A_3$. We will denote by $P_i,I_i,S_i$ the indecomposable projective, injective and simple modules corresponding to the vertex $i$. Let $e=e_3$ be the idempotent corresponding to the vertex $3$. It is clear that condition (a) and (b) are satisfied.
Then $A/AeA\simeq K(1\xrightarrow{a}2)$ and $eAe\simeq K$. Let us consider the following silting complexes for these algebras.  For $A/AeA$ we consider  $T_Y:=\bar{P}_1[1]\oplus\bar{P}_2$, where $\bar{P}_i$ is the indecomposable projective $A/AeA$-module corresponding the the vertex $i$. For $eAe$ we take $T_X:=eAe$. Then in Recollement \ref{lastRec2} we have: $j_!eAe=eAe\otimes_{eAe}eA\simeq e_3A=P_3$, $i_*T_Y=i_*(\bar{P}_1[1]\oplus\bar{P}_2)=(\bar{P}_1[1]\oplus\bar{P}_2)\otimes_{A/AeA}(A/AeA)\simeq I_2[1]\oplus S_2$. Note that $I_2\simeq (P_3\xrightarrow{ab}P_1)$ and  $S_2\simeq (P_3\xrightarrow{b}P_2)$ in $\Db{A}$. We need to construct $susp(P_3)[1]$-envelope of $I_2[1]\oplus S_2$. The maximal $s$ such that  $Hom(I_2[1]\oplus S_2, P_3[1][s])\neq 0$ is $1$, the $add(P_3[2])$-envelope of $I_2[1]\oplus S_2$ is given by the chain map $h$ depicted by the following diagram:
$$
\begin{xymatrix}{P_3 \ar[d]^{id} \ar[r]^-{(0,ab)^t} & P_3\oplus P_1 \ar[r]^-{(b,0)}& P_2\\
P_3&&}.
\end{xymatrix}
$$
The cocone of this  map is the complex $P_3\oplus P_1\xrightarrow{(b,0)}P_2$, which is isomorphic to  $P_1[1]\oplus S_2$ in $\Db{A}$. We denote by $u$ the  map  representing the morphism $P_1[1]\oplus S_2\longrightarrow I_2[1]\oplus S_2$ from the cocone.  The $add(P_3[1])$-envelope of $P_1[1]\oplus S_2$ is the chain map $g$ given by the following diagram: 
$$
\begin{xymatrix}{P_3\oplus P_1 \ar[d]^{(id,0)} \ar[r]^-{(b,0)} & P_2\\
P_3&},
\end{xymatrix}
$$
the cocone of which is the complex $P_1\stackrel{0}{\longrightarrow}P_2$ isomorphic to   $P_1[1]\oplus P_2$. Denote by $v$ the  map  representing the morphism $P_1[1]\oplus P_2\longrightarrow P_1[1]\oplus S_2$ from the cocone. 

By Lemma \ref{lem.inductive step construction envelopes} the $susp(P_3)[1]$-envelope $f$ of $I_2[1]\oplus S_2$ is the map from $I_2[1]\oplus S_2$ to the cone of $u\circ v$. It easily follows that the desired triangle  
$\tilde{T}_Y\longrightarrow i_*T_Y\stackrel{f}{\longrightarrow} U[1]\stackrel{+}{\longrightarrow}$ is, up to isomorphism in $\Db{A}$,  of the form $P_1[1]\oplus P_2\longrightarrow I_2[1]\oplus S_2 \longrightarrow P_3[2]\oplus P_3[1]\stackrel{+}{\longrightarrow}$. The glued silting complex is then  $j_!T_X\oplus\tilde{T}_Y\cong  P_1[1]\oplus P_2\oplus P_3$.
\end{example}

\begin{example}
Let $A$ be any finite dimensional algebra and $e\in A$ be  an idempotent satisfying conditions (a) and (b) above. Let $(\mathcal{X}^{\leq 0},\mathcal{X}^{\geq 0})$ be the canonical t-structure in $\Db{eAe}$, induced by the tilting object $eAe$, and let  $(\mathcal{Y}^{\leq 0},\mathcal{Y}^{\geq 0})$ be a partial silting t-structure in $\Db{A/AeA}$ induced by a compact partial silting object $T_Y$ in $\Db{A/AeA}$ such that $T_Y\in\mathcal{D}^{\leq 0}(A/AeA)$. Fix a quasi-isomorphism $s:P^\bullet\longrightarrow i_*(T_Y)$, where $P^\bullet\in\mathcal{K}^{\leq 0}(\text{proj-}A)$ is assumed to be minimal, i.e. such that the image of the  differential $d^k:P^k\longrightarrow P^{k+1}$ is contained in $P^{k+1}\text{rad}(A)$, for each $i\in\mathbb{Z}$. The exact sequence of complexes $0\rightarrow P^\bullet (1-e)A\longrightarrow P^\bullet\longrightarrow P^\bullet/P^\bullet (1-e)A\rightarrow 0$ splits in each degree and induces  a  triangle $\tilde{T}_Y\longrightarrow i_*T_Y\stackrel{h}{\longrightarrow} U[1]\stackrel{+}{\longrightarrow}$ in $\Db{A}$, where $h$ is a $j_!(\mathcal{X}^{\leq 0})[1]$-envelope. In particular $T=eA\oplus P^\bullet (1-e)\cong j_!(eAe)\oplus\tilde{T}_Y$ is a partial silting object of $\mathcal{D}^b(\text{mod}-A)$ which generates the glued t-structure by Corollary \ref{cor.gluing wrt bounded subcategories}.

\end{example}
\begin{proof}
Each projective right $A$-module $P$ decomposes as $_eP\oplus _{1-e}P$, with $_eP\in\text{add}(eA)$ and $_{1-e}P\in\text{add}((1-e)A)$. Since $eA(1-e)=0$, $_eP(1-e)A=0$ and $P(1-e)A= _{1-e}P$. It follows that $ _{1-e}P^\bullet\cong P^\bullet (1-e)A$ is a subcomplex of $P^\bullet$ and the three terms in the exact sequence $0\rightarrow P^\bullet (1-e)A\longrightarrow P^\bullet\longrightarrow P^\bullet/P^\bullet (1-e)A\rightarrow 0$ are complexes of projectives.
The component $P^0$ of the complex $P^\bullet$ is a projective cover of $H^0(i_*T_Y)$. Since $H^0(i_*T_Y)$ is a quotient of a module in $\text{add}(A/AeA)$, we get $P^0\in\text{add}((1-e)A)$. Thus, $P^\bullet/P^\bullet (1-e)A\cong {}_eP^\bullet$ belongs to $\mathcal{K}^{<0}(\text{add}(eA))$.

Note that $ _ePe\otimes_{eAe}eA \cong {}_eP$. Moreover, $j_!(_eP^{\bullet}e) \cong {}_eP^{\bullet} e\otimes_{eAe}^\mathbb{L} eA \cong {}_eP^\bullet e\otimes_{eAe} eA \cong {}_e P^\bullet\in j_! (D^{< 0}(\text{mod-}eAe)\cap D^b(\text{mod-}eAe))=j_!(\mathcal{X}^{\leq 0})[1]$. On the other hand,  $\text{Im}(j_!)=\mathcal{K}^{-}(\text{add}(eA))\cap D^b(\text{mod-}A)$, when we view $\mathcal{K}^{-}(\text{add}(eA))$ as a full subcategory of $D(A)$. Hence $\text{Hom}_{D(A)}(P^\bullet (1-e)A,-)$ vanishes on $\text{Im}(j_!)$, and  so $\pi:P^\bullet\longrightarrow P^\bullet/P^\bullet (1-e)A$ is a $j_!(\mathcal{X}^{\leq 0})[1]$-preenvelope. It is easy to see that $\pi$ is a left minimal morphism in $D^b(\text{mod-}A)$,  and therefore a $j_!(\mathcal{X}^{\leq 0})[1]$-envelope of $P^\bullet\cong i_*T_Y$. 
\end{proof}

\bibliographystyle{plain}
%\bibliography{bibfile2}

\end{document}